\documentclass[a4paper,12pt]{article}

\usepackage{amsmath}
\usepackage{multirow,bigdelim}
\usepackage{amscd}
\usepackage{amssymb,amsfonts,amsmath,amsthm}
\usepackage{verbatim}

\usepackage{amsmath}
\usepackage{amsthm}
\usepackage{amssymb}

\usepackage{latexsym}
\usepackage{array}
\usepackage{enumerate}

\numberwithin{equation}{section}

\usepackage{amsmath,amssymb,mathrsfs,amsthm}
\usepackage[all]{xy}
\usepackage{graphicx}
\usepackage{ulem}
\usepackage{ascmac}
\usepackage{mathtools}

\newcommand{\rHom}{\mathrm{RHom}}
\newcommand{\BDC}{{\mathbf{D}}^{\mathrm{b}}}

\newcommand{\Mod}{\mathrm{Mod}}

\newcommand{\shom}{{\mathcal{H}}om}

\newcommand{\CC}{\mathbb{C}}

\newcommand{\RR}{\mathbb{R}}
\newcommand{\QQ}{\mathbb{Q}}
\newcommand{\ZZ}{\mathbb{Z}}
\newcommand{\D}{\mathcal{D}}
\newcommand{\E}{\mathcal{E}}
\newcommand{\F}{\mathcal{F}}
\newcommand{\G}{\mathcal{G}}
\newcommand{\M}{\mathcal{M}}
\newcommand{\N}{\mathcal{N}}

\newcommand{\PP}{{\mathbb P}}

\newcommand{\LL}{{\mathbb L}}

\newcommand{\an}{{\rm an}}

\renewcommand{\dim}{{\rm dim}}

\newcommand{\e}{\varepsilon}

\newcommand{\supp}{{\rm supp}}

\newcommand{\Int}{{\rm Int}}

\newcommand{\tl}[1]{\widetilde{#1}}
\newcommand{\uotimes}[1]{\otimes_{#1}}

\newcommand{\simto}{\overset{\sim}{\longrightarrow}}

\newcommand{\op}{\mbox{\scriptsize op}}
\newcommand{\SD}{\mathcal{D}}
\newcommand{\SDop}{\mathcal{D}^{\mbox{\scriptsize op}}}
\newcommand{\SO}{\mathcal{O}}
\newcommand{\SA}{\mathcal{A}}
\newcommand{\SM}{\mathcal{M}}
\newcommand{\SN}{\mathcal{N}}
\newcommand{\SL}{\mathcal{L}}
\newcommand{\SK}{\mathcal{K}}
\newcommand{\SE}{\mathcal{E}}
\newcommand{\SF}{\mathcal{F}}
\newcommand{\SG}{\mathcal{G}}

\newcommand{\Modcoh}{\mathrm{Mod}_{\mbox{\scriptsize coh}}}
\newcommand{\Modhol}{\mathrm{Mod}_{\mbox{\scriptsize hol}}}
\newcommand{\Modrh}{\mathrm{Mod}_{\mbox{\scriptsize rh}}}

\newcommand{\BDCcoh}{{\mathbf{D}}^{\mathrm{b}}_{\mbox{\scriptsize coh}}}

\newcommand{\BDChol}{{\mathbf{D}}^{\mathrm{b}}_{\mbox{\scriptsize hol}}}
\newcommand{\BDCrh}{{\mathbf{D}}^{\mathrm{b}}_{\mbox{\scriptsize rh}}}
\newcommand{\DD}{\mathbb{D}}
\newcommand{\Lotimes}[1]{\overset{L}{\otimes}_{#1}}
\newcommand{\Dotimes}{\overset{D}{\otimes}}

\newcommand{\Potimes}{\overset{+}{\otimes}}
\newcommand{\Pboxtimes}{\overset{+}{\boxtimes}}
\newcommand{\rhom}{{\rm R}{\mathcal{H}}om}
\newcommand{\rihom}{{\rm R}{\mathcal{I}}hom}
\newcommand{\Prihom}{{\rm R}{\mathcal{I}}hom^+}
\newcommand{\Prhom}{\rhom^+}
\newcommand{\I}{{\rm I}}
\newcommand{\che}[1]{\overset{\vee}{#1}}
\newcommand{\var}[1]{\overline{#1}}
\newcommand{\BEC}{{\mathbf{E}}^{\mathrm{b}}}
\newcommand{\Q}{\mathbf{Q}}
\newcommand{\EE}{\mathbb{E}}
 
\newcommand{\bs}{\backslash}
\newcommand{\ch}{\rm char}
\newcommand{\inj}{``\varinjlim"}
\newcommand{\ms}[1]{{\rm SS}({#1})}
\newcommand{\T}{{\rm T}}
\newcommand{\bfR}{\mathbf{R}}
\newcommand{\bfL}{\mathbf{L}}
\newcommand{\bfD}{\mathbf{D}}
\newcommand{\rmR}{{\rm R}}
\newcommand{\rmE}{{\rm E}}
\newcommand{\rmD}{{\rm D}}
\newcommand{\rmt}{{\rm t}}
\newcommand{\bfE}{\mathbf{E}}
\newcommand{\rmL}{{\rm L}}
\renewcommand{\Re}{\rm Re}
\renewcommand{\ch}{{\rm char}}
\newcommand{\mi}{{n_i}}

\newtheorem{theorem}{Theorem}[section]

\newtheorem{corollary}[theorem]{Corollary}
\newtheorem{lemma}[theorem]{Lemma}
\newtheorem{proposition}[theorem]{Proposition}

\theoremstyle{definition}
\newtheorem{definition}[theorem]{Definition}
\theoremstyle{remark}
\newtheorem{remark}[theorem]{\sc Remark}
\newtheorem{example}[theorem]{\sc Example}

\title{On Irregularities of Fourier Transforms\\ of 
Regular Holonomic $\SD$-Modules 
\footnote{{\bf 2010 Mathematics 
Subject Classification: }32C38, 32S60, 34M35, 35A27}}

\author{ Yohei ITO
\footnote{Graduate School of Mathematical Science, The University of 
Tokyo, 3-8-1, Komaba, 
Meguro, Tokyo, 153-8914, Japan. 
E-mail: yitoh@ms.u-tokyo.ac.jp } 
and Kiyoshi TAKEUCHI 
\footnote{Institute of Mathematics, University  of 
Tsukuba, 1-1-1, Tennodai, 
Tsukuba, Ibaraki, 305-8571, Japan. 
E-mail: takemicro@nifty.com} }

\date{}

\begin{document}
\maketitle

\begin{abstract}
We study Fourier transforms of regular holonomic 
$\SD$-modules. By using the 
theory of Fourier-Sato transforms of enhanced 
ind-sheaves developed by Kashiwara-Schapira and 
D'Agnolo-Kashiwara, a formula for 
their enhanced solution complexes will be 
obtained. Moreover we show that some 
parts of their characteristic cycles and 
irregularities are expressed by 
the geometries of the original $\SD$-modules. 
\end{abstract}

\section{Introduction}\label{sec:1}
The theory of Fourier transforms of 
$\SD$-modules is a beautiful subject in algebraic analysis. 
They interchange algebraic $\SD$-modules on complex 
vector spaces $\CC^N$ with those on their duals. 
Especially, the case $N=1$ has been studied precisely by many 
mathematicians such as Bloch-Esnault \cite{BE04}, 
Malgrange \cite{Mal91}, Mochizuki \cite{Mochi10}, Sabbah \cite{Sab08} etc. 
On the other hand, after a groundbreaking 
development in the theory of irregular meromorphic 
connections by Kedlaya \cite{Ked10, Ked11} and Mochizuki 
\cite{Mochi09, Mochi11}, in \cite{DK16} D'Agnolo and Kashiwara 
established the Riemann-Hilbert correspondence for 
irregular holonomic $\SD$-modules. For this purpose, 
they introduced enhanced ind-sheaves extending the 
classical notion of ind-sheaves introduced by 
Kashiwara-Schapira \cite{KS01}. 
Subsequently, in \cite{KS16-2} 
Kashiwara and Schapira adapted this new notion to 
the Fourier-Sato transforms of Tamarkin \cite{Tama08} 
and developed a new theory of Fourier-Sato transforms 
for enhanced ind-sheaves which correspond to those 
for algebraic holonomic $\SD$-modules. 
Recently, in \cite{DK17} by making use of 
these results effectively, D'Agnolo and Kashiwara 
succeeded in studying Fourier transforms of 
holonomic $\SD$-modules on $\CC$ very precisely. 
Note that in this particular case $N=1$ 
Fourier transforms of regular holonomic 
$\SD$-modules were studied successfully 
in the previous paper D'Agnolo-Hien-Morando-Sabbah 
\cite{DHMS17} by a different method. In conclusion, 
thanks to the new theory of enhanced ind-sheaves 
of \cite{DK16}, we now understand the 
Fourier transforms of $\SD$-modules on $\CC$ 
more clearly than before. However, 
we know only little in the 
higher-dimensional case $N \geq 2$. 
The following beautiful theorem 
is due to Brylinski 
\cite{Bry86}. Set $X= \CC^N$ and 
recall that a constructible sheaf 
$\SF  \in\BDC_{\CC-c}(\CC_{X})$ on it 
is called monodromic if 
its cohomology sheaves are locally constant 
on each $\CC^*$-orbit in $\CC^N$. 
Let $Y \simeq \CC^N$ be the dual 
of $X= \CC^N$. 

\begin{theorem}[{\rm (Brylinski \cite{Bry86})}]\label{th-Bry}
Let $\SM$ be an algebraic regular holonomic 
$\SD$-module on $X= \CC^N$. Assume that its 
solution complex $Sol_{X}(\SM)$ is 
monodromic. Then its 
Fourier transform $\SM^\wedge$ is regular 
and $Sol_{Y}(\SM^\wedge)$ is monodromic. 
\end{theorem}

To the best of our knowledge, except for this theorem 
not much is known for Fourier transforms of 
$\SD$-modules on $\CC^N$ for $N \geq 2$. 
Recall also that some topological results 
related to them were obtained in 
Kashiwara-Schapira \cite{KS97, KS16-2} and D'Agnolo \cite{D'Agnolo}.
Thus the higher-dimensional case $N \geq 2$ still remains very mysterious. 

In this paper, we 
clarify this situation in the light of 
the more sophisticated theories of 
Kashiwara-Schapira \cite{KS01} and 
D'Agnolo-Kashiwara \cite{DK16}. 
Especially we study 
the Fourier transforms of regular holonomic 
$\SD$-modules on $X= \CC^N$ for $N \geq 2$. 
For this purpose, we make use of the 
theory of Fourier-Sato transforms of enhanced 
ind-sheaves developed by Kashiwara-Schapira \cite{KS16-2} 
and D'Agnolo-Kashiwara \cite{DK17}. In particular, 
we obtain the following result. 
For an algebraic regular holonomic $\SD$-module $\SM$  
on $X= \CC^N_z$ 
denote by $\ch(\SM)\subset T^\ast X\simeq X\times Y$
its characteristic variety. Let 
$p : X\times Y\to X$ and 
$q : X\times Y\to Y$ 
be the projections. Then we define 
a (Zariski) open subset $\Omega \subset Y=\CC_w^N$ by:
\begin{equation*}
w\in\Omega \quad \Longleftrightarrow \quad 
\begin{cases}
\mbox{
there exists an open neighborhood $U$ of $w$ in Y}\\
\mbox{such that the restriction $q^{-1}(U)\cap\ch(\SM)\to U$}\\
\mbox{of $q : X\times Y\to Y$ is an unramified covering.
}
\end{cases}
\end{equation*}

\noindent Since $\ch(\SM)$ is $\CC^\ast$-conic,
$\Omega\subset Y=\CC_w^N$ is also $\CC^\ast$-conic.
Denote by $k\geq0$ the degree of the covering
$q^{-1}(\Omega)\cap\ch(\SM)\to\Omega$.
For a point $w\in\Omega\subset Y=\CC^N$,
let $\{\mu_1(w), \ldots, \mu_k(w)\} =
q^{-1}(w)\cap\ch(\SM)\subset T^\ast X$ be its
fiber by $q^{-1}(\Omega)\cap\ch(\SM)\to \Omega$. 
For $1\leq i\leq k$ set
\[\alpha_i(w) := p(\mu_i(w))\in X=\CC^N.\]
and denote by 
$m_i>0$ the multiplicity of $\SM$ at 
$\mu_i(w)\in\ch(\SM)$. 
Let $i_Y : 
Y=\CC^N\xhookrightarrow{\ \ \ }\var{Y}=\PP^N$ 
be the projective compactification of $Y$. 
We extend the Fourier transform $\SM^{\wedge} \in\Modhol(\SD_Y)$ 
to a holonomic $\SD$-module 
$\tl{\SM^{\wedge}} := i_{Y \ast} ( \SM^{\wedge} ) 
\simeq \bfD i_{Y\ast} ( \SM^{\wedge} )$ 
on $\var{Y}$. 
Let $\var{Y}^{\an}$ be the underlying complex manifold 
of $\var{Y}$ and 
define the analytification 
$\tl{\SM^{\wedge}}^{\an}\in\Modhol(\SD_{\var{Y}^{\an}})$ 
of $\tl{\SM^{\wedge}}$ by 
$\tl{\SM^{\wedge}}^{\an} = \SO_{\var{Y}^{\an}}
\otimes_{\SO_{\var{Y}}}\tl{\SM^{\wedge}}$. 
Then we have the following 
formula for the enhanced solution complex 
\[Sol_{\var Y}^\rmE(\tl{\SM^{\wedge}}) := 
Sol_{\var{Y}^{\an}}^\rmE( \tl{\SM^{\wedge}}^{\an} )
\in \BEC(\I\CC_{\var{Y}^{\an}})\]
of $\tl{\SM^{\wedge}}^{\an}$. 

\begin{theorem}\label{th-A} 
Let $U\subset\Omega\subset Y=\CC^N$
be a connected and simply connected open subset of $\Omega$.
Then we have an isomorphism
\[\pi^{-1}\CC_U\otimes
\Big(Sol_{\var Y}^\rmE( \tl{\SM^\wedge} )\Big)
\simeq
\bigoplus_{i=1}^k
\pi^{-1}\CC_U\otimes
\Big(\underset{a\to+\infty}{\inj}\ 
\CC_{\{t\geq {\Re} \langle \alpha_i(w), 
w \rangle +a\}}^{\oplus m_i}\Big)\]
in $\BEC(\I\CC_{\var{Y}^{\an}})$, where 
$\langle \cdot ,  \cdot \rangle : X\times Y\to\CC$ 
is the canonical
paring. In particular, the restriction 
$\SM^\wedge |_{\Omega}$ of the Fourier 
transform $\SM^\wedge$ to $\Omega$ is an 
algebraic integrable connection of rank 
$\sum_{i=1}^k m_i$. 
\end{theorem}

\noindent Namely the regularity of $\SM$ implies the 
$\CC^*$-conicness of the smooth locus $\Omega$ of 
the Fourier transform $\SM^\wedge$
(in fact, we can also show that $Sol_Y(\M^\wedge$) is monodromic
as we see in \cite{IT18}). 
Note that a different expression 
for the rank of  $\SM^\wedge$ 
at generic points of $Y= \CC^N$ 
was given also 
by Brylinski in \cite[Corollaire 8.6]{Bry86}. 
Our proof of Theorem \ref{th-A} is based on 
the arguments in that of 
Esterov-Takeuchi \cite[Theorem 5.5]{ET15}
(for their applications see \cite{AET15}) 
and relies on some careful observation of 
the geometric situation at infinity of 
the perverse sheaf 
$\SF := Sol_X(\SM)[N]\in\BDC_{\CC-c}(\CC_{X^\an})$ 
(see the proof of Theorem \ref{thm-2} 
and Lemma \ref{newlea}). 
Recall that in \cite{Ado94} Adolphson introduced 
confluent (i.e. irregular) 
$A$-hypergeometric systems on $\CC^N$ 
by extending the classical notion of 
non-confluent (i.e. regular) 
ones of Gelfand-Kapranov-Zelevinsky 
\cite{GKZ89}, \cite{GKZ90}. He showed also that 
they have (non-empty) 
$\CC^*$-conic smooth loci in 
$\CC^N$. Recently, in 
Saito \cite{Saito11}, Schultz-Walther \cite{SW09}, 
\cite{SW12} and Esterov-Takeuchi \cite{ET15}, 
the authors found that Adolphson's 
confluent $A$-hypergeometric systems 
are Fourier transforms of some special regular 
holonomic $\SD$-modules. This motivated us 
to formulate Theorem \ref{th-A} for 
Fourier transforms of general regular 
holonomic $\SD$-modules. In the proof of 
Theorem \ref{th-A}, we first rewrite 
the Fourier-Sato transforms of enhanced ind-sheaves 
as in D'Agnolo-Kashiwara \cite{DK17} to obtain 
the geometric situation 
similar to the one in the proof of 
Esterov-Takeuchi \cite[Theorem 5.5]{ET15}. 
Then we apply the Morse theoretical 
argument in it to the solution complex 
of $\SM$. By Theorem \ref{th-A}, at generic 
points $v \in Y \setminus \Omega$ where 
$D:=Y \setminus \Omega$ is a smooth hypersurface 
we obtain also the irregularity and the exponential 
factors of $\SM^\wedge$ along it. 
More precisely, let $D_{\rm reg} \subset D$ be 
the smooth part of $D$ and $v \in D_{\rm reg}$ 
such a generic point. 
Take a subvariety $M\subset Y$ of $Y= \CC^N$ which
intersects $D_{\rm reg}$ at $v$ transversally. 
We call it a normal slice of $D$ at $v$. 
By definition $M$ is smooth and of dimension $1$ 
on a neighborhood of $v$. 
Let $i_M : M\xhookrightarrow{\ \ \ } Y=\CC^N$ be 
the inclusion map and set 
$\SK = {\bfD}i_M^\ast\SM^\wedge\in\Modhol(\SD_M)$. 
Then we can describe the irregularity 
${\rm irr} ( \SK(\ast\{v\}) )$ of the 
meromorphic connection 
$\SK(\ast\{v\})$ on $M$ along $\{ v \} \subset M$ 
as follows. Shrinking the normal slice $M$ 
if necessary we may assume that 
$M= \{ u \in \CC \ | \ | u | < 
\varepsilon \}$ for some $\varepsilon >0$, 
$\{ v \} = \{ u =0 \}$ and 
$M \setminus \{ v \} \subset \Omega$. 
Let $i_0 : 
M \setminus \{v\} \hookrightarrow \Omega$ 
be the inclusion map and define 
(possibly multi-valued) holomorphic 
functions $\varphi_i: M \setminus \{ v \} 
\rightarrow \CC$ ($1 \leq i \leq k$) by 
\[ \varphi_i ( u )= \langle 
\alpha_i ( i_0(u) ), i_0(u) \rangle. \]
Then it is easy to see that $\varphi_i ( u )$ 
are Laurent Puiseux series of $u$ 
(see Kirwan \cite[Section 7.2]{Kir92} etc.). 
For each Laurent Puiseux series 
\[ \varphi_i ( u )= 
\sum_{a \in \QQ} c_{i,a} u^a \qquad (c_{i,a} \in \CC ) \]
set $r_i= \min \{ a \in \QQ \ | \ c_{i,a} \not= 0 \}$ and 
define its pole order 
${\rm ord}_{\{ v \}}( \varphi_i ) \geq 0$ by 
\[
{\rm ord}_{\{ v \}}( \varphi_i )
=
\begin{cases}
- r_i
 & ( r_i <0)\\
\\
0 & (\mbox{\rm otherwise}).
\end{cases}
\]
In \cite[p35]{Sab93} Sabbah introduced the 
classical  Hukuhara-Levelt-Turrittin theorem as
\[\mbox{`` This result is analogous to 
Puiseux theorem for plane algebraic
(or algebroid) curves ".}\]
It is indeed the case for Fourier transforms of regular holonomic
$\D$-modules as we see in the following theorem.
\begin{theorem}\label{thm-neww}
The exponential factors appearing in the 
Hukuhara-Levelt-Turrittin decomposition
of the meromorphic connection $\SK(\ast\{v\})$ at $v\in M$
are the pole parts of $-\varphi_i\ (1\leq i \leq k)$.
Moreover for any $1\leq i\leq k$ the 
multiplicity of the pole part of $-\varphi_i$
is equal to $m_i$.
In particular we have
\[ {\rm irr} ( \SK(\ast\{v\}) ) = 
\sum_{i=1}^k m_i \cdot 
{\rm ord}_{\{ v \}}( \varphi_i ). \]
\end{theorem}

This result would be useful for the 
study of the irregularities of 
confluent $A$-hypergeometric functions. 
Recall that the irregularity 
${\rm irr} ( \SK(\ast\{v\}) )$ of 
$\SK(\ast\{v\})$ is a non-negative integer and 
equal to $- \chi_v \big(Sol_M ( \SK(\ast\{v\}) ) \big)$, 
where 
\[\chi_v\big(Sol_M( \SK(\ast\{v\}) )\big) :=
\sum_{j\in\ZZ}(-1)^j\dim H^jSol_M ( \SK(\ast\{v\}) )_v\]
is the local Euler-Poincar\'e index of 
$Sol_M( \SK(\ast\{v\}) )$ at the point $v \in M$ 
(see Sabbah \cite{Sab93} etc.). Moreover by 
Theorem \ref{th-A}, for linear subspaces 
$\LL \simeq \CC$ of the dual $Y= \CC^N$ such that 
$\LL \cap \Omega = \LL \setminus \{ 0 \}$ 
we obtain a 
formula for the exponential factors at infinity 
of the restrictions $\SM^\wedge |_{\LL}$ 
of $\SM^\wedge$ to them. 
See Theorem \ref{thm-3} for the details. 
This result extends 
(some part of) our previous one 
\cite[Theorem 5.5]{ET15} for 
confluent $A$-hypergeometric systems 
to Fourier transforms of general regular 
holonomic $\SD$-modules $\SM$. 
In the course of the proof of 
these results, we use the 
following technical result which may be 
of independent interest. 

\begin{proposition}\label{prop-A}
Let $X$ be a complex manifold, $D\subset X$
a normal crossing divisor in it and
$\SM_i\ (i=1, 2)$ $($analytic$)$ holonomic $\D_X$-modules. 
Denote by $\varpi_X : \tl{X}\to X$ 
the real blow-up of $X$ along $D$.
Let $V\subset X\bs D$ be an open sector in $X$ along $D$ 
and assume that we have an isomorphism
\begin{equation*}
\pi^{-1}\CC_V\otimes Sol_X^{\rmE}(\SM_1)
\simeq
\pi^{-1}\CC_V\otimes Sol_X^{\rmE}(\SM_2).
\end{equation*} 
Let $W\subset \tl{X}$ be an open subset of $\tl{X}$
such that $W\cap\varpi_X^{-1}(D)\neq\emptyset,
\var{W}\subset{\rm Int}\Big(\var{\varpi^{-1}_XV}\Big)$. 
Then we have an isomorphism
\begin{equation*}
\SM_1^\SA|_W
\simeq
\SM_2^\SA|_W
\end{equation*}
 of $\SD_{\tl{X}}^\SA$-modules on $W$ 
{\rm (}for the definition of $\SD_{\tl{X}}^\SA$ 
etc., see Section \ref{sec:8}{\rm )}. 
\end{proposition}

This result follows from a more essential one in 
Theorem \ref{thm-9}. 
Namely we can reconstruct the $\SD_{\tl{X}}^\SA$-module 
structure of $\SM^\SA$ on $W\subset \tl{X}$ 
by the enhanced ind-sheaf 
$\pi^{-1}\CC_V \otimes Sol_X^{\rmE}(\SM )$. 
We regard it as a directional (or sectorial) 
refinement of the irregular Riemann-Hilbert 
correspondence of \cite{DK16}. In fact 
Proposition \ref{prop-A} is a consequence of 
the extended Riemann-Hilbert 
correspondence of Kashiwara-Schapira 
\cite[Theorem 4.5]{KS16}.
For the converse of Proposition \ref{prop-A} see Theorem \ref{thm-10}.
If a meromorphic 
connection $\SM$ along $D$ admits a good 
lattice in the sense of Mochizuki 
\cite{Mochi11} we can also know 
its exponential factors from 
$\pi^{-1}\CC_V \otimes Sol_X^{\rmE}(\SM )$. 
See Theorem \ref{new-new-thm} for the details. 

 Let $\SM$ be an algebraic regular holonomic 
$\SD$-module on $X= \CC^N$. Then 
by our formula for the enhanced solution complex 
$Sol_{\var Y}^{\rmE}( \tl{\SM^\wedge} )$ 
we can calculate also some part of 
the characteristic cycle 
of the Fourier transform $\SM^\wedge$. 
To explain a special case of this result, first 
we define a ``conification" of the perverse sheaf 
$\SF = Sol_X(\SM)[N]\in\BDC_{\CC-c}(\CC_{X^\an})$ as follows. 
Let $j=i_X 
: X=\CC^N\xhookrightarrow{\ \ \ }\var{X}=\PP^N$ be the 
projective compactification 
of $X= \CC^N$ and $h$ the (local) defining equation of
the hyperplane at infinity 
$H_\infty :=\var{X}\bs X \simeq \PP^{N-1}$ in $\var{X}$
such that $H_\infty = h^{-1}(0)$. 
Moreover let $\gamma : X\bs\{0\}=
\CC^N\bs\{0\}\twoheadrightarrow
H_\infty=\PP^{N-1}$ be the canonical projection.
Then
\[\SG := \gamma^{-1} \psi_h(
j_! \SF)\in\BDC_{\CC-c}(\CC_{X\bs\{0\}})\]
is a perverse sheaf on $X\bs\{0\}$. 
More precisely, the nearby cycle sheaf 
$\psi_h( j_! \SF)$ is defined globally on $H_{\infty}$ 
by the corresponding $\SD$-module on it. 
We call it the conification of $\SF$. 
In particular, $\SG$ is monodromic. 
We extend it to a monodromic perverse sheaf 
on the whole $X= \CC^N$ and denote it also by $\SG$. 
Let $\SN\in\Modrh(\SD_X)$ be the regular holonomic 
$\SD_X$-module such that $Sol_X(\SN)[N]\simeq\SG$. 
Now we recall the well-known relationship between 
the characteristic cycle ${\rm CC}(\SN)$ of $\SN$ 
and that of its Fourier transform $\SN^\wedge$. 
Take $\CC^\ast$-conic subvarieties $V_i\subset X$ 
of $X$ and positive integer $n_i>0\ 
(1\leq i\leq r)$ such that 
\[{\rm CC}(\SN) = \sum_{i=1}^r 
n_i\cdot[T^\ast_{V_i}X].\]
Then by the natural identification 
$T^\ast X \simeq X \times Y \simeq T^\ast Y$ for any
$1\leq i\leq r$ there exists a 
$\CC^\ast$-conic subvariety 
$W_i\subset Y$ of $Y=\CC^N$ 
such that $T^\ast_{V_i}X=T^\ast_{W_i}Y$ 
(see Gelfand-Kapranov-Zelevinsky 
\cite[\S1.3]{GKZ94}). 
In this situation it is well-known that 
\[{\rm CC}(\SN^\wedge) = \sum_{i=1}^r n_i
\cdot[T^\ast_{W_i}Y].\]
For the Fourier transform $\SM^\wedge$ 
of the original $\SM$ we obtain the 
following result. 

\begin{theorem}\label{th-B} 
Assume that $d_{W_i}=N-1$ and 
$\F=Sol_X(\M)[N]\in\BDC_{\CC-c}(\CC_{X^{\an}})$ is moderate 
at infinity over a neighborhood of 
a generic point  $v\in (W_i)_{\rm reg}$ 
in $Y \setminus \{ 0 \}$ 
{\rm (}see Definition \ref{new-defc}{\rm )}. Then the multiplicity 
${\rm mult}_{T^\ast_{W_i}Y}\SM^\wedge \geq0$
of the Fourier transform $\SM^\wedge$ along $T^\ast_{W_i}Y$
is given by
\begin{align*}
{\rm mult}_{T^\ast_{W_i}Y}\SM^\wedge
&=
{\rm mult}_{T^\ast_{V_i}X}\SN
+ {\rm irr} ( \SK(\ast\{v\}) ) 
\\
(&\geq {\rm mult}_{T^\ast_{V_i}X}\SN =\mi > 0).
\end{align*}
In particular, the conormal bundle $T^\ast_{W_i}Y$ is
contained in the characteristic variety $\ch(\SM^\wedge)$ 
of $\SM^\wedge$ and we have $W_i\subset D= Y\bs\Omega$. 
\end{theorem}
For a more general formula, 
see Theorem \ref{new-thm-7}.

\section{Preliminary Notions and Results}\label{uni-sec:2}
In this section, we briefly recall some basic notions 
and results which will be used in this paper. 
We assume here that the reader is familiar with the 
theory of sheaves and functors in the framework of 
derived categories. For them we follow the terminologies 
in \cite{KS90} etc. For a topological space $M$ 
denote by $\BDC(\CC_M)$ the derived category 
consisting of bounded 
complexes of sheaves of $\CC$-vector spaces on it.

\subsection{Ind-sheaves}\label{sec:3}
We recall some basic notions 
and results on ind-sheaves. References are made to 
Kashiwara-Schapira \cite{KS01} and \cite{KS06}. 
Let $M$ be a good topological space (which is locally compact, 
Hausdorff, countable at infinity and has finite soft dimension). 
We denote by $\Mod(\CC_M)$ the abelian category of sheaves 
of $\CC$-vector spaces on it and by $\I\CC_M$ 
that of ind-sheaves. Then there exists a 
natural exact embedding $\iota_M : \Mod(\CC_M)\to\I\CC_M$ 
of categories. We sometimes omit it.
It has an exact left adjoint $\alpha_M$, 
that has in turn an exact fully faithful
left adjoint functor $\beta_M$: 
 \[\xymatrix@C=60pt{\Mod(\CC_M)  \ar@<1.0ex>[r]^-{\iota_{M}} 
 \ar@<-1.0ex>[r]_- {\beta_{M}} & \I\CC_M 
\ar@<0.0ex>[l]|-{\alpha_{M}}}.\]

The category $\I\CC_M$ does not have enough injectives. 
Nevertheless, we can construct the derived category $\BDC(\I\CC_M)$ 
for ind-sheaves and the Grothendieck six operations among them. 
We denote by $\otimes$ and $\rihom$ the operations 
of tensor products and internal homs respectively. 
If $f : M\to N$ be a continuous map, we denote 
by $f^{-1}, \rmR f_\ast, f^!$ and $\rmR f_{!!}$ the 
operations of inverse images,
direct images, proper inverse images and proper direct images 
respectively. 
We set also $\rhom := \alpha_M\circ\rihom$. 
We thus obtain the functors
\begin{align*}
\iota_M &: \BDC(\CC_M)\to \BDC(\I\CC_M),\\
\alpha_M &: \BDC(\I\CC_M)\to \BDC(\CC_M),\\
\beta_M &: \BDC(\CC_M)\to \BDC(\I\CC_M),\\
\otimes &: \BDC(\I\CC_M)\times\BDC(\I\CC_M)\to\BDC(\I\CC_M), \\
\rihom &: \BDC(\I\CC_M)^{\op}\times\BDC(\I\CC_M)\to\BDC(\I\CC_M), \\
\rhom &: \BDC(\I\CC_M)^{\op}\times\BDC(\I\CC_M)\to\BDC(\CC_M), \\
\rmR f_\ast &: \BDC(\I\CC_M)\to\BDC(\I\CC_N),\\
f^{-1} &: \BDC(\I\CC_N)\to\BDC(\I\CC_M),\\
\rmR f_{!!} &: \BDC(\I\CC_M)\to\BDC(\I\CC_N),\\
f^! &: \BDC(\I\CC_N)\to\BDC(\I\CC_M).
\end{align*}

Note that $(f^{-1}, \rmR f_\ast)$ and 
$(\rmR f_{!!}, f^!)$ are pairs of adjoint functors.
We may summarize the commutativity of the various functors 
we have introduced in the table below. 
Here, $``\circ"$ means that the functors commute,
and $``\times"$ they do not.
\begin{table}[h]
\begin{equation*}
   \begin{tabular}{l||c|c|c|c|c|c|c}
    {} & $\otimes$ & $f^{-1}$ & $\rmR f_\ast$ & $f^!$ & 
$\rmR f_{!!}$ & $\underset{}{\varinjlim}$ &  $\varprojlim$ \\ \hline \hline 
    $\underset{}{\iota}$ & $\circ$ & $\circ$ & $\circ$ & 
$\circ$ & $\times$ & $\times$ & $\circ$  \\ \hline
    $\underset{}{\alpha}$ & $\circ$ & $\circ$ & $\circ$ & 
$\times$ & $\circ$ & $\circ$ & $\circ$ \\ \hline
    $\underset{}{\beta}$ & $\circ$ & $\circ$ & $\times$  
& $\times$ & $\times$ & $\circ$& $\times$\\ \hline 
     $\underset{}{\varinjlim}$ & $\circ$ & $\circ$ & 
$\times$ & $\circ$ & $\circ$ &\multicolumn{2}{|c}{} \\\cline{1-6}
     $\underset{}{\varprojlim}$ & $\times$ & $\times$ & 
$\circ$ & $\times$ & $\times$ &\multicolumn{2}{|c}{} \\\cline{1-6}
   \end{tabular}
   \end{equation*}
\end{table}

\subsection{Ind-sheaves on Bordered Spaces}\label{sec:4} 
For the results in this subsection, we refer to 
D'Agnolo-Kashiwara \cite{DK16}. 
A bordered space is a pair $M_{\infty} = (M, \che{M})$ of
a good topological space $\che{M}$ and an open subset $M\subset\che{M}$.
A morphism $f : (M, \che{M})\to (N, \che{N})$ of bordered spaces
is a continuous map $f : M\to N$ such that the first projection
$\che{M}\times\che{N}\to\che{M}$ is proper on
the closure $\var{\Gamma}_f$ of the graph $\Gamma_f$ of $f$ 
in $\che{M}\times\che{N}$.
If also the second projection $\var{\Gamma}_f\to\che{N}$ is proper, 
we say that $f$ is semi-proper. 
The category of good topological spaces embeds into that
of bordered spaces by the identification $M = (M, M)$. 
We define the triangulated category of ind-sheaves on 
$M_{\infty} = (M, \che{M})$ by 
\[\BDC(\I\CC_{M_\infty}) := 
\BDC(\I\CC_{\che{M}})/\BDC(\I\CC_{\che{M}\backslash M}).\]
The quotient functor
\[\mathbf{q} : \BDC(\I\CC_{\che{M}})\to\BDC(\I\CC_{M_\infty})\]
has a left adjoint $\mathbf{l}$ and a right 
adjoint $\mathbf{r}$, both fully faithful, defined by 
\[\mathbf{l}(\mathbf{q} F) := \CC_M\otimes F,\hspace{25pt} 
\mathbf{r}(\mathbf{q} F) := \rihom(\CC_M, F). \]
For a morphism $f : M_\infty\to N_\infty$ 
of bordered spaces, 
the Grothendieck's operations 
\begin{align*} 
\otimes &: \BDC(\I\CC_{M_\infty})\times
\BDC(\I\CC_{M_\infty})\to\BDC(\I\CC_{M_\infty}), \\
\rihom &: \BDC(\I\CC_{M_\infty})^{\op}\times
\BDC(\I\CC_{M_\infty})\to\BDC(\I\CC_{M_\infty}), \\
\rmR f_\ast &: \BDC(\I\CC_{M_\infty})\to\BDC(\I\CC_{N_\infty}),\\
f^{-1} &: \BDC(\I\CC_{N_\infty})\to\BDC(\I\CC_{M_\infty}),\\
\rmR f_{!!} &: \BDC(\I\CC_{M_\infty})\to\BDC(\I\CC_{N_\infty}),\\
f^! &: \BDC(\I\CC_{N_\infty})\to\BDC(\I\CC_{M_\infty}) \\
\end{align*}
are defined by
\begin{align*} 
\mathbf{q}(F)\otimes\mathbf{q}(G) &:=
\mathbf{q}(F\otimes G), \\
\rihom(\mathbf{q}(F), \mathbf{q}(G)) &:=
\mathbf{q}\big(\rihom(F, G)\big), \\
\rmR f_\ast(\mathbf{q}(F)) &:=
\mathbf{q}\big(\rmR {\rm pr}_{2\ast}\rihom(
\CC_{\Gamma_f}, {\rm pr}_1^!F)\big),\\
f^{-1} (\mathbf{q}(G))&:=\mathbf{q}\big(\rmR{\rm pr_{1!!}}
(\CC_{\Gamma_f}\otimes {\rm pr_2}^{-1}G)\big),\\
\rmR f_{!!}(\mathbf{q}(F)) &:=
\mathbf{q}\big(\rmR{\rm pr_2}_{!!}(
\CC_{\Gamma_f}\otimes{\rm pr_1}^{-1}F)\big),\\
f^!(\mathbf{q}(G)) &:=\mathbf{q}
\big(\rmR {\rm pr}_{1\ast}\rihom(\CC_{\Gamma_f}, {\rm pr}_2^!G)\big)
\end{align*}
respectively, where ${\rm pr_1} : \che{M}\times\che{N}\to\che{M}$ and 
${\rm pr_2} : \che{M}\times\che{N}\to\che{N}$ 
are the projections. Moreover, there exists a natural embedding 
\[\xymatrix@C=30pt@M=10pt{\BDC(\CC_M)\ar@{^{(}->}[r] & 
\BDC(\I\CC_{M_\infty}).}\]

\subsection{Enhanced Sheaves}\label{sec:5}
For the results in this subsection, see 
Tamarkin \cite{Tama08},
Kashiwara-Schapira \cite{KS16-2} and 
D'Agnolo-Kashiwara \cite{DK17}. 
Let $M$ be a good topological space. 
We consider the maps 
\[M\times\RR^2\xrightarrow{p_1, p_2, \mu}M
\times\RR\overset{\pi}{\longrightarrow}M\]
where $p_1, p_2$ are the first and the second projections 
and we set $\pi (x,t):=x$ and 
$\mu(x, t_1, t_2) := (x, t_1+t_2)$. 
Then the convolution functors for 
sheaves on $M \times \RR$ are defined by
\begin{align*}
F_1\Potimes F_2 &:= \rmR \mu_!(p_1^{-1}F_1\otimes p_2^{-1}F_2),\\
\Prhom(F_1, F_2) &:= \rmR p_{1\ast}\rhom(p_2^{-1}F_1, \mu^!F_2).
\end{align*}
We define the triangulated category of enhanced sheaves on $M$ by 
\[\BEC(\CC_M) := \BDC(\CC_{M\times\RR})/\pi^{-1}\BDC(\CC_M). \] 
Then the quotient functor
\[\Q : \BDC(\CC_{M \times \RR} )  \to\BEC(\CC_M)\]
has fully faithful left and right adjoints $\bfL^\rmE, \bfR^\rmE$ defined by 
\[\bfL^\rmE(\Q F) := (\CC_{\{t\geq0\}}\oplus\CC_{\{t\leq 0\}})
\Potimes F ,\hspace{20pt} \bfR^\rmE(\Q G) :
=\Prhom(\CC_{\{t\geq0\}}\oplus\CC_{\{t\leq 0\}}, G), \]
where $\{t\geq0\}$ stands for $\{(x, t)\in 
M\times\RR\ |\ t\geq0\}$ and $\{t\leq0\}$ is 
defined similarly. The convolution functors 
are defined also for enhanced sheaves. We denote them 
by the same symbols $\Potimes$, $\Prhom$. 
For a continuous map $f : M \to N $, we 
can define naturally the operations 
$\bfE f^{-1}$, $\bfE f_\ast$, $\bfE f^!$, $\bfE f_{!}$ 
for enhanced sheaves. 
We have also a 
natural embedding $\e : \BDC( \CC_M) \to \BEC( \CC_M)$ defined by 
\[ \e(F) := \Q(\CC_{\{t\geq0\}}\otimes\pi^{-1}F). \] 
For a continuous function $\varphi : U\to \RR$ 
defined on an open subset $U \subset M$ of $M$  we define 
the exponential enhanced sheaf by 
\[ {\rm E}_{U|M}^\varphi := 
\Q(\CC_{\{t+\varphi\geq0\}} ), \]
where $\{t+\varphi\geq0\}$ stands for 
$\{(x, t)\in M\times{\RR}\ |\ x\in U, t+\varphi(x)\geq0\}$.

\subsection{Enhanced Ind-sheaves}\label{sec:6}
We recall some basic notions 
and results on enhanced ind-sheaves. References are made to 
D'Agnolo-Kashiwara \cite{DK16} and 
Kashiwara-Schapira \cite{KS16}. 
Let $M$ be a good topological space.
Set $\RR_\infty := (\RR, \var{\RR})$ for 
$\var{\RR} := \RR\sqcup\{-\infty, +\infty\}$,
and let $t\in\RR$ be the affine coordinate. 
We consider the maps 
\[M\times\RR_\infty^2\xrightarrow{p_1, p_2, \mu}M
\times\RR_\infty\overset{\pi}{\longrightarrow}M\]
where $p_1, p_2$ and $\pi$ are morphisms
of bordered spaces induced by the projections.
And $\mu$ is a morphism of bordered spaces induced by the map 
$M\times\RR^2\ni(x, t_1, t_2)\mapsto(x, t_1+t_2)\in M\times\RR$.
Then the convolution functors for 
ind-sheaves on $M \times \RR_\infty$ are defined by
\begin{align*}
F_1\Potimes F_2 &:= \rmR\mu_{!!}(p_1^{-1}F_1\otimes p_2^{-1}F_2),\\
\Prihom(F_1, F_2) &:= \rmR p_{1\ast}\rihom(p_2^{-1}F_1, \mu^!F_2).
\end{align*}
Now we define the triangulated category 
of enhanced ind-sheaves on $M$ by 
\[\BEC(\I\CC_M) := \BDC(\I\CC_{M 
\times\RR_\infty})/\pi^{-1}\BDC(\I\CC_M).\]
Note that we have a natural embedding of categories
\[\BEC(\CC_M) \xhookrightarrow{\ \ \ }\BEC(\I\CC_M).\]
The quotient functor
\[\Q : \BDC(\I\CC_{M\times\RR_\infty})\to\BEC(\I\CC_M)\]
has fully faithful left and right adjoints $\bfL^\rmE,\bfR^\rmE$ defined by 
\[\bfL^\rmE(\Q K) := (\CC_{\{t\geq0\}}\oplus\CC_{\{t\leq 0\}})
\Potimes K ,\hspace{20pt} \bfR^\rmE(\Q K) :
=\Prihom(\CC_{\{t\geq0\}}\oplus\CC_{\{t\leq 0\}}, K), \]
where $\{t\geq0\}$ stands for 
$\{(x, t)\in M\times\var{\RR}\ |\ t\in\RR, t\geq0\}$ 
and $\{t\leq0\}$ is defined similarly.

The convolution functors 
are defined also for enhanced ind-sheaves. We denote them 
by the same symbols $\Potimes$, $\Prihom$. 
For a continuous map $f : M \to N $, we 
can define also the operations 
$\bfE f^{-1}$, $\bfE f_\ast$, $\bfE f^!$, $\bfE f_{!!}$ 
for enhanced ind-sheaves. For example, 
by the natural morphism $\tl{f}: M \times \RR_{\infty} \to 
N \times \RR_{\infty}$ of bordered spaces associated to 
$f$ we set $\bfE f_\ast ( \Q K)= \Q(\rmR \tl{f}_{\ast}(K))$. 
The other operations are defined similarly. 
We thus obtain the six operations $\Potimes$, $\Prihom$,
$\bfE f^{-1}$, $\bfE f_\ast$, $\bfE f^!$, $\bfE f_{!!}$ 
for enhanced ind-sheaves .
Moreover we denote by $\rmD_M^\rmE$ 
the Verdier duality functor for enhanced ind-sheaves.
We have outer hom functors
\begin{align*}
\rihom^\rmE(K_1, K_2) &:= \rmR\pi_\ast\rihom(\bfL^\rmE K_1, \bfL^\rmE K_2)
\simeq \rmR\pi_\ast\rihom(\bfL^\rmE K_1, \bfR^\rmE K_2),\\
\rhom^\rmE(K_1, K_2) &:= \alpha_M\rihom^\rmE(K_1, K_2),\\
\rHom^\rmE(K_1, K_2) &:=\rmR\Gamma(M; \rhom^\rmE(K_1, K_2)),
\end{align*}
with values in $\BDC(\I\CC_M), 
\BDC(\CC_M)$ and $\BDC(\CC)$, respectively. 
Moreover for $F\in\BDC(\I\CC_M)$ and $K\in\BEC(\I\CC_M)$ the objects 
\begin{align*}
\pi^{-1}F\otimes K &:=\Q(\pi^{-1}F\otimes \bfL^\rmE K),\\
\rihom(\pi^{-1}F, K) &:=\Q\big(\rihom(\pi^{-1}F, \bfR^\rmE K)\big). 
\end{align*}
in $\BEC(\I\CC_M)$ are well-defined. 
Set $\CC_M^\rmE := \Q 
\Bigl(``\underset{a\to +\infty}{\varinjlim}"\ \CC_{\{t\geq a\}}
\Bigr)\in\BEC(\I\CC_M)$. 
Then we have 
natural embeddings $\e, e : \BDC(\I\CC_M) \to \BEC(\I\CC_M)$ defined by 
\begin{align*}
\e(F) & := \Q(\CC_{\{t\geq0\}}\otimes\pi^{-1}F) \\ 
e(F) &:=  \CC_M^\rmE\otimes\pi^{-1}F
\simeq \CC_M^\rmE\Potimes\e(F). 
\end{align*}

For a continuous function $\varphi : U\to \RR$ 
defined on an open subset $U \subset M$ of $M$  we define 
the exponential enhanced ind-sheaf by 
\[\EE_{U|M}^\varphi := 
\CC_M^\rmE\Potimes {\rm E}_{U|M}^\varphi
=
\CC_M^\rmE\Potimes
\Q\CC_{\{t+\varphi\geq0\}}
=
 \Q 
\Bigl(``\underset{a\to +\infty}{\varinjlim}"\ \CC_{\{t+\varphi\geq a\}}
\Bigr)\]
where $\{t+\varphi\geq0\}$ stands for 
$\{(x, t)\in M\times\var{\RR}\ |\ t\in\RR, x\in U, t+\varphi(x)\geq0\}$.

\subsection{$\SD$-Modules}\label{sec:7}
In this subsection we recall some basic notions 
and results on $\SD$-modules. 
References are made to \cite{Bjo93}, 
\cite{HTT08}, \cite[\S 7]{KS01}, 
\cite[\S 8, 9]{DK16},
\cite[\S 3, 4, 7]{KS16} and 
\cite[\S 4, 5, 6, 7, 8]{Kas16}. 
For a complex manifold $X$ we denote by 
$d_X$ its complex dimension. 
Denote by $\SO_X, \Omega_X$ and $\SD_X$ 
the sheaves of holomorphic functions, 
holomorphic differential forms of top degree 
and holomorphic differential operators, respectively. 
Let $\BDC(\SD_X)$ be the bounded derived category 
of left $\SD_X$-modules and $\BDC(\SDop_X)$ 
be that of right $\SD_X$-modules. 
Moreover we denote by $\BDCcoh(\SD_X)$, $\BDC_{\rm good}(\SD_X)$,
$\BDChol(\SD_X)$ and $\BDCrh(\SD_X)$ the full triangulated subcategories
of $\BDC(\SD_X)$ consisting of objects with coherent, good, 
holonomic and regular holonomic cohomologies, 
respectively.
For a morphism $f : X\to Y$ of complex manifolds, 
denote by $\Dotimes, \rhom_{\SD_X}, \bfD f_\ast, \bfD f^\ast$ 
the standard operations for $\SD$-modules. 
We define also the duality functor $\DD_X : \BDCcoh(\SD_X)^{\op} 
\simto \BDCcoh(\SD_X)$ 
by 
\[\DD_X(\SM):=\rhom_{\SD_X}(\SM, \SD_X)
\uotimes{\SO_X}\Omega^{\otimes-1}_X[d_X].\]
Note that there exists 
an equivalence of categories 
$( \cdot )^{\rm r} : \Mod(\SD_X)\simto\Mod(\SDop_X)$ given by 
\[\SM^{\rm r}:=\Omega_X\uotimes{\SO_X}\SM.\]
The classical de Rham and solution functors are defined by  
\begin{align*}
DR_X &:  \BDCcoh (\SD_X)\to\BDC(\CC_X),
\hspace{40pt}\SM \longmapsto \Omega_X\Lotimes{\SD_X}\SM, \\
Sol_X &: \BDCcoh (\SD_X)^{\op}\to\BDC(\CC_X),
\hspace{30pt}\SM \longmapsto \rhom_{\SD_X}(\SM, \SO_X).
\end{align*}
Then for $\SM\in\BDCcoh(\SD_X)$ 
we have an isomorphism $Sol_X(\SM)[d_X]\simeq DR_X(\DD_X\SM)$. 
For a closed hypersurface $D\subset X$ 
in $X$ we denote by $\SO_X(\ast D)$ 
the sheaf of meromorphic functions on $X$ with poles in $D$. 
Then for $\SM\in\BDC(\SD_X)$ we set 
\[\SM(\ast D) := \SM\Dotimes\SO_X(\ast D).\]
For $f\in\SO_X(\ast D)$ and $U := X\bs D$, set 
\begin{align*}
\SD_Xe^f &:= \SD_X/\{P \in \SD_X  \ |\ Pe^f|_U = 0\}, \\
\SE_{U|X}^f &:= \SD_Xe^f(\ast D).
\end{align*}
Note that $\SE_{U|X}^f$ is holonomic and there exists an isomorphism 
\[\DD_X(\SE_{U|X}^f)(\ast D) \simeq \SE_{U|X}^{-f}.\]
Namely $\SE_{U|X}^f$ is a meromorphic connection associated 
to $d+df$. 

One defines the ind-sheaf $\SO_X^\rmt$ of tempered 
holomorphic functions 
as the Dolbeault complex with coefficients in the ind-sheaf 
of tempered distributions. 
More precisely, denoting by $\overline{X}$ the complex conjugate manifold 
to $X$ and by $X_{\RR}$ the underlying real analytic manifold of $X$,
we set
\[\SO_X^\rmt := \rihom_{\SD_{\overline{X}}}(
\SO_{\overline{X}}, \mathcal{D}b_{X_\RR}^\rmt), \]
where $\mathcal{D}b_{X_\RR}^\rmt$ is the ind-sheaf of 
tempered distributions on $X_\RR$
(for the definition see \cite[Definition 7.2.5]{KS01}). 
Moreover, we set 
 \[\Omega_X^\rmt := \beta_X\Omega_X\otimes_{\beta_X\SO_X}\SO_X^\rmt.\] 
Then the tempered de Rham and solution functors 
are defined by 
\begin{align*}
DR_X^\rmt &: \BDCcoh (\SD_X)\to\BDC(\I\CC_X),
\hspace{40pt} 
\SM \longmapsto \Omega_X^\rmt\Lotimes{\SD_X}\SM, 
\\
Sol_X^\rmt &: \BDCcoh (\SD_X)^{\op}\to\BDC(\I\CC_X), 
\hspace{40pt}
\SM \longmapsto \rihom_{\SD_X}(\SM, \SO_X^\rmt). 
\end{align*}
Note that we have isomorphisms
\begin{align*}
Sol_X(\SM) &\simeq \alpha_XSol_X^\rmt(\SM), \\
DR_X(\SM) &\simeq \alpha_XDR_X^\rmt(\SM), \\
Sol_X^\rmt(\SM)[d_X] &\simeq DR_X^\rmt(\DD_X\SM).
\end{align*}

Let $i : X\times\RR_\infty\to X\times\PP$ be
the natural morphism of bordered spaces and
$\tau\in\CC\subset\PP$ the affine coordinate 
such that $\tau|_\RR$ is that of $\RR$. 
We then define objects $\SO_X^\rmE\in\BEC(\I\SD_X)$ and 
$\Omega_X^\rmE\in\BEC(\I\SD_X^{\op})$ by 
\begin{align*}
\SO_X^\rmE &:= \rihom_{\SD_{\overline{X}}}(
\SO_{\overline{X}}, \mathcal{D}b_{X_\RR}^{\T})\\
&\simeq i^!\bigl((\SE_{\CC|\PP}^{-\tau})^{\rm r} 
\Lotimes{\SD_\PP}\SO_{X\times\PP}^\rmt\bigr)[1]
\simeq i^!\rihom_{\SD_\PP}(\SE_{\CC|\PP}^{\tau}, 
\SO_{X\times\PP}^\rmt)[2],\\
\Omega_X^\rmE &:= \Omega_X\Lotimes{\SO_X}\SO_X^\rmE\simeq 
i^!(\Omega_{X\times\PP}^\rmt\Lotimes{\SD_\PP}
\SE_{\CC|\PP}^{-\tau})[1], 
\end{align*}
where $\mathcal{D}b_{X_\RR}^{\T}$ 
stand for the enhanced ind-sheaf 
of tempered distributions on $X_\RR$ 
(for the definition see \cite[Definition 8.1.1]{DK16}). 
We call $\SO_X^\rmE$ the enhanced ind-sheaf 
of tempered holomorphic functions. 
Note that there exists an isomorphism
\[i_0^!\bfR^\rmE\SO_X^\rmE\simeq\SO_X^\rmt, \]
where $i_0 : X\to X\times\RR_\infty$ is the 
inclusion map of bordered spaces induced by $x\mapsto (x, 0)$. 
The enhanced de Rham and solution functors 
are defined by 
\begin{align*}
DR_X^\rmE &: \BDCcoh (\SD_X)\to\BEC(\I\CC_X),
\hspace{40pt} 
\SM \longmapsto \Omega_X^\rmE\Lotimes{\SD_X}\SM,
\\
Sol_X^\rmE &: \BDCcoh (\SD_X)^{\op}\to\BEC(\I\CC_X), 
\hspace{40pt} 
\SM \longmapsto \rihom_{\SD_X}(\SM, \SO_X^\rmE). 
\end{align*}
Then for $\SM\in\BDCcoh(\SD_X)$ 
we have isomorphism $Sol_X^\rmE(\SM)[d_X]\simeq DR_X^\rmE(\DD_X\SM)$ 
and 
$Sol^{\rmt}_X(\M)\simeq
i_0^!{\bfR^\rmE}Sol_X^{\rmE}(\M).$ 
We recall the following results of \cite{DK16}.
\begin{theorem}
\begin{enumerate}\label{thm-4}
\item[\rm{(i)}] For $\SM\in\BDC_{\rm hol}(\SD_X)$ there
is an isomorphism in $\BEC(\I\CC_X)$
\[\rmD_X^\rmE\bigl(DR_X^\rmE(\SM)\bigr) \simeq Sol_X^\rmE(\SM)[d_X].\]

\item[\rm{(ii)}] Let $f : X\to Y$ be a morphism of complex manifolds.
Then for $\SN\in\BDC_{\rm hol}(\SD_Y)$ 
there is an isomorphism in $\BEC(\I\CC_X)$
\[Sol_X^\rmE({\bfD} f^\ast\SN) \simeq \bfE f^{-1}Sol_Y^\rmE(\SN).\]

\item[\rm{(iii)}] Let $f : X\to Y$ be a morphism of complex manifolds
and $\SM\in\BDC_{\rm good}(\SD_X)\cap\BDC_{\rm hol}(\SD_X)$.
If $\supp(\SM)$ is proper over Y then there 
is an isomorphism in $\BEC(\I\CC_Y)$
\[Sol_Y^\rmE({\bfD} f_\ast\SM)[d_Y] \simeq 
\bfE f_\ast Sol_X^\rmE(\SM )[d_X].\]

\item[\rm{(iv)}] For $\SM_1, \SM_2\in\BDC_{\rm hol}(\SD_X)$, 
there exists an isomorphism in $\BEC(\I\CC_X)$
\[Sol_X^\rmE(\SM_1\Dotimes\SM_2)\simeq Sol_X^\rmE(\SM_1)
\Potimes Sol_X^\rmE(\SM_2).\]

\item[\rm{(v)}] If $\SM\in\BDC_{\rm hol}(\SD_X)$ and $D\subset X$
is a closed hypersurface, then there 
are isomorphisms in $\BEC(\I\CC_X)$
\begin{align*}
Sol_X^\rmE(\SM(\ast D)) &\simeq \pi^{-1}
\CC_{X\bs D}\otimes Sol_X^\rmE(\SM),\\
DR_X^\rmE(\SM(\ast D)) &\simeq \rihom(
\pi^{-1}\CC_{X\bs D}, DR_X^\rmE(\SM)).
\end{align*}

\item[\rm{(vi)}] Let $D$ be a closed hypersurface in $X$ and 
$f\in\SO_X(\ast D)$ a meromorphic function along $D$.
Then there exists an isomorphism in $\BEC(\I\CC_X)$
\[Sol_X^\rmE\big(\mathscr{E}_{X\backslash D | X}^\varphi\big) 
\simeq \EE_{X\backslash D | X}^{\Re\varphi}.\]
\end{enumerate}
\end{theorem}

Finally, we also recall the 
following theorem of \cite{KS16}
\begin{theorem}[{\cite[Theorem 4.5 (Extended Riemann-Hilbert 
Correspondence)]{KS16}}]\label{thm-5}
There exists an isomorphism functorial 
with respect to $\SM\in\BDC_{\rm hol}(\SD_X) : $
\[\SM\Lotimes{\SO_X}\SO_X^\rmE
\simto
\rihom^+(Sol_X^\rmE(\SM), \SO_X^\rmE) \]
in $\BEC(\I\SD_X)$.
Moreover, there exists an isomorphism functorial
with respect to $\SM\in\BDC_{\rm hol}(\SD_X) :$
\[
\SM\Dotimes\SO_X^\rmt
\simto \rihom^\rmE(Sol_X^\rmE(\SM), \SO_X^\rmE)
\]
in $\BDC(\I\SD_X)$.
\end{theorem}

\begin{corollary}[{\cite[Theorem 9.5.3 (Irregular 
Riemann-Hilbert Correspondence)]{DK16}}]\label{cor-4}
There exists an isomorphism functorial 
with respect to $\SM\in\BDC_{\rm hol}(\SD_X) :$
\[\SM\simto \rhom^\rmE(Sol_X^\rmE(\SM), \SO_X^\rmE)\]
in $\BDC_{\rm hol}(\SD_X)$.
\end{corollary}

\section{Some Auxiliary Results on Meromorphic 
Connections etc.}\label{sec:8}
In this section we prepare some 
auxiliary results on meromorphic connections etc. 
First we recall some notions and results on 
$\SD_{\tl{X}}^\SA$ in \cite[\S 7]{DK16}. 
Let $X$ be a complex manifold and $D \subset X$ a 
normal crossing divisor in it. Denote by 
$\varpi_X : \tl{X}\to X$ the real blow-up of $X$ along 
$D$ (sometimes we denote it simply by $\varpi$). 
Then we set
\begin{align*}
\SO_{\tl{X}}^\rmt &:= \rhom_{\varpi^{-1}
\SD_{\overline{X}}}(\varpi^{-1}\SO_{\overline{X}}, 
\mathcal{D}b_{\tl{X_\RR}}^\rmt),\\
\SA_{\tl{X}} &:= \alpha_{\tl{X}}\SO_{\tl{X}}^\rmt,\\
\SD_{\tl{X}}^\SA &:= \SA_{\tl{X}}\otimes_{\varpi^{-1}\SO_X}
\varpi^{-1}\SD_X, 
\end{align*}
where $\mathcal{D}b_{\tl{X}}^\rmt$ stands for the ind-sheaf 
of tempered distributions on $\tl{X}$ 
(for the definition see \cite[Notation 7.2.4]{DK16}). 
Recall that a section of $\SA_{\tl X}$ is a holomorphic function having
moderate growth at $\varpi_X^{-1}(D)$.
Note that $\SA_{\tl{X}}$ and 
$\SD_{\tl{X}}^\SA$ are sheaves of rings on $\tl{X}$. 
We define also enhanced ind-sheaves 
$\SO_{\tl{X}}^\rmE\in\BEC(\I\SD_{\tl{X}}^{\SA})$
and $\Omega_{\tl{X}}^\rmE\in\BEC(\I(\SD_{\tl{X}}^{\SA})^{\op})$ by 
\begin{align*}
\SO_{\tl{X}}^\rmE &:= \rihom_{\varpi^{-1}
\SD_{\overline{X}}}(\varpi^{-1}\SO_{\overline{X}},
 \mathcal{D}b_{\tl{X_\RR}}^{\T})\\
&\simeq k^!\bigl((\SE_{\CC|\PP}^{-\tau})^{\rm r} 
\Lotimes{\SD_\PP}\SO_{\tl{X}\times\PP}^\rmt\bigr)[1]
\simeq k^!\rihom_{\SD_\PP}(\SE_{\CC|\PP}^{\tau}, 
\SO_{\tl{X}\times\PP}^\rmt)[2],\\
\Omega_{\tl{X}}^\rmE &:= \varpi^{-1}\Omega_X
\Lotimes{\varpi^{-1}\SO_X}\SO_{\tl{X}}^\rmE
\simeq k^!(\Omega_{\tl{X}\times\PP}^\rmt
\Lotimes{\SD_\PP}\SE_{\CC|\PP}^{-\tau})[1],
\end{align*}
where $k : \tl{X}\times\RR_\infty\to\tl{X}\times\PP$ is 
the natural morphism of bordered spaces
and $\mathcal{D}b_{\tl{X_\RR}}^{\T}$ 
stands for the enhanced ind-sheaf 
of tempered distributions on $\tl{X_\RR}$ 
(for the definition see \cite[(7.6.1)]{KS16}). 

For $\SM\in\BDC(\SD_X)$ we define 
an object $\SM^{\SA}\in\BDC(\SD_{\tl{X}}^\SA)$ by 
\begin{align*}
\SM^{\SA} &:= \SD_{\tl{X}}^\SA\Lotimes
{\varpi^{-1}\SD_X}\varpi^{-1}\SM
\simeq\SA_{\tl{X}}
\Lotimes{\varpi^{-1}\SO_X}\varpi^{-1}\SM.
\end{align*}
Note that if $\SM$ is a holonomic $\SD_X$-module such that 
$\SM\simto\SM(\ast D)$ and 
${\rm sing.supp} (\SM)\subset D$,
then one has $\SM^\SA\simeq
\SD_{\tl{X}}^\SA\otimes_{\varpi^{-1}\SD_X}\varpi^{-1}\SM$
(see \cite[Lemma 7.3.2]{DK16}).
Moreover we have an isomorphism $\M^\SA\simto \M(\ast D)^\SA$
for any holonomic $\D_X$-module $\M$ (see \cite[Lemma 7.2.2]{DK16}).
For $\mathscr{M}\in\BDC(\SD_{\tl{X}}^\SA)$ we define
the enhanced de Rham and solution functors on $\tl{X}$ by 
\begin{align*}
Sol_{\tl{X}}^\rmE(\mathscr{M}) &:= 
\rihom_{\SD_{\tl{X}}^\SA}(\mathscr{M}, \SO_{\tl{X}}^\rmE), \\
DR_{\tl{X}}^\rmE(\mathscr{M}) &:= 
\Omega_{\tl{X}}^\rmE\Lotimes{\SD_{\tl{X}}^\SA}\mathscr{M}
\end{align*}
respectively. 
Recall that we have isomorphisms
\begin{align*}
DR_{\tl{X}}^\rmE(\SM^\SA) &\simeq
\bfE\varpi^!DR_X^\rmE\bigl(\SM(\ast D)\bigr)
\simeq
\bfE\varpi^! \rihom(\pi^{-1}\CC_{X\setminus D}, DR_X(\M)),\\
Sol_{\tl{X}}^\rmE(\SM^\SA)&\simeq
\bfE\varpi^! \rihom(\pi^{-1}\CC_{X\setminus D}, Sol_X^\rmE(\M))
\end{align*}
and
\begin{align*}
\bfE\varpi_\ast DR_{\tl{X}}^\rmE(\SM^\SA)&\simeq
DR_X^\rmE\bigl(\SM(\ast D)\bigr)\simeq
\rihom(\pi^{-1}\CC_{X\setminus D}, DR_X^\rmE(\M)),\\
\bfE\varpi_\ast Sol_{\tl{X}}^\rmE(\SM^\SA)&\simeq
\rihom(\pi^{-1}\CC_{X\setminus D}, Sol_X^\rmE(\M))
\end{align*}
for $\SM\in\BDChol(\SD_X)$
(see \cite[Corollary 9.2.3, p191 and Theorem 9.2.2]{DK16}).

\begin{definition}\label{def-A}
Let $X$ be a complex manifold and $D \subset X$ a 
normal crossing divisor in it. 
Then we say that a holonomic $\SD_X$-module 
$\SM$ has a normal form along $D$ if\\
(i) $\SM\simto\SM(\ast D)$\\
(ii) $\rm{sing.supp}(\SM)\subset D$\\
(iii) for any $\theta\in\varpi^{-1}(D)\subset\tl{X}$,
there exist an open neighborhood $U\subset X$
of $\varpi({\theta})$,
finitely many $\varphi_i\in\Gamma(U;\SO_X(\ast D))$
and an open neighborhood $V$ of $\theta$ 
with $V\subset\varpi^{-1}(U)$
such that
\[
\SM^\SA|_V
\simeq
\Bigl(
\bigoplus_i\bigl(\SE_{U\bs D|U}^{\varphi_i}\bigr)^\SA
\Bigr)
|_V.\]
\end{definition}

\begin{lemma}\label{lem-1}
Let $X$ be a complex manifold and 
$D\subset X$ a normal crossing divisor in it
and $\M$ a holonomic $\SD_X$-module.
Then for the dual $(\SM^\SA)^\ast := 
\rhom_{\SD_{\tl{X}}^\SA}(\SM^\SA, \SD_{\tl{X}}^\SA)
\otimes_{\varpi^{-1} \SO_X }\varpi^{-1}\Omega_X^{\otimes-1}[d_X]$
of the $\SD_{\tl{X}}^\SA$-modules $\SM^\SA$ 
and the holonomic $\SD_X$-module $\DD_X(\SM)(\ast D)$ 
we have an isomorphism
\[(\SM^\SA)^\ast\simeq \Bigl(\DD_X(\SM)(\ast D)\Bigr)^\SA.\]
In particular, there exists an isomorphism
\[DR_{\tl{X}}^\rmE\Bigl(\bigl(\DD_X(\SM)(\ast D)\bigr)^\SA\Bigr) \simeq 
Sol_{\tl{X}}^\rmE(\SM^{\SA})[d_X].\]
If moreover $\SM$ has a normal form along $D$, then 
the holonomic $\SD_X$-module $\DD_X(\SM)(\ast D)$ 
has also a normal form along $D$. 
\end{lemma}

\begin{proof}
Let 
\[0\to\SD_X^{N_k}\to\SD_X^{N_{k-1}}
\to\cdots\to\SD_X^{N_1}\to\SD_X^{N_0}\to\SM\to0\]
be a (local) free resolution of $\SM$.
Set
\[\SL^\bullet := [0\to\SD_X^{N_k}\to\cdots\to\SD_X^{N_0}\to0]\]
so that we have a quasi-isomorphism $\SL^\bullet \simto \SM$.
Hence we obtain an isomorphism
\[\DD_X(\SM)\simeq\SK^\bullet
:= \mathcal{H}om_{\SD_X}(\SL^\bullet, \SD_X)\otimes_{\SO_X}
\Omega_X^{\otimes-1}[d_X].\]
By applying the exact 
functor $( \cdot )( \ast D)= ( \cdot ) \Dotimes\SO_X(\ast D)$ to it, 
we obtain also a quasi-isomorphism
\[\SN := \DD_X(\SM)(\ast D)\simeq\SK^\bullet(\ast D).\]
Obviously we have an isomorphism
\[(\SL^\bullet)^\SA\simeq\SM^\SA.\]
We thus obtain the desired isomorphism
\begin{align*}
(\SM^\SA)^\ast = \rhom_{\SD_{\tl{X}}^\SA}
((\SL^\bullet)^\SA, \SD_{\tl{X}}^\SA)
\otimes_{\varpi^{-1}  \SO_X }\varpi^{-1}
\Omega_X^{\otimes-1}[d_X]
\simeq
(\SK^\bullet)^\SA
\simeq
\SN^\SA.
\end{align*}
The remaining assertion can be shown easily 
by using this isomorphism and \cite[Lemma 6.1.2]{DK16}.
\end{proof}

A ramification of $X$ along $D$ on a neighborhood $U$ 
of $x \in D$ is a finite map $p : X'\to U$
of complex manifolds 
of the form $z' \mapsto 
z=(z_1,z_2, \ldots, z_n)= 
 p(z') = (z'^{m_1}_1,\ldots, z'^{m_r}_r, z'_{r+1},\ldots,z'_n)$ 
for some $(m_1, \ldots, m_r)\in (\ZZ_{>0})^r$, where 
$(z'_1,\ldots, z'_n)$ is a local coordinate system of $X'$ and 
$(z_1, \ldots, z_n)$ is that of 
$U$ such that $D \cap U=\{z_1\cdots z_r=0\}$. 

\begin{definition}\label{def-B}
Let $X$ be a complex manifold and $D \subset X$ a 
normal crossing divisor in it. 
Then we say that a holonomic $\SD_X$-module $\SM$ has a 
quasi-normal form along $D$ if it satisfies 
the conditions (i) and (ii) above,
and if for any $x \in D$ 
there exists a ramification $p : X'\to U$ 
on a neighborhood $U$ of it such that $\bfD p^\ast(\SM|_U)$
has a normal form along $p^{-1}(D\cap U)$.
\end{definition}

Note that $\bfD p^\ast(\SM|_U)$ as well 
as $\bfD p_\ast\bfD p^\ast(\SM|_U)$
is concentrated in degree zero and $\SM|_U$ is a 
direct summand of $\bfD p_\ast\bfD p^\ast(\SM|_U)$. 
The following fundamental result is due to
Kedlaya and Mochizuki.

\begin{theorem}[\cite{Ked10, Ked11, Mochi09, Mochi11}]
For a holonomic $\SD_X$-module $\SM$ and $x\in X$,
there exist an open neighborhood $U$ of $x$, 
a closed hypersurface $Y\subset U$,
a complex manifold $X'$ and
a projective morphism $f : X'\to U$ such that\\
{\rm (i)} $ \rm{sing.supp} (\SM)\cap U\subset Y$,\\
{\rm (ii)} $D:=f^{-1}(Y)$ is a normal crossing divisor in $X'$,\\
{\rm (iii)} $f$ induces an isomorphism $X'\bs D\simto U\bs Y$,\\
{\rm (iv)} $(\bfD f^\ast\SM)(\ast D)$ has a quasi-normal form along $D$.
\end{theorem}
This is a generalization of the classical 
Hukuhara-Levelt-Turrittin theorem to higher dimensions. 

\begin{proposition}\label{prop-1}
Let $X$ be a complex manifold and $D\subset X$
a normal crossing divisor in it.
Assume that a holonomic $\SD$-modules $\SM$ has
a quasi-normal form along $D$ for a ramification map
$f : Y\to X$ and set $D':=f^{-1}(D) \simeq D$.
Denote by $\varpi_X : \tl{X}\to X$ $($resp. $\varpi_Y : \tl{Y}\to Y)$
the real blow-up of $X$ $($resp. $Y)$ along $D$ $($resp. $D')$.
For a point $y_0\in\varpi^{-1}_{Y}(D')$, let $W\subset\tl{Y}$
be its sufficiently small open neighborhood for which
there exists an open subset $U$ of $Y$ containing
$\varpi_Y(W)$ and $\varphi_i\in\Gamma(U; \SO_Y(\ast D'))$
$(1\leq i \leq m)$ such that we have an isomorphism 
\begin{equation*}
\bigl(\bfD f^{\ast}\SM\bigr)^\SA|_W
\simeq
\Bigl(\bigoplus_{i=1}^{m}\bigl(\SE_{U\bs 
D'|U}^{\varphi_i})^\SA\bigr)\Bigr) |_W
\end{equation*}
of $\SD_{\tl{Y}}^\SA$-modules on $W$.
Let $V'\subset Y\bs D'$ be an open sector in $Y$
along $D'$ such that $\var{\varpi^{-1}_{Y}(V')}\subset W$
and set $V=f(V')\subset X\bs D$. 
Finally, for $1\leq i \leq m$ let $\tl{\varphi}_i\in\Gamma(V; \SO_X)$
be a holomorphic function on the sector $V$ along $D$ such that
$\tl{\varphi_i}\circ (f|_{V'})=\varphi_i|_{V^{\prime}}$.
Then we have an isomorphism
\begin{equation*}
\pi^{-1}\CC_V\otimes Sol_X^\rmE(\SM)
\simeq
\bigoplus_{i=1}^m\bigl(\pi^{-1}\CC_V\otimes
\EE_{V |X}^{\Re \tl{\varphi_i} }\bigr).
\end{equation*}
\end{proposition}

\begin{proof}
Let $g : \tl{Y}\to\tl{X}$ be the lift of $f : Y\to X$
i.e. the unique continuous map for which 
we have a commutative diagram
\begin{equation*}
\xymatrix@M=7pt@C=50pt{
\tl{Y}\ar@{->}[r]^-g\ar@{->}[d]_-{\varpi_Y} & 
\tl{X}\ar@{->}[d]^-{\varpi_X} \\
Y\ar@{->}[r]_-f & X.
}
\end{equation*}
For the sufficiently small $W\subset \tl{Y}$
it induces a homeomorphism $g|_W : W\simto g(W)$.
Then by \cite[Theorem 9.1.2(ii) and Corollary 9.2.3]{DK16}
we have an isomorphism
\begin{equation*}
\bfE(g|_W)_\ast(DR_{\tl{Y}}^\rmE((\bfD f^\ast\SM)^\SA)|_{\pi^{-1}(W)}[d_Y])
\simeq
DR_{\tl{X}}^\rmE(\SM^{\SA})|_{\pi^{-1}(g(W))}[d_X]
\eqno(3.1)
\end{equation*}
(in this case we have $d_X=d_Y$ ).
For the Verdier duality functor
$\rmD_X^\rmE : \BEC(\I\CC_X)^{\op}\to\BEC(\I\CC_X)$
we also obtain a chain of isomorphisms
\begin{align*}
\rmD_X^\rmE(\pi^{-1}\CC_V\otimes Sol_X^\rmE(\SM))
&\simeq
\Prihom(\pi^{-1}\CC_V\otimes Sol_X^\rmE(\SM), \omega_X^\rmE)\\
&\simeq
\rihom(\pi^{-1}\CC_V, \Prihom(Sol_X^\rmE(\SM), \omega_X^\rmE))\\
&\simeq
\rihom(\pi^{-1}\CC_V, DR_X^\rmE(\SM)[d_X])\\
&\simeq
\bfE\varpi_{X\ast}\rihom(\pi^{-1}\CC_{\varpi^{-1}_X(V)}, 
DR_{\tl{X}}^\rmE(\SM^\SA)[d_X]),
\end{align*}
where in the third (resp. forth) isomorphism we used
\cite[Corollary 9.4.9]{DK16} (resp. \cite[Corollary 9.2.3]{DK16}). 
Combining it with $(3.1)$, it follows from our
assumption that there exist isomorphisms
\begin{align*}
\rmD_X^\rmE\bigl(\pi^{-1}\CC_V\otimes Sol_X^\rmE(\SM)\bigr)
&\simeq
\bfE\varpi_{X\ast}\bfE g_\ast
\rihom\Big(\pi^{-1}\CC_{\varpi^{-1}_Y(V')}, DR_{\tl{Y}}^\rmE
\big((\bfD f^{\ast}\SM)^\SA\big)[d_Y]\Bigr)\\
&\simeq
\bfE f_\ast \bfE\varpi_{Y\ast}
\rihom\Big(\pi^{-1}
\CC_{\varpi^{-1}_Y(V')}, \bigoplus_{i=1}^m 
DR_{\tl{Y}}^\rmE\bigl((\SE_{U\bs 
D'|U}^{\varphi_i})^\SA\bigr)[d_Y]\Big)\\
&\simeq
\bigoplus_{i=1}^m\bfE f_\ast \rmD_Y^\rmE\bigl(\pi^{-1}
\CC_{V'}\otimes Sol_Y^\rmE(\SE_{U\bs D'|U}^{\varphi_i})\bigr)\\
&\simeq 
\rmD_X^\rmE\Bigl(\bigoplus_{i=1}^m\bfE f_\ast\big(\pi^{-1}
\CC_{V'}\otimes\EE_{U\bs D'|U}^{\Re\varphi_i}\big)\Bigr)\\
&\simeq 
\rmD_X^\rmE\Bigl(\bigoplus_{i=1}^m\bfE f_\ast(\pi^{-1}
\CC_{V'}\otimes \bfE f^{-1}\EE_{V|X}^{\Re\tl{\varphi}_i})\Bigr)\\
&\simeq
\rmD_X^\rmE\Bigl(\bigoplus_{i=1}^m(\pi^{-1}
\CC_{V}\otimes \EE_{V|X}^{\Re\tl{\varphi}_i})\Bigr)\\
\end{align*}
where in the last step we used the projection formula.
By applying the functor $\rmD_X^\rmE$ to the both sides,
we obtain the desired isomorphism.
\end{proof}

\begin{remark}\label{rem-1}
In Proposition \ref{prop-1} we assumed that $\SM\simto\SM(\ast D)$.
However this condition is not really necessary. 
Indeed by Theorem \ref{thm-4} (v), 
for the holonomic  $\SD_X$-module 
$\SN=\SM(\ast D)$ there exists an isomorphism
\begin{equation*}
\pi^{-1}\CC_V\otimes Sol_X^\rmE(\SM)
\simeq
\pi^{-1}\CC_V\otimes Sol_X^\rmE(\SN).
\end{equation*}
Moreover, obviously we have also an isomorphism
$(\bfD f^\ast\SM)^\SA\simeq(\bfD f^\ast\SN)^\SA$.
\end{remark}

By this Remark \ref{rem-1}, Proposition \ref{prop-1} 
and the classical Hukuhara-Levelt-Turrittin theorem 
we obtain the following corollary.

\begin{corollary}\label{cor-1}
Let $X$ be a Riemann surface and $D\subset X$ a point in it.
Let $\SM$ be a holonomic $\SD_X$-module.
Then for any direction
$\theta\in S_DX=\overset{\circ}{T}_DX/\RR_{>0}\simeq S^1$
there exists its sectorial neighborhood $V_\theta\subset X\bs D$
and some Puiseux series $\psi_i\in\Gamma(V_\theta; \SO_X)$ 
($1 \leq i \leq m$) 
for which we have an isomorphism
\begin{equation*}
\pi^{-1}\CC_{V_\theta}\otimes Sol_X^\rmE(\SM)
\simeq
\bigoplus_{i=1}^m(\pi^{-1}\CC_{V_\theta}\otimes
\EE_{V_{\theta}|X}^{\Re \psi_i }).
\end{equation*}
\end{corollary}
Note that this result is stated without proof in 
D'Agnolo-Kashiwara \cite{DK17}.

Proposition \ref{prop-1} can be deduced also from the following theorem.
\begin{theorem}\label{thm-10}
Let $X$ be a complex manifold and $D$ a normal crossing divisor in it.
For $\SM\in\BDC_{\rm hol}(\SD_X)$ and
an open subset $W$ of $\tl{X}$ such that
$W\cap\varpi^{-1}(D)\neq\emptyset$,
we set $\mathscr{K} := {\bf E}i_W^{-1}Sol_{\tl X}^{\rmE}(\M^\SA) 
= Sol_W^{\rmE}(\M^\SA|_W)$,
where $i_W : W\xhookrightarrow{\ \ \ } \tl X$ is the inclusion map.
Then for any sector $V\subset X\setminus D$ along $D$ 
such that $\tl{V} := \var{\varpi^{-1}(V)}\subset W$, 
there exists an isomorphism
\[\pi^{-1}\CC_V\otimes Sol_X^{\rmE}(\M)\simeq
{\bf E}\varpi_{\ast}(\pi^{-1}\CC_{\varpi^{-1}(V)}\otimes
{\bf E}i_{\tl{V}\ast}{\bf E} j^{-1}\mathscr{K})\]
in $\BEC(\I\CC_X)$, where
$j : {\tl V}\xhookrightarrow{\ \ \ } W$ and 
$i_{\tl V} : \tl V \xhookrightarrow{\ \ \ } {\tl X}$
are the inclusion maps.
\end{theorem}

\begin{proof}
First, we shall prove the isomorphism
\[\pi^{-1}\CC_{X\setminus D}\otimes K\simto
\pi^{-1}\CC_{X\setminus D}\otimes
\rihom(\pi^{-1}\CC_{X\setminus D}, K)\]
for an object $K$ of $\BEC(\I\CC_X)$. 
By \cite[Proposition 3.2.9 (iii)]{DK16} 
there exist isomorphisms 
\[\rihom(\pi^{-1}\CC_X, K)\simeq K,\ \ 
\pi^{-1}\CC_{X\setminus D}\otimes\rihom(\pi^{-1}\CC_{D}, K)\simeq 0.\]
Hence by applying the contravariant functor
\[\pi^{-1}\CC_{X\setminus D}\otimes\rihom(\pi^{-1}(\cdot), K)\]
to the distinguished triangle
\[\CC_{X\setminus D} \longrightarrow \CC_X\longrightarrow \CC_{D}\xrightarrow{+1}\]
we obtain the desired isomorphism
\[\pi^{-1}\CC_{X\setminus D}\otimes K\simto
\pi^{-1}\CC_{X\setminus D}\otimes
\rihom(\pi^{-1}\CC_{X\setminus D}, K).\]
On the other hand, for $L\in\BEC(\I\CC_X)$
we have
\begin{align*}
\rihom(\pi^{-1}\CC_{X\setminus D}, 
{\bf E}\varpi_\ast{\bf E}\varpi^!L)
&\simeq
\rihom(\pi^{-1}({\rm R}\varpi_\ast\varpi^{-1}\CC_{X\setminus D}), L)\\
&\simeq
\rihom(\pi^{-1}\CC_{X\setminus D}, L).
\end{align*}
We thus obtain a sequence of isomorphisms
\begin{align*}
\pi^{-1}\CC_V\otimes Sol_X^{\rmE}(\M)
&\simeq
\pi^{-1}\CC_V\otimes(\pi^{-1}\CC_{X\setminus D}\otimes Sol_X^{\rmE}(\M))\\
&\simeq
\pi^{-1}\CC_V\otimes(\pi^{-1}\CC_{X\setminus D}\otimes
\rihom(\pi^{-1}\CC_{X\setminus D}, Sol_X^{\rmE}(\M)))\\
&\simeq
\pi^{-1}\CC_V\otimes(\pi^{-1}\CC_{X\setminus D}\otimes
\rihom(\pi^{-1}\CC_{X\setminus D}, 
{\bf E}\varpi_\ast{\bf E}\varpi^!Sol_X^{\rmE}(\M)))\\
&\simeq
\pi^{-1}\CC_V\otimes\rihom(\pi^{-1}\CC_{X\setminus D}, 
{\bf E}\varpi_\ast{\bf E}\varpi^!Sol_X^{\rmE}(\M))\\
&\simeq
\pi^{-1}\CC_V\otimes{\bf E}\varpi_\ast{\bf E}\varpi^!
\rihom(\pi^{-1}\CC_{X\setminus D}, Sol_X^{\rmE}(\M)).\\
\end{align*}
Therefore by using the isomorphism
\[Sol_{\tl X}^{\rmE}(\M^\SA)\simeq
{\bf E}\varpi^!\rihom(\pi^{-1}\CC_{X\setminus D}, Sol_X^{\rmE}(\M))\]
(see \cite[p191]{DK16})
we obtain isomorphisms
\begin{align*}
\pi^{-1}\CC_V\otimes Sol_X^{\rmE}(\M)
&\simeq
\pi^{-1}\CC_V\otimes{\bf E}\varpi_\ast Sol_{\tl X}^{\rmE}(\M^\SA)\\
&\simeq
{\bf E}\varpi_{\ast}(\pi^{-1}\CC_{\varpi^{-1}(V)}\otimes Sol_{\tl X}^{\rmE}(\M^\SA))\\
&\simeq
{\bf E}\varpi_{\ast}(\pi^{-1}\CC_{\varpi^{-1}(V)}\otimes
(\pi^{-1}\CC_{\tl V}\otimes Sol_{\tl X}^{\rmE}(\M^\SA)))\\
&\simeq
{\bf E}\varpi_{\ast}(\pi^{-1}\CC_{\varpi^{-1}(V)}\otimes
{\bf E}i_{\tl{V}\ast}{\bf E}i_{\tl{V}}^{-1}Sol_{\tl X}^{\rmE}(\M^{\SA}))\\
&\simeq
{\bf E}\varpi_{\ast}(\pi^{-1}\CC_{\varpi^{-1}(V)}\otimes
{\bf E}i_{\tl{V}\ast}{\bf E} j^{-1}\mathscr{K}).
\end{align*}
\end{proof} 
We can prove the converse of this theorem see Theorem \ref{thm-9}.

\begin{theorem}\label{prop-3}
Let $X$ be a complex manifold,
$D\subset X$  a normal crossing divisor in it
and $\tl{X}$ the real blow-up of $X$ along $D$.
Let ${\delta} : \tl{X}\hookrightarrow\tl{X}\times\tl{X}$ 
be the diagonal map.
Then for $\SM\in\BDC_{\rm hol}(\SD_X)$ 
we have isomorphisms
\begin{align*}
\SM^\SA &\simto\rhom^\rmE\Big(Sol_{\tl{X}}^\rmE
\big(\SM^\SA\big), \SO_{\tl{X}}^\rmE\Big),\\
\SM^\SA &\simto\rhom^\rmE\Big(\CC_{\tl{X}}^\rmE, 
\bfE{\delta}^!\big(DR_{\tl{X}}^\rmE(\SM^\SA)\Pboxtimes
\SO_{\tl{X}}^\rmE\big)\Big)[d_X]
\end{align*}
in $\BEC(\I\SD_{\tl{X}}^\SA)$.
\end{theorem}

\begin{proof}
First, we shall prove the isomorphism
\[\SM^\SA \simto
\rhom^\rmE\Big(Sol_{\tl{X}}^\rmE\big(\SM^\SA\big), \SO_{\tl{X}}^\rmE\Big).\]
Recall that we have already the canonical morphism
\[\SM^\SA \to
\rhom^\rmE\Big(Sol_{\tl{X}}^\rmE\big(\SM^\SA\big), \SO_{\tl{X}}^\rmE\Big).\]
in \cite[proof of Lemma 9.6.6]{DK16}.
We shall prove that this morphism is an isomorphism.
Let $\tl{\varpi} : \tl{X}\times
\RR_\infty\to X\times\RR_\infty$
be the natural morphism of bordered spaces.
We may assume that $\M\simeq\M(\ast D)$.
Then by the isomorphisms 
\begin{align*}
\bfE\varpi^{-1}Sol_X^\rmE(\SM)
&\simeq
\tl{\varpi}^{-1}\pi^{-1}\CC_{X\bs D}
\otimes Sol_{\tl{X}}^\rmE(\SM^\SA),\\
\SO_{\tl{X}}^\rmE &\simeq \bfE\varpi^{!}
\rihom(\pi^{-1}\CC_{X\bs D}, \SO_X^\rmE)
\end{align*}
(see \cite[(9.6.6)]{DK16}),
we obtain a sequence of isomorphisms
\begin{align*}
\rhom^\rmE\Big(Sol_{\tl{X}}^\rmE
\big(\SM^\SA\big), \SO_{\tl{X}}^\rmE\Big)
&\simeq
\rhom^\rmE\Big(Sol_{\tl{X}}^\rmE\big(\SM^\SA\big),
 \rihom\big(\tl{\varpi}^{-1}\pi^{-1}\CC_{X\bs D},
  \SO_{\tl{X}}^\rmE\big)\Big)\\
  &\simeq
\rhom^\rmE\Big(\tl{\varpi}^{-1}\pi^{-1}\CC_{X\bs D}\otimes
Sol_{\tl{X}}^\rmE\big(\SM^\SA\big),
  \SO_{\tl{X}}^\rmE\Big)\\
&\simeq
\rhom^\rmE\big(\bfE\varpi^{-1}Sol_X^\rmE(\SM),
 \SO_{\tl{X}}^\rmE\big)\\
&\simeq
\rhom^\rmE\big(\bfE\varpi^{-1}Sol_X^\rmE(\SM), \bfE\varpi^{!}
\rihom(\pi^{-1}\CC_{X\bs D}, \SO_X^\rmE)\big)\\
&\simeq
\alpha_{\tl{X}}\rmR\varpi^{!}\rihom^\rmE\big(Sol_X^\rmE(\SM), 
\rihom(\pi^{-1}\CC_{X\bs D}, \SO_X^\rmE)\big)\\
&\simeq
\alpha_{\tl{X}}\rmR\varpi^{!}\rihom\Big( \CC_{X\bs D}, 
\rihom^\rmE\big(Sol_X^\rmE(\SM), \SO_X^\rmE\big)\Big).
\end{align*}
Moreover, by Theorem \ref{thm-5}
we have an isomorphism
\[
\SM\Lotimes{\SO_X}\SO_X^\rmt
\simto
\rihom^\rmE\big(Sol_X^\rmE(\SM), \SO_X^\rmE\big).
\]
We thus obtain isomorphisms
\begin{align*}
\rhom^\rmE\Big(Sol_{\tl{X}}^\rmE
\big(\SM^\SA\big), \SO_{\tl{X}}^\rmE\Big)
&\simeq
\alpha_{\tl{X}}\rmR\varpi^{!}\rihom( \CC_{X\bs D}, 
\SM\Lotimes{\SO_X}\SO_X^\rmt)\\
&\simeq
\alpha_{\tl{X}}\rmR\varpi^{!}(\rihom( \CC_{X\bs D}, 
\SO_X^\rmt)\Lotimes{\SO_X}\SM)\\
&\simeq
\alpha_{\tl{X}}(\SO_{\tl{X}}^\rmt
\Lotimes{\varpi^{-1}\SO_X}\varpi^{-1}\SM)\\
&\simeq
\alpha_{\tl{X}}\SO_{\tl{X}}^\rmt
\Lotimes{\varpi^{-1}\SO_X}\varpi^{-1}\SM\\
&\simeq
\SA_{\tl{X}}\Lotimes{\varpi^{-1}\SO_X}\varpi^{-1}\SM\\
&\simeq
\SM^\SA
\end{align*}
where in the third isomorphism we used \cite[Theorem 7.2.7]{DK16}.

Next, we shall prove the isomorphism
\[\rhom^\rmE\Big(Sol_{\tl{X}}^\rmE\big(\SM^\SA\big), \SO_{\tl{X}}^\rmE\Big)
\simeq
\rhom^\rmE\Big(\CC_{\tl{X}}^\rmE, \bfE{\delta}^!\big(
DR_{\tl{X}}^\rmE(\SM^\SA)\Pboxtimes\SO_{\tl{X}}^\rmE\big)\Big)[d_X].
\]
By \cite[Proposition 4.9.23]{DK16}
we have an isomorphism
\begin{align*}
\rhom^{\rmE}\Big(\CC_{\tl{X}}^{\rmE},{\bf E}{\delta}^!
\big(DR_{\tl{X}}^{\rmE}(\SM^\SA)\Pboxtimes
\SO_{\tl{X}}^{\rmE}\big)\Big)[d_X]
\simeq
\rhom^\rmE\Big(D_{\tl{X}}^\rmE\big(DR_{\tl{X}}^\rmE(\SM^\SA)\big), 
\SO_{\tl{X}}^\rmE\Big)[d_X].
\end{align*}
Since there exists an isomorphism
\begin{align*}
D_{\tl{X}}^\rmE\big(DR_{\tl{X}}^\rmE(\SM^\SA)\big)
\simeq
\bfE\varpi^{-1}Sol_X^\rmE(\SM)
\simeq
\tl{\varpi}^{-1}\pi^{-1}\CC_{X\bs D}
\otimes Sol_{\tl{X}}^\rmE(\SM^\SA)[d_X]
\end{align*}
we obtain a sequence of isomorphisms 
\begin{align*}
&\rhom^\rmE\Big(D_{\tl{X}}^\rmE\big(DR_{\tl{X}}^\rmE(\SM^\SA)\big), 
\SO_{\tl{X}}^\rmE\Big)[d_X]\\
&\simeq
\rhom^\rmE\Big(\tl{\varpi}^{-1}\pi^{-1}\CC_{X\bs D}
\otimes Sol_{\tl{X}}^\rmE(\SM^\SA), \SO_{\tl{X}}^\rmE\Big)\\
&\simeq
\rhom^\rmE\Big(Sol_{\tl{X}}^\rmE(\SM^\SA), 
\rihom\big(\tl{\varpi}^{-1}\pi^{-1}\CC_{X\bs D}, 
\SO_{\tl{X}}^\rmE\big)\Big)\\
&\simeq
\rhom^\rmE\big(Sol_{\tl{X}}^\rmE(\SM^\SA),\SO_{\tl{X}}^\rmE\big).
\end{align*}
\end{proof}

The following proposition is useful to know 
the exponential types of holonomic $\SD$-modules from 
their enhanced solution complexes. 

\begin{proposition}\label{prop-2}
Let $X$ be a complex manifold, $D\subset X$
a normal crossing divisor in it and
$\SM_i\ (i=1, 2)$ holonomic $\D_X$-modules.
Let $V\subset X\bs D$ be an open sector in $X$ along $D$
and assume that we have an isomorphism
\begin{equation*}
\pi^{-1}\CC_V\otimes Sol_X^\rmE(\SM_1)
\simeq
\pi^{-1}\CC_V\otimes Sol_X^\rmE(\SM_2).
\end{equation*} 
Let $W\subset \tl{X}$ be an open subset of 
the real blow-up $\tl{X}$
such that $W\cap\varpi^{-1}(D)\neq\emptyset$ and 
$\var{W}\subset{\rm Int}\Big(\var{\varpi^{-1}(V)}\Big)$ .
Then we have an isomorphism
\begin{equation*}
\SM_1^\SA|_W
\simeq
\SM_2^\SA|_W
\end{equation*}
 of $\SD_{\tl{X}}^\SA$-modules on W.
\end{proposition}

\begin{proof}
Let $W\subset\tl{X}$ be an open subset of $\tl{X}$ 
such that $W\cap\varpi^{-1}(D)\neq\emptyset$ and 
$\var{W}\subset{\rm Int}\Big(\var{\varpi^{-1}(V)}\Big)$. 
Let $\delta : \tl{X}\hookrightarrow\tl{X}\times\tl{X}$ 
be the diagonal map. 
Then by Theorem \ref{prop-3} there exist isomorphisms 
\begin{equation*}
\SM_i^\SA \simeq \rhom^\rmE\Big(\CC_{\tl{X}}^\rmE, 
\bfE \delta^! \big(DR_{\tl{X}}^\rmE(\SM_i^\SA)
\Pboxtimes\SO_{\tl{X}}^\rmE\big)\Big)[d_X] 
\hspace{20pt}(i=1, 2).
\end{equation*}
Hence it suffices to show that for any $G\in\BEC(\I\CC_W)$
we have an isomorphism
\begin{equation*}
{\rm Hom}_{\BEC(\I\CC_W)}(G, 
DR_{\tl{X}}^\rmE(\SM_1^\SA)|_{\pi^{-1}(W)})
\simeq 
{\rm Hom}_{\BEC(\I\CC_W)}(G, 
DR_{\tl{X}}^\rmE(\SM_2^\SA)|_{\pi^{-1}(W)}).
\end{equation*}
Let $j : W\hookrightarrow\tl{X}$ be the inclusion map.
Then for $i=1, 2$ there exist isomorphisms
\begin{align*}
{\rm Hom}_{\BEC(\I\CC_W)}(G, 
DR_{\tl{X}}^\rmE(\SM_i^\SA)|_{\pi^{-1}(W)})
&\simeq
{\rm Hom}_{\BEC(\I\CC_W)}(G, 
\bfE j^!DR_{\tl{X}}^\rmE(\SM_i^\SA))\\
&\simeq
{\rm Hom}_{\BEC(\I\CC_{\tl{X}})}(\bfE j_{!!}G, 
DR_{\tl{X}}^\rmE(\SM_i^\SA))\\
&\simeq
{\rm Hom}_{\BEC(\I\CC_{\tl{X}})}(\bfE j_{!!}G, 
\bfE\varpi^!DR_X^\rmE(\SM_i(\ast D)))\\
&\simeq
{\rm Hom}_{\BEC(\I\CC_X)}(
\bfE\varpi_{!!}\bfE j_{!!}G, DR_X^\rmE(\SM_i(\ast D))).
\end{align*}
Set $G' := \bfE\varpi_{!!}\bfE j_{!!}G\in\BEC(\I\CC_X)$.
Then for $i=1, 2$ we obtain isomorphisms
\begin{align*}
{\rm Hom}_{\BEC(\I\CC_W)}(G, 
DR_{\tl{X}}^\rmE(\SM_i^\SA)|_{\pi^{-1}W})
&\simeq
{\rm Hom}_{\BEC(\I\CC_X)}(G', 
\rihom(\pi^{-1}\CC_{X\bs D}, DR_X^\rmE(\SM_i))) \\
&\simeq
{\rm Hom}_{\BEC(\I\CC_X)}(\pi^{-1}
\CC_{X\bs D}\otimes G', DR_X^\rmE(\SM_i))\\
&\simeq
{\rm Hom}_{\BEC(\I\CC_X)}(\pi^{-1}
\CC_{V}\otimes G', DR_X^\rmE(\SM_i))\\
&\simeq
{\rm H^0}\rmR\Gamma(X; \rhom^\rmE(\pi^{-1}
\CC_{V}\otimes G', DR_X^\rmE(\SM_i))), 
\end{align*}
where in the third isomorphism we used 
$G' \simeq \pi^{-1} \CC_{\var{\varpi (W)}} 
\otimes G'$ and $(X \bs D) \cap \var{\varpi (W)}
= V \cap \var{\varpi (W)}$
which follows from our assumption $\var{W}\subset{\rm Int}\Big(\var{\varpi^{-1}(V)}\Big)$. 
Since we have 
\begin{align*}
\rhom^\rmE(\pi^{-1}\CC_{V}\otimes G', DR_X^\rmE(\SM_i))
&\simeq
\alpha_X\rmR\pi_\ast\rihom(\pi^{-1}\CC_{V}\otimes G', DR_X^\rmE(\SM_i))\\
&\simeq
\alpha_X\rmR\pi_\ast\rihom(G', \rihom(\pi^{-1}\CC_{V}, DR_X^\rmE(\SM_i))),
\end{align*}
it remains for us to prove the isomorphism
\[\rihom(\pi^{-1}\CC_{V}, DR_X^\rmE(\SM_1))
\simeq\rihom(\pi^{-1}\CC_{V}, DR_X^\rmE(\SM_2)).\]
But this follows immediately by applying the 
Verdier duality functor $\rmD_X^\rmE$ to the isomorphism 
\[\pi^{-1}\CC_V\otimes Sol_X^\rmE(\SM_1)
\simeq
\pi^{-1}\CC_V\otimes Sol_X^\rmE(\SM_2)\]
(see the proof of Proposition \ref{prop-1}).
This completes the proof.
\end{proof}

By Proposition \ref{prop-2} (and the proof of Corollary \ref{cor-1})
we obtain the following result.

\begin{corollary}\label{cor-2}
Let $X$ be a Riemann surface and $D\subset X$ a point in it.
Let $\SM_1$ and $\SM_2$ be holonomic $\SD_X$-modules. 
Assume that for a point $\theta\in S_DX\simeq\varpi^{-1}(D)\simeq S^1$
there exists its sectorial neighborhood $V_\theta\subset X\bs D$
such that we have an isomorphism
\begin{equation*}
\pi^{-1}\CC_{V_\theta}\otimes Sol_X^\rmE(\SM_1)
\simeq
\pi^{-1}\CC_{V_\theta}\otimes Sol_X^\rmE(\SM_2).
\end{equation*}
Then there exists an open neighborhood $W$ of $\theta$ in 
the real blow-up $\tl{X}$ on which we have an isomorphism
\begin{equation*}
\SM_1^\SA|_W
\simeq
\SM_2^\SA|_W
\end{equation*}
of $\SD_{\tl{X}}^\SA$-modules on $W$.
\end{corollary}

We can prove Proposition \ref{prop-2} also using the following theorem.
\begin{theorem}\label{thm-9}
Let $X$ be a complex manifold and $D$ a normal crossing divisor in it.
For $\SM\in\BDC_{\rm hol}(\SD_X)$ and a 
sector $V\subset X\setminus D$ along $D$
we set $K:=\pi^{-1}\CC_V\otimes Sol_X^{\rmE}(\SM)$.
Then for any open subset $W$ of $\tl{X}$ such that
$W\cap\varpi^{-1}(D)\neq\emptyset, \var{W}\subset \Int\Big(
\var{\varpi^{-1}(V)}\Big)$,
there exists  an isomorphism
\[\SM^\SA|_W\simeq\rhom^\rmE\Big(\big({\bfE}\varpi^!
\rihom(\pi^{-1}\CC_{X\bs D}, K)\big)|_W, \SO_{\tl{X}}^\rmE|_W\Big)\]
in $\BDC(\SD_{\tl{X}}^\SA|_W)$.
\end{theorem}

\begin{proof}
By $\SM^\SA|_W\simeq 
\rhom^\rmE(Sol_{\tl{X}}^\rmE(\SM^\SA)|_W, \SO_{\tl{X}}^\rmE|_W)$
(see Theorem \ref{prop-3}), 
it is enough to show
\[
Sol_{\tl{X}}^\rmE(\SM^\SA)|_W
\simeq
\Big({\bfE}\varpi^!\rihom(\pi^{-1}\CC_{X\bs D}, K)\Big)|_W.
\]
We consider the following diagram
\[\xymatrix@C=40pt@M=11pt{
\tl{X}\ar@{->}[r]^-\varpi & X\\
W\ar@{->}[r]_-{\varpi|_W}\ar@{^{(}->}[u]^-{i_W} & \varpi(W).
\ar@{^{(}->}[u]_-{j}
}\]
Let $\tl{i}_W : W\times
\RR_\infty\to \tl{X}\times\RR_\infty$ 
and $\tl{\varpi} :  \tl{X} \times
\RR_\infty\to X \times\RR_\infty$ 
be the natural morphisms of bordered spaces.
Then we obtain a sequence of isomorphisms:
\begin{align*}
Sol_{\tl{X}}^\rmE(\SM^\SA)|_W
&\simeq
{\bfE}i_W^{-1}{\bfE}\varpi^!\rihom\big(\pi^{-1}\CC_{X\bs D}, 
Sol_X^\rmE(\SM)\big)\\
&\simeq
{\bfE}i_W^{-1}\rihom\big(\tl{\varpi}^{-1}\pi^{-1}\CC_{X\bs D}, 
{\bfE}\varpi^!Sol_X^\rmE(\SM)\big)\\
&\simeq
{\bfE}i_W^{-1}\rihom\big(\tl{\varpi}^{-1}\pi^{-1}\CC_{X\bs D}, 
{\bfE}\varpi^{-1}Sol_X^\rmE(\SM)\big)\\
&\simeq
\rihom\big(\ \tl{i}_W^{-1}\tl{\varpi}^{-1}\pi^{-1}\CC_{X\bs D}, 
{\bfE}i_W^{-1}{\bfE}\varpi^{-1}Sol_X^\rmE(\SM)\big),
\end{align*}
where the third isomorphism follows from the fact that 
$\varpi$ is an isomorphism over $X\bs D$.
Since we may assume $\SM\simto\SM(\ast D)$, 
there exists a sequence of isomorphisms:
\begin{align*}
{\bfE}i_W^{-1}{\bfE}\varpi^{-1}Sol_X^\rmE(\SM)
&\simeq
{\bfE}i_W^{-1}{\bfE}\varpi^{-1}\big(\pi^{-1}\CC_{X\bs D}
\otimes Sol_X^\rmE(\SM)\big)\\
&\simeq
{\bfE}(\varpi|_W)^{-1}{\bfE}j^{-1}\big(\pi^{-1}\CC_{X\bs D}
\otimes Sol_X^\rmE(\SM)\big)\\
&\simeq
{\bfE}(\varpi|_W)^{-1}\big(\pi^{-1}\CC_{j^{-1}(X\bs D)}
\otimes {\bfE}j^{-1}Sol_X^\rmE(\SM)\big)\\
&\simeq
{\bfE}(\varpi|_W)^{-1}\big(\pi^{-1}\CC_{j^{-1}(V)}
\otimes {\bfE}j^{-1}Sol_X^\rmE(\SM)\big)\\
&\simeq
{\bfE}(\varpi|_W)^{-1}{\bfE}j^{-1}\big(\pi^{-1}\CC_V
\otimes Sol_X^\rmE(\SM)\big)\\
&\simeq
{\bfE}i_W^{-1}{\bfE}\varpi^{-1}K.
\end{align*}
Therefore we obtain isomorphisms
\begin{align*}
Sol_{\tl{X}}^\rmE(\SM^\SA)|_W
&\simeq
\rihom\big(\ \tl{i}_W^{-1}\tl{\varpi}^{-1}\pi^{-1}\CC_{X\bs D}, 
{\bfE}i_W^{-1}{\bfE}\varpi^{-1}K\big)\\
&\simeq
\Big({\bfE}\varpi^!\rihom\big(\pi^{-1}\CC_{X\bs D},
 K\big)\Big)|_W.
\end{align*}
\end{proof}

From now, until the end of this section, let $X$
be a smooth algebraic variety and $Z\subset X$ a subvariety in it.
Set $U := X\bs Z$ and let $X^\an, Z^\an, U^\an$
be the underlying complex analytic spaces of $X, Z, U$ respectively.
If there is no risk of confusion, we sometimes denote them
simply by $X, Z, U$ for short.
Let
\[Z\xhookrightarrow{\ i_Z\ } X\xhookleftarrow{\ i_U\ }U\]
be the inclusion maps.
We denote the corresponding morphisms of complex analytic spaces
by the same symbols $i_Z$ and $i_U$.
For an algebraic coherent $\SD_X$-module
$\SM\in\Modcoh(\SD_X)$ on $X$, by using
its analytification $\SM^\an=\SO^\an\otimes_{\SO_X}\SM
\in\Modcoh(\SD_{X^\an})$ we set
$Sol_X(\SM):=Sol_{X^\an}(\SM^\an)\in\BDC(\CC_{X^\an})$.
We define $Sol^{\rmt}_X(\SM)\in{\BDC}({\I}{\CC}_{X^{\an}}),
Sol_X^\rmE(\SM)\in{\BEC}(\I\CC_{X^{\an}})$, etc. similarly.
Recall that there exists an isomorphism
\[Sol_X(\SM)|_Z:=i_Z^{-1}Sol_X(\SM)
\simeq i_Z^{-1}\alpha_Xi_0^!\bfR^\rmE Sol_X^\rmE(\SM).\]

\begin{lemma}\label{llema} 
For $F\in\BEC(\I\CC_{X^\an})$ we have an isomorphism
\[i_Z^{-1}\alpha_Xi_0^!\bfR^\rmE(\bfE i_{Z\ast}\bfE i_Z^{-1}F)
\simeq\alpha_Z\rmR\pi_\ast\rihom
(\CC_{\{t\geq0\}}\oplus\CC_{\{t\leq0\}}, \bfE i_Z^{-1}F).\]
\end{lemma}

\begin{proof}
Let $\sigma : X^\an\times\RR_\infty^2\to X^\an\times\RR_\infty$
be a morphism of bordered spaces induced  by the map
$X^\an\times\RR^2\ni(x, t_1, t_2)\mapsto (x, t_2-t_1)\in X^\an\times\RR$,
We define also a morphism $j_0 : X^\an\times\RR_\infty\to 
X^\an\times\RR_\infty^2$
of bordered spaces by the map $X^\an\times\RR\ni(x, t)
\mapsto(x, 0, t)\in X^\an\times\RR^2$.
Then we have isomorphisms
\begin{align*}
&i_Z^{-1}\alpha_Xi_0^!\bfR^\rmE(\bfE i_{Z\ast}\bfE i_Z^{-1}F)\\
&\simeq
i_Z^{-1}\alpha_Xi_0^!\rmR p_{1\ast}\rihom
(\sigma^{-1}(\CC_{\{t\geq0\}}\oplus\CC_{\{t\leq0\}}), 
p_2^!\bfE i_{Z\ast}\bfE i_Z^{-1}F)\\
&\simeq
i_Z^{-1}\alpha_X\rmR\pi_\ast j_0^!\rihom
(\sigma^{-1}(\CC_{\{t\geq0\}}\oplus\CC_{\{t\leq0\}}), 
p_2^!\bfE i_{Z\ast}\bfE i_Z^{-1}F)\\
&\simeq
i_Z^{-1}\alpha_X\rmR\pi_\ast\rihom
(\CC_{\{t\geq0\}}\oplus\CC_{\{t\leq0\}}, \bfE i_{Z\ast}\bfE i_Z^{-1}F)\\
&\simeq
i_Z^{-1}\alpha_X\rmR\pi_\ast\tl{i}_{Z\ast}\rihom
(\CC_{\{t\geq0\}}\oplus\CC_{\{t\leq0\}},\bfE i_Z^{-1}F)\\
&\simeq
i_Z^{-1}\alpha_Xi_{Z\ast} \rmR\pi_\ast\rihom
(\CC_{\{t\geq0\}}\oplus\CC_{\{t\leq0\}},\bfE i_Z^{-1}F)\\
&\simeq
i_Z^{-1}i_{Z\ast}\alpha_Z\rmR\pi_\ast\rihom
(\CC_{\{t\geq0\}}\oplus\CC_{\{t\leq0\}},\bfE i_Z^{-1}F)\\
&\simeq
\alpha_Z\rmR\pi_\ast\rihom
(\CC_{\{t\geq0\}}\oplus\CC_{\{t\leq0\}},\bfE i_Z^{-1}F).
\end{align*}
\end{proof}

Applying this lemma to $F=Sol_X^\rmE(\SM)\in\BEC(\I\CC_{X^\an})$
and the distinguished triangle
\begin{align*}
\pi^{-1}\CC_{X^\an\bs Z^\an}\otimes F \longrightarrow F \longrightarrow
\pi^{-1}\CC_{Z^\an}\otimes F\simeq \bfE i_{Z\ast}\bfE
 i_Z^{-1}F\overset{+1}{\longrightarrow}
\end{align*}
associated to it, we obtain a distinguished triangle
\begin{align*}
&i_Z^{-1}\alpha_Xi_0^!\bfR^\rmE
(\pi^{-1}\CC_{X^\an\bs Z^\an}\otimes Sol_X^\rmE(\SM))
\longrightarrow Sol_X(\SM)|_Z
\longrightarrow\\
&\alpha_Z\rmR\pi_\ast\rihom
(\CC_{\{t\geq0\}}\oplus\CC_{\{t\leq0\}},\bfE i_Z^{-1}Sol_X^\rmE(\SM))
\overset{+1}{\longrightarrow}.
\end{align*}

From now assume moreover that $Z$ is
a smooth hypersurface $D$ in $X$
and that $\SM\in\Modhol(\SD_X)$.
Then by Theorem\ref{thm-4} (v)
we have isomorphisms
\begin{align*}
i_D^{-1}\alpha_Xi_0^!\bfR^\rmE
(\pi^{-1}\CC_{X^\an\bs D^\an}\otimes Sol_X^\rmE(\SM))
&\simeq i_D^{-1}\alpha_Xi_0^!\bfR^\rmE Sol_X^\rmE(\SM(\ast D))\\
&\simeq Sol_X(\SM(\ast D))|_D.
\end{align*}
We thus obtain a distinguished triangle
\begin{align*}
&Sol_X(\SM(\ast D))|_D\longrightarrow
Sol_X(\SM)|_D\longrightarrow\\
&\alpha_D\rmR\pi_\ast\rihom
(\CC_{\{t\geq0\}}\oplus\CC_{\{t\leq0\}},\bfE i_D^{-1}Sol_X^\rmE(\SM))
\overset{+1}{\longrightarrow}.
\end{align*}
Since we have a distinguished triangle
\[\rmR\Gamma_D(\SM)\longrightarrow\SM
\longrightarrow i_{U\ast} i_U^{-1}\SM 
\simeq \SM(\ast D) 
\overset{+1}{\longrightarrow}\]
for the algebraic $\SD_X$-module 
$\rmR\Gamma_D(\SM) = \bfD i_{D\ast}\bfD i_D^\ast\SM[-1]$ 
(see \cite[Proposition 1.7.1 (iii)]{HTT08}),
in this case there exist also isomorphisms
\begin{align*}
Sol_X(\rmR\Gamma_D(\SM))|_D
&\simeq
Sol_X(\bfD i_{D\ast}\bfD i_D^\ast\SM[-1])|_D\\
&\simeq
i_D^{-1}\alpha_Xi_0^!\bfR^\rmE \big(
Sol_X^\rmE(\bfD i_{D\ast}\bfD i_D^\ast\SM)[1]\big)\\
&\simeq
i_D^{-1}\alpha_Xi_0^!\bfR^\rmE \big(\bfE i_{D\ast}
\bfE i_D^{-1}Sol_X^\rmE(\SM)\big)\\
&\simeq
\alpha_D\rmR\pi_\ast\rihom
(\CC_{\{t\geq0\}}\oplus\CC_{\{t\leq0\}},\bfE i_D^{-1}Sol_X^\rmE(\SM)).
\end{align*}
The following Proposition is useful to calculate
\[Sol_X(\SM(\ast D))|_D
\simeq
i_D^{-1}\alpha_XSol_X^\rmt(\SM(\ast D)).\]

\begin{proposition}\label{prop-5}
Assume that $X$ is a Riemann surface and $D$ is a point in it.
For a meromorphic function $\varphi\in\SO_X(\ast D)$
having a pole of order $k>0$ in $D\in X$,
we consider $\SE_{X\bs D | X}^{\varphi}
\in\Mod_{\rm hol}(\SD_{X^\an})$.
Let $u\in\CC$ be a local coordinate of $X$ at $D$
such that $D=\{u=0\}$ and for $a\gg0$
set
\[Q_a := \overline{\{u\in X\ |\ u\neq0, 
{\Re}(\varphi(u))\geq a\}}\subset X.\]
Then we have isomorphisms
\[
H^jSol_X^\rmt\big(\SE_{X\bs D | X}^{\varphi}\big)
\simeq
\begin{cases}
``\underset{a\to +\infty}{\varinjlim}"\ \CC_{X\bs Q_a} & (j=0)\\
\\
\hspace{30pt} \CC_D^{\oplus k} & (j=1)\\
\\
\hspace{35pt}0 & (\mbox{\rm otherwise}).
\end{cases}
\]
Moreover we have 
\[
\dim H^jSol_X \big(\SE_{X\bs D | X}^{\varphi}\big)_D
=
\begin{cases}
k & (j=1)\\
\\
0 & (\mbox{\rm otherwise}).
\end{cases}
\]
\end{proposition}

\begin{proof}
We set $\SN:=\SE_{X\bs D | X}^{\varphi}\in\Modhol(\SD_{X^\an})$.
Recall that by (the proof of) Lemma 
\ref{llema} we have isomorphisms
\begin{align*}
Sol_X^\rmt(\SN) &\simeq
i_0^!\bfR^\rmE Sol_X^\rmE(\SN)\\
&\simeq
\rmR\pi_\ast\rihom
(\CC_{\{t\geq0\}}\oplus\CC_{\{t\leq0\}}, Sol_X^\rmE(\SN)).
\end{align*}
Moreover by Theorem \ref{thm-4} (vi)
there exists also an isomorphism
\[Sol_X^\rmE(\SN)
\simeq
``\underset{a\to +\infty}{\varinjlim}"
\ \CC_{\{t\geq-{\Re}\varphi+a\}}.
\]
For any sufficiently large $a\gg0$ we can easily show that
\begin{align*}
\rhom(\CC_{\{t\geq0\}}, \CC_{\{t\geq-{\Re}\varphi+a\}})
\simeq
\rmR\Gamma_{\{t\geq0\}}\CC_{\{t\geq-{\Re}\varphi+a\}}
\simeq
 \CC_{\{t>0,\ t\geq-{\Re}\varphi+a\}}.
\end{align*}
Similarly for $a\gg0$ we have
\[ \rhom(\CC_{\{t \leq0 \}}, \CC_{\{t\geq-{\Re}\varphi+a\}})
\simeq
 \CC_{\{t<0,\ t\geq-{\Re}\varphi+a\}}. \]
Let $\var{\pi} : X\times\var{\RR}\to X$ be the projection
and $j : X\times \RR\xhookrightarrow{\ \ \ }
 X\times \var{\RR}$ the inclusion map. 
Then for $a\gg0$ it is easy to see that 
\begin{align*}
& \rmR\pi_\ast\rihom
( \CC_{\{t\leq0\}},  \CC_{\{t\geq-{\Re}\varphi+a\}} )
\\
&\simeq
\rmR\var{\pi}_\ast\rhom_{\CC_{X\times\var{\RR}}}(
\CC_{X\times \RR}, \CC_{\{t<0,\ t\geq-{\Re}\varphi+a\}})
\\
&\simeq
\rmR\var{\pi}_\ast \rmR j_\ast 
\CC_{\{t<0,\ t\geq-{\Re}\varphi+a\}} \simeq 0. 
\end{align*}
We thus obtain an isomorphism 
\[ \rmR\pi_\ast\rihom
( \CC_{\{t\leq0\}},  Sol_X^\rmE(\SN) ) \simeq 0. \] 
For $a\gg0$ let us calculate 
\[ \rmR\pi_\ast\rihom
( \CC_{\{t\geq0\}},  \CC_{\{t\geq-{\Re}\varphi+a\}} ) 
\simeq \rmR\var{\pi}_\ast \rmR j_\ast 
\CC_{\{t>0,\ t\geq-{\Re}\varphi+a\}}. \]
The stalk of this complex at the point $D\in X$ is isomorphic to 
\[\rmR\Gamma(D\times\var{\RR}; 
\rmR j_\ast\CC_{\{t>0,\ t\geq-{\Re}\varphi+a\}}).\]
We also see that
\[\big(\rmR j_\ast\CC_{\{t>0,\ t\geq-{\Re}\varphi+a\}}\big)_{(D, +\infty)}
\simeq\CC^k[-1].\]
Indeed, for $b \gg 0$ let us set 
\[R_b := \{u\in X\ |\ u\neq0, 
{\Re}(\varphi(u))\geq -b \} \subset X.\]
Then for a sufficiently small open 
neighborhood $U$ of the point $(D, + \infty )$ 
in $X \times \var{\RR}$ the set 
\[ U \cap \{ ( u, t) \in X \times \var{\RR} \ | \ 
t>0,\ u \not= 0, \ t\geq-{\Re}\varphi ( u ) +a\} \]
is homotopic to $R_b$. 
This implies that the stalk at $D \in X$ is isomorphic to 
$\CC^k[-1]$. Moreover its stalk at a point $P \in X \setminus D$ 
is isomorphic to $\CC$ (resp. $0$) if 
$P \in X \setminus Q_a$ (resp. $P \in Q_a$). For $j \in \ZZ$ 
we thus obtain an isomorphism 
\[
H^j \rmR\pi_\ast\rihom
( \CC_{\{t\geq0\}},  \CC_{\{t\geq-{\Re}\varphi+a\}} ) 
\simeq
\begin{cases}
 \CC_{X\bs Q_a} & (j=0)\\
\\
\CC_D^{\oplus k} & (j=1)\\
\\
0 & (\mbox{\rm otherwise}).
\end{cases}
\]
The remaining assertion follows from 
the isomorphism $\alpha_X Sol_X^\rmt 
(\SE_{X\bs D | X}^{\varphi}) 
\simeq Sol_X (\SE_{X\bs D | X}^{\varphi})$. 
\end{proof}

Note that a special case of this proposition was 
proved in Kashiwara-Schapira 
\cite[Proposition 7.3 and Remark 7,4]{KS03}.
In the situation of Proposition \ref{prop-5}, 
for a meromorphic connection $\SN$ on $X$ along 
$D \subset X$ i.e. a holonomic 
$\SD_{X^\an}$-module $\SN$ such that 
$\SN (*D) \simeq \SN$ we define its irregularity 
${\rm irr} ( \SN ) \in \ZZ$ by 
\[ {\rm irr} ( \SN ) = 
\dim H^1 Sol_X ( \SN )_D - \dim H^0 Sol_X ( \SN )_D. \]
We know that it is a non-negative integer 
(see Sabbah \cite[Chapter II, the proof of 
Theorem 1.3.10]{Sab93} etc.). Moreover 
the meromorphic connection $\SN$ is 
regular if and only if ${\rm irr} ( \SN )=0$. 
The last assertion in Proposition \ref{prop-5} 
implies that ${\rm irr} ( \SE_{X\bs D | X}^{\varphi} )
=k$. We can generalize this result as follows. 

\begin{proposition}\label{prop-nn7}
In the situation of Proposition \ref{prop-5}, 
let $\varpi_X : \tl{X}\to X$ be 
the real blow-up of $X$ along $D$. 
For a meromorphic connection $\SN$ on $X$ along 
$D \subset X$ assume that there exist 
meromorphic functions  $\varphi_i \in\SO_X(\ast D)$ 
$(1 \leq i \leq m)$ such that 
for any point $\theta \in S_DX 
\simeq \varpi_X^{-1}(D)\simeq S^1$ 
there exists its sectorial neighborhood $V_\theta\subset X\bs D$
for which we have an isomorphism
\begin{equation*}
\pi^{-1}\CC_{V_\theta}\otimes Sol_X^\rmE( \SN )
\simeq
\pi^{-1}\CC_{V_\theta}\otimes 
\big( \bigoplus_{i=1}^m 
Sol_X^\rmE (\SE_{X\bs D | X}^{\varphi_i}) \big).
\end{equation*}
For $1 \leq i \leq m$ such that 
$\varphi_i$ has a pole along $D$ 
{\rm (}resp. is holomorphic at $D${\rm )} 
we denote by ${\rm ord}_D( \varphi_i )>0$ 
its pole order {\rm (}resp. we set 
${\rm ord}_D( \varphi_i )=0${\rm)}. Then we have 
\[ {\rm irr} ( \SN ) = 
\sum_{i=1}^m {\rm ord}_D( \varphi_i ). \]
In particular, if ${\rm ord}_D( \varphi_i )=0$ 
for any $1 \leq i \leq m$, then 
$\SN$ is regular. 
\end{proposition}

\begin{proof}
The proof is similar to that of Proposition \ref{prop-5}. 
Shrinking $X$ if necessary we may assume that 
$X= \{ u \in \CC \ | \ | u | < 
\varepsilon \}$ for some $\varepsilon >0$, 
$D= \{ u =0 \}$ and $X \setminus D$ is 
covered by some sectors 
$V_{\theta_1}, V_{\theta_2}, \ldots, V_{\theta_l} 
\subset X \setminus D$ for which we have isomorphisms 
\begin{equation*}
\pi^{-1}\CC_{V_{\theta_j}}\otimes Sol_X^\rmE( \SN )
\simeq
\pi^{-1}\CC_{V_{\theta_j}}\otimes 
\Big( \bigoplus_{i=1}^m 
Sol_X^\rmE (\SE_{X\bs D | X}^{\varphi_i}) \Big) 
\qquad (1 \leq j \leq l).
\end{equation*}
Then by the Mayer-Vietoris exact sequences associated to 
the open covering $X \setminus D= \cup_{j=1}^l 
V_{\theta_j}$ of $X \setminus D$ we can modify the proof of 
Proposition \ref{prop-5} to our case. 
\end{proof}

We have also the following result.
We call a finite sum
\[ \varphi( u )= 
\sum_{a \in \QQ,\hspace{1pt} a\leq0} c_{a} u^a \qquad (c_{a} \in \CC ) \]
a Puiseux polynomial of $u^{-1}$.
Here we regard it as a function on a sector $V\subset\CC\setminus\{0\}$
by fixing the branches of the monomials $u^a$.

\begin{proposition}\label{new prop}
In the situation of Proposition \ref{prop-5},
let $\varpi_X : \tl X\to X$ be the real blow-up of $X$ along $D=\{u=0\}$.
Let $\varphi_1, \ldots, \varphi_m, \psi_1, \ldots, \psi_m$
be Puiseux polynomials of $u^{-1}$ without constant terms.
Assume that for a point $\theta\in S_DX\simeq 
\varpi^{-1}_X(D)\simeq S^1$
there exists its open neighborhood $U$ in $\tl X$ 
on which we have an isomorphism
\[\Phi : \bigoplus_{j=1}^m \SA_{\tl X} e^{\varphi_j}\simto
\bigoplus_{i=1}^m \SA_{\tl X} e^{\psi_i}\]
of $\D_{\tl X}^\SA$-modules,
where $\SA_{\tl X} e^{\varphi_j}\simeq \SA_{\tl X}\ (1\leq j\leq m)$
and $\SA_{\tl X} e^{\psi_i}\simeq \SA_{\tl X}\ (1\leq i\leq m)$
are the natural $\D_{\tl X}^\SA$-modules associated to the functions 
$e^{\varphi_j}\ (1\leq j\leq m)$ and $e^{\psi_i}\ 
(1\leq i\leq m)$, respectively.
Then after reordering $\varphi_j$'s and $\psi_i$'s 
for any $1\leq j\leq m$
we have $\varphi_j = \psi_j$.
\end{proposition} 

\begin{proof}
For a sector $V\subset X\setminus D$ we define binary relations
$\underset{V}{\succ}$ and  $\underset{V}{=}$
on the set of Puiseux polynomials of $u^{-1}$ by
\[\varphi \underset{V}{\succ} \psi  \ \Longleftrightarrow\ 
{\Re \varphi} > {\Re \psi} \mbox{ on } V,\ \ 
\ \ \varphi \underset{V}{=} \psi  \ \Longleftrightarrow\ 
{\Re \varphi} = {\Re \psi} \mbox{ on } V.\]
We set also
\[\varphi \underset{V}{\succeq} \psi \ \Longleftrightarrow\ 
\varphi \underset{V}{\succ} \psi  \mbox{ or }
\varphi\underset{V}{=} \psi.\]
Note that the condition ${\Re \varphi} = {\Re \psi} \mbox{ on } V$
for two Puiseux polynomials of $u^{-1}$ without constant terms
implies $\varphi = \psi$. Hence the relation 
$\underset{V}{\succeq}$ defines a partial order 
on the set of such Puiseux polynomials. 
We can choose a point $\theta^{\prime} \in S_DX \cap U$ and its 
small and narrow sectorial neighborhood
$V\subset X\setminus D$ so that
after reordering $\varphi_j\ (1\leq j\leq m)$ 
and $\psi_i\ (1\leq i\leq m)$
we have
\[\varphi_1 \underset{V}{\succeq} \varphi_2 \underset{V}{\succeq} \cdots
\underset{V}{\succeq} \varphi_m\underset{V}{\succeq}\psi_m \
\mbox{ and }
\ \psi_1 \underset{V}{\succeq} \psi_2 \underset{V}{\succeq} \cdots
\underset{V}{\succeq} \psi_m.\]
We can also choose $V$ so that for any $1\leq j\leq m$ and $1\leq i \leq m$
one of the conditions $\varphi_j \underset{V}{\succ}\psi_i,
\varphi_j \underset{V}{=}\psi_i$
and $\varphi_j \underset{V}{\prec}\psi_i$ is satisfied.
Then for any $1\leq j\leq m$, $1\leq i\leq m$ and connected open subset
$W\subset \tl X$ such that $W\cap\varpi^{-1}_X(D) \neq \emptyset$, 
$W\subset \Int\Big(\var{\varpi^{-1}_X(V)}\Big)$
we have an isomorphism
\[\Gamma\big(W; \shom_{\D_{\tl X}^\SA}
(\SA_{\tl X} e^{\varphi_j}, \SA_{\tl X} e^{\psi_i})\big)
\simeq
\begin{cases}
 \CC & (\varphi_j \underset{V}{\preceq}\psi_i)\\
\\
0 & (\varphi_j \underset{V}{\succ}\psi_i).
\end{cases}
\]
Moreover if $\varphi_j \underset{V}{\preceq}
\psi_i$ the $1$-dimensional vector space
\[\Gamma\big(W; \shom_{\D_{\tl X}^\SA}
(\SA_{\tl X} e^{\varphi_j}, \SA_{\tl X} e^{\psi_i})\big)
\simeq \CC\]
is generated by the morphism
\[1\cdot e^{\varphi_j}\longmapsto
e^{\varphi_j-\psi_i}\cdot e^{\psi_i}.
\]
This implies that the restriction of the isomorphism
\[\Phi : \bigoplus_{j=1}^m \SA_{\tl X} e^{\varphi_j}\ 
(\simeq \SA_{\tl X}^{\oplus m})\simto
\bigoplus_{i=1}^m \SA_{\tl X} e^{\psi_i}\
(\simeq \SA_{\tl X}^{\oplus m})\]
to $W$ is represented by the invertible matrix
$$F(u) := (f_{ij}(u))_{1\leq i, j\leq m} \in M_m(\SA_{\tl X}(W)),$$
where we have
\[f_{ij}(u) = C_{ij}e^{\varphi_j-\psi_i}\in \SA_{\tl X}(W)\]
for some constants $C_{ij}\in\CC$ such that $C_{ij}=0$
if $\varphi_j \underset{V}{\succ}\psi_i$.
First let us consider the case where we have
$\varphi_{m-1} \underset{V}{\succ}\varphi_m\ 
(\underset{V}{\succeq} \psi_m)$.
Then the matrix $F(u)$ is a block upper 
triangular matrix of the form
\[F(u) = 
\left(
\begin{array}{cccc}
& \ast\\
\tl F(u) &\vdots \\
&\ast \\ 
0\cdots 0 & \ast
\end{array}
\right)\ \in M_m(\SA_{\tl X}(W)).\]
For it to be invertible,
we have  $C_{m\hspace{1pt}m}\neq0$
and hence $\varphi_m=\psi_m$ on $V$.
Since $\det F(u)\in\SA_{\tl X}(W)$ is invertible in $\SA_{\tl X}(W)$,
the same is true also for $\det {\tl F}(u)\in\SA_{\tl X}(W)$.
Then the morphism
\[{\tl \Phi} : \bigoplus_{j=1}^{m-1} \SA_{\tl X} e^{\varphi_j}\simto
\bigoplus_{i=1}^{m-1} \SA_{\tl X} e^{\psi_i}\]
of $\D_{\tl X}^\SA$-modules on $W$ induced by the matrix $\tl F(u)$
is an isomorphism. Hence the induction on the rank $m$ proceeds.

Next let us consider the general case where we have
\[\varphi_{m-k} \underset{V}{\succ}\varphi_{m-k+1} \underset{V}{=}
\cdots\underset{V}{=}\varphi_m\ (\underset{V}{\succeq} \psi_m)\] 
for some $k\geq 1$. Then the matrix $F(u)$ has the form
\[F(u)=
\begin{array}{lcccccccr}
&&\\
\ldelim({3}{4pt}[]   & &  & \rdelim){3}{1pt}[]\\
&\mbox{\smash{\Huge $\ast$}}& \mbox{\smash{\Huge $\ast$}}\\
& 0\cdots 0  & \ast \cdots \ast\\ 
&\multicolumn{8}{l}{\underbrace{
\hspace{2em}}_{\mbox{$m-k$}}\ \ \underbrace{
\hspace{2em}}_{\mbox{$k$}}}
\end{array}
\hspace{20pt} \in M_m(\SA_{\tl X}(W)).\]
For it to be invertible, we have
\[\big(\varphi_{m-k+1} \underset{V}{=}\varphi_{m-k+2} 
\underset{V}{=}
\cdots\underset{V}{=}\big)\varphi_m \underset{V}{=} \psi_m\]
and 
\[(C_{m\hspace{1pt} m-k+1}, C_{m\hspace{1pt}m-k+2},
\ldots, C_{m\hspace{1pt}m})\neq (0, 0, \ldots, 0).\]
Assume that the condition
\[\psi_{m-k+1} \underset{V}{=}\psi_{m-k+2} \underset{V}{=}
\cdots\underset{V}{=}\psi_m\]
is not satisfied. 
Then $\psi_{m-k+1} \underset{V}{\succ} \psi_m$ and 
the isomorphism
\[\Psi := \Phi^{-1} : \bigoplus_{i=1}^m \SA_{\tl X} 
e^{\psi_i}\simto
\bigoplus_{j=1}^m \SA_{\tl X} e^{\varphi_j}\]
is represented by a block upper triangular matrix
\[G(u) := (g_{ij}(u))_{1\leq i, j \leq m}\in 
M_m(\SA_{\tl X}(W))\]
of the form
\[G(u) =
\begin{array}{rcccccccll}
&&\\
\ldelim({4}{4pt}[]  & & &\rdelim){4}{4pt}[]&
\rdelim\}{2}{10pt}[$m-k$] \\
 &\tl G(u)&\mbox{\smash{\Large $\ast$}}&&\\ 
 &       &      &      &\rdelim\}{2}{10pt}[$k$] \\
&O & O\mbox{\smash{ \Large $\ast$}}\\
&\multicolumn{9}{l}{
\underbrace{\hspace{2em}}_{\mbox{$m-k$}}\ \ 
\underbrace{\hspace{2em}}_{\mbox{$k$}}}
\end{array}
\hspace{55pt}\in M_m(\SA_{\tl X}(W)) .\]
This is a contradiction. Hence we have 
\[\psi_{m-k+1} \underset{V}{=}\psi_{m-k+2} \underset{V}{=}
\cdots\underset{V}{=}\psi_m\]
and the matrix $F(u)$ has the form
\[F(u) = 
\left(
\begin{array}{cc}
\tl F(u) &\mbox{{\large $\ast$}}\\
O & B
\end{array}
\right)\in M_m(\SA_{\tl X}(W)).\]
for an invertible constant matrix $B\in M_k(\CC)$.
Similarly to the case $k=1$ the function $\det {\tl F}
(u)\in\SA_{\tl X}(W)$
is invertible in $\SA_{\tl X}(W)$ and the 
induction on the rank $m$ proceeds.
\end{proof}

By this proposition we obtain the following useful result.
In the situation of Proposition \ref{prop-5} 
let $\N$ be a meromorphic connection
of rank $m$ along the point $D=\{u =0\}\subset X$.
Denote by $\widehat{\SO}_{D|X}$ the formal completion
of $\SO_X$ along $D\subset X$. Then by 
the Hukuhara-Levelt-Turrittin theorem
after a ramification the formalization
\[\widehat{\N} := \widehat{\SO}_{D|X}\otimes_{\SO_{X, D}}\N_D
\in\Mod(\widehat{\SO}_{D|X})\]
of $\N$ along $D$ admits a decomposition by some Puiseux polynomials
$\psi_1(u), \ldots, \psi_m(u)$ of $u^{-1}$ without constant terms.
We call them the exponential factors of $\N$.

\begin{corollary}
In the situation as above, 
assume that for Puiseux 
polynomials $\varphi_1(u), \ldots, \varphi_m(u)$
 of $u^{-1}$ without constant terms and a point $\theta\in 
S_DX\simeq\varpi^{-1}_X(D)\simeq S^1$
there exists its open neighborhood $U$ 
in $\tl X$ on which we have an isomorphism
\[\N^\SA \simeq \bigoplus_{j=1}^m \SA_{\tl X} e^{\varphi_j}\]
of $\D_{\tl X}^\SA$-modules.
Then after reordering $\varphi_j$'s and $\psi_i$'s for any $1\leq j\leq m$
we have $\varphi_j = \psi_j$. 
Namely $\varphi_1(u), \ldots, \varphi_m(u)$ are 
the exponential factors of $\N$ counting with 
multiplicities. 
\end{corollary}

\begin{proof}
By the Hukuhara-Levelt-Turrittin theorem after 
shrinking $U$ we obtain an isomorphism
\[\N^\SA\simeq\bigoplus_{i=1}^m \SA_{\tl X}e^{\psi_i}\]
of $\D_{\tl X}^\SA$-modules on $U$.
Then the assertion follows from Proposition \ref{new prop}.
\end{proof}

By this corollary and Corollary \ref{cor-2} we 
obtain the following result.

\begin{theorem}\label{new-thm}
In the situation as above, assume that 
for convergent Laurent Puiseux series
$\varphi_1, \ldots, \varphi_m$ of $u$ and 
a point $\theta \in S_DX 
\simeq \varpi_X^{-1}(D)\simeq S^1$ 
there exists its sectorial neighborhood $V_\theta\subset X\bs D$
for which we have an isomorphism
\begin{equation*}
\pi^{-1}\CC_{V_\theta}\otimes Sol_X^\rmE( \SN )
\simeq
\bigoplus_{j=1}^m\EE_{V_\theta|X}^{{\Re}\varphi_j}.
\end{equation*}
Then after reordering $\varphi_j$'s and $\psi_i$'s for any $1\leq j\leq m$
the pole part of $\varphi_j$ coincides with $\psi_j$.
In particular, we have
\[ {\rm irr} ( \SN ) = 
\sum_{i=1}^m {\rm ord}_D( \varphi_i ). \]
\end{theorem}
\noindent Note that Proposition \ref{prop-nn7} 
is a very special case of this theorem. 
Theorem \ref{new-thm} could be deduced also 
from \cite[Lemma 6.3.4]{DK17}. 
Our arguments can be applied to exponential 
factors of meromorphic connections 
in higher dimensions as follows. 
Let $X$ be a complex manifold and 
$D \subset X$ a normal crossing divisor in it. 
Let us take local coordinates 
$(u_1, \ldots, u_l, v_1, \ldots, v_{d_X-l})$ 
of $X$ such that $D= \{ u_1 u_2 \cdots u_l=0 \}$. 
We define a partial order $\leq$ on the 
set $\ZZ^l$ by 
\[ a \leq a^{\prime} \ \Longleftrightarrow 
\ a_i \leq a_i^{\prime} \ (1 \leq i \leq l).\] 
Then for a meromorphic function $\varphi\in\SO_X(\ast D)$
on $X$ along $D$ by using its Laurent expansion
\[ \varphi = \sum_{a \in \ZZ^l} c_a( \varphi )(v) \cdot 
u^a \ \in \SO_X(\ast D) \]
with respect to $u_1, \ldots, u_{l}$
we define its order 
${\rm ord}( \varphi ) \in \ZZ^l$ to be 
the minimum 
\[ \min \Big( \{ a \in \ZZ^l \ | \
c_a( \varphi ) \not= 0 \} \cup \{ 0 \} \Big) \]
if it exists. In \cite[Chapter 5]{Mochi11} 
Mochizuki defined the notion of 
good sets of irregular values on $(X,D)$ to be 
finite subsets $S \subset \SO_X(\ast D)/ \SO_X$ 
satisfying some properties. 
We do not recall here its precise definition. 
Just recall that for 
any $\varphi \not= \psi$ in such a set $S$ 
the order ${\rm ord}( \varphi - \psi ) \in \ZZ^l$ 
is defined and the leading term 
$c_{{\rm ord}( \varphi - \psi )} ( \varphi - \psi ) 
(v)$ of its Laurent expansion does not vanish 
at any point $v \in Y= \{ u_1=u_2= \cdots =u_l=0 \} 
\subset D$. Let us call a finite subset 
$S \subset \SO_X(\ast D)/ \SO_X$ 
satisfying only this property a quasi-good 
set of irregular values on $(X,D)$. 

\begin{proposition}\label{new-new prop}
In the situation as above,
let $\varpi_X : \tl X\to X$ be the 
real blow-up of $X$ along the normal 
crossing divisor $D$. Assume that 
$\varphi_1, \ldots, \varphi_m$ 
$($resp. $\psi_1, \ldots, \psi_m)$ 
$\in \SO_X(\ast D)/ \SO_X$ form 
a quasi-good set of irregular values on $(X,D)$. 
Assume also that for a point $\theta \in  
\varpi^{-1}_X( Y ) \subset \varpi^{-1}_X( D )$ 
there exists its open neighborhood $U$ in $\tl X$ 
on which we have an isomorphism
\[\Phi : \bigoplus_{j=1}^m \SA_{\tl X} e^{\varphi_j}\simto
\bigoplus_{i=1}^m \SA_{\tl X} e^{\psi_i}\]
of $\D_{\tl X}^\SA$-modules. 
Then after reordering $\varphi_j$'s and $\psi_i$'s 
for any $1\leq j\leq m$
we have $\varphi_j = \psi_j$.
\end{proposition} 

\begin{proof}
The proof being very similar to that of Proposition 
\ref{new prop}, we shall freely use the 
notations etc. in it. As in the case $l=1$ 
we can choose a sector $V\subset X\setminus D$ 
along $D$ such that $\var{\varpi^{-1}_X(V)} \subset U$ so that 
after reordering $\varphi_j\ (1\leq j\leq m)$ 
and $\psi_i\ (1\leq i\leq m)$ 
we have 
\[\varphi_1 \underset{V}{\succeq} \varphi_2 \underset{V}{\succeq} \cdots
\underset{V}{\succeq} \varphi_m\ 
\mbox{ and }\ \psi_1 \underset{V}{\succeq} \psi_2
 \underset{V}{\succeq} \cdots
\underset{V}{\succeq} \psi_m.\]
Suppose that $\varphi_m \not= \psi_m$ in 
$\SO_X(\ast D)/ \SO_X$. Then there exists a 
weight vector $b=(b_1, \ldots, b_l) \in \ZZ_{>0}^l$ 
such that 
\[ L: = \min \{ \langle a, b \rangle \ | \ a \in \ZZ^l, 
c_a( \varphi_m - \psi_m ) \not=0 \} <0 \]
and the set $\{ a \in \ZZ^l \ | \ 
c_a( \varphi_m - \psi_m ) \not=0, \langle a, b \rangle 
=L \}$ consists of a single point. 
Define a subset $K \subset V$ of the sector $V$ by 
\[ K= \{ (s^{b_1}e^{i \theta_1}, \ldots, 
s^{b_l}e^{i \theta_l}, v_1, \ldots, v_{d_X-l}) \in V 
 \ | \ 0 < s <1 \} \subset V. \] 
Then the restriction of the function 
$\Re ( \varphi_m - \psi_m)$ to $K$ tends to 
$+ \infty$ or $- \infty$ as $s \rightarrow +0$. 
Replacing $\Phi$ with $\Phi^{-1}$ if necessary, 
we may assume that the restriction of the function 
$e^{\varphi_m - \psi_m}$ to $K$ 
increases rapidly as $s \rightarrow +0$. 
Hence we have $C_{m1}=C_{m2}= \cdots = C_{mm}=0$. 
This contradicts to our assumption that 
$\Phi$ is an isomorphism. Then we can 
continue the arguments 
in the proof of Proposition \ref{new prop}. 
\end{proof}

\noindent In the situation of Proposition 
\ref{new-new prop} let $\N$ be a meromorphic connection 
of rank $m$ on $X$ along the the normal 
crossing divisor $D$. If it admits an unramified good lattice 
(at a point $v \in Y= \{ u_1=u_2= \cdots =u_l=0 \} 
\subset D$) in the sense of Mochizuki \cite{Mochi11}, 
we call the elements of $\SO_X(\ast D)/ \SO_X$ 
appearing in the formal decomposition 
the exponential factors of $\N$. Recall that they 
form a good set of irregular values on $(X,D)$. Then by 
Propositions \ref{prop-2} and 
\ref{new-new prop} we obtain the following result. 

\begin{theorem}\label{new-new-thm}
In the situation as above, assume that 
the meromorphic connection $\N$ 
admits an unramified good lattice. Assume also that 
$\varphi_1, \ldots, \varphi_m \in 
\SO_X(\ast D)/ \SO_X$ form 
a quasi-good set of irregular values on $(X,D)$ 
and for a point 
$\theta \in  \varpi^{-1}_X( Y ) \subset \varpi^{-1}_X( D )$ 
there exists its sectorial neighborhood $V_\theta \subset X\bs D$
for which we have an isomorphism
\begin{equation*}
\pi^{-1}\CC_{V_\theta}\otimes Sol_X^\rmE( \SN )
\simeq
\bigoplus_{j=1}^m\EE_{V_\theta|X}^{{\Re}\varphi_j}.
\end{equation*}
Then $\varphi_1, \ldots, \varphi_m \in 
\SO_X(\ast D)/ \SO_X$ are the 
exponential factors of $\N$ 
counting with multiplicities. 
\end{theorem}

\noindent Obviously by ramification maps we can 
generalize this theorem to meromorphic connection 
admitting good lattices in the 
sense of Mochizuki \cite{Mochi11}. 

 The following lemma will be used in 
the proof of Theorem \ref{thm-2}. 
\begin{lemma}\label{newlea} 
 Let $M$ be the complex vector space $\CC^n$ with the standard coordinates
 $x=(x_1, x_2, \ldots, x_n)$ and $N, H\subset M$ its 
linear subspaces defined by
 $$N:=\{x_1=x_2=0\}\subset H:=\{x_1=0\}\subset M.$$ 
 For $\F\in\BDC(\CC_{M\bs H})$ assume that there exists sufficiently small
 $0 < \varepsilon\ll 1$ such that its micro-support 
$SS(\F)\subset T^\ast(M\bs H)$
 does not intersect
\begin{equation*}
U_\varepsilon =
\left\{(x, \xi)\in T^\ast(M\bs H)\ \left|\ 
\xi_1\in\CC,\ \sqrt{\sum_{i=3}^{n}|\xi_i|^2 }\ <\ \varepsilon |\xi_2|
\right.\right\}.
\end{equation*}
 Let $\rho : M^N\to M$ be the {\rm (}complex{\rm )} blow-up of $M$ along $N$, 
 $E\subset M^N$ its exceptional divisor and
 $H'\subset M^N$ the proper transform of $H$ in it. 
 Let $f : M^N\to \PP^1$ be the holomorphic map 
induced by the meromorphic function
 $\displaystyle \frac{x_2}{x_1}$ on $M$ and 
 $\iota : M\bs H\xhookrightarrow{\ \ \ }M^N$ the inclusion map.
 Then for any point $Q\in E\bs H'$ we have the vanishing
 $${\rmR}\Gamma_{({\Re}f)^{-1}(\{t\in\RR | 
t\geq{\Re}f(Q)\})}(\iota_!\F)_Q\simeq 0.$$
 \end{lemma}
 
 \begin{proof}
 There exists an affine open subset $W\simeq\CC^n$ of $M^N$ with
 the coordinates $y=(y_1, y_2, \ldots, y_n)$ such that 
the restriction of the morphism
 $\rho : M^N\to M$ to it is given by
 $$(y_1, y_2, \ldots, y_n)\longmapsto (x_1, x_2, 
\ldots, x_n)=(y_1, y_1y_2, y_3,\ldots, y_n)$$
 and $E\bs H'=E\cap W=\{y\in W\ |\ y_1=0\}.$
 Moreover the restriction $f|_W$ of $f$ to it is given by $(f|_W)(y)=y_2$.
 By the isomorphism $W\bs E\simeq M\bs H$ we regard 
$\F$ as an object of $\BDC(\CC_{W\bs E})$.
 Then it is easy to see that its micro-support does not intersect
 \begin{equation*}
U'_\varepsilon =
\left\{(y, \eta)\in T^\ast(W\bs E)\ \left|\ 
\eta_1\in\CC, \sqrt{\sum_{i=3}^n |\eta_i|^2}\ <\ 
\varepsilon\frac{| \eta_2 |}{| y_1 |}
\right.\right\}.
\end{equation*}
In particular,  the micro-support of the 
restriction of $\F\in\BDC(\CC_{W\bs E})$ to
 $(W\bs E)\cap \{|y_1|<1\}$ does not intersect
\begin{equation*}
V_\varepsilon =
\left\{(y, \eta)\in T^\ast((W\bs E)\cap \{|y_1|<1\})\ \left|\ 
\eta_1\in\CC, \sqrt{\sum_{i=3}^n |\eta_i|^2}\ <\ 
\varepsilon| \eta_2 |
\right.\right\}.
\end{equation*} 
Let $\varpi_{\rm tot} : \tl{W}^{\rm tot}_E=
 \{\zeta\in\CC\ |\ |\zeta|=1\}\times \RR_t\times\CC^{n-1}\to W=\CC^n$
 be the total real blow-up of $W=\CC^n$ along $E\cap W=\{y_1=0\}$
 of  D'Agnolo-Kashiwara (see \cite[\S 7.1]{DK16}) defined by 
 $$(\zeta, t, y_2, \ldots, y_n)\mapsto (t\zeta, y_2, \ldots, y_n)$$
 and identify $W\bs E$ with the open subset $\{t>0\}$ 
of $\tl{W}^{\rm tot}_E$.
 Let $j : W\bs E\xhookrightarrow{\ \ \ }\tl{W}_E^{\rm tot}$ 
be the inclusion map and for a point
 $Q=(0, b_2, \ldots, b_n)\in E\bs H'=E\cap W=\{y_1=0\}$ set
 $$Z= ({\Re}f)^{-1}(\{t\in\RR\ |\ t\geq {\Re}f(Q)\})
=\{y\in W\ |\ {\Re}y_2\geq {\Re}b_2\}.$$
 Then we have isomorphisms
 \begin{align*}
 {\rmR}\Gamma_Z(\iota_!\F)_Q &\simeq
 ({\rmR}\Gamma_Z{\rmR}(\varpi_{\rm tot})_\ast(j_!\F))_Q\\
 &\simeq
  ({\rmR}(\varpi_{\rm tot})_\ast {\rmR}\Gamma_{(
\varpi_{\rm tot})^{-1}(Z)}(j_!\F))_Q\\
  &\simeq
  {\rmR}\Gamma((\varpi_{\rm tot})^{-1}(Q) ; {\rmR}\Gamma_{(
\varpi_{\rm tot})^{-1}(Z)}(j_!\F)).
 \end{align*}
 Hence it suffices to show the vanishing
 $$({\rmR}\Gamma_{(\varpi_{\rm tot})^{-1}(Z)}(j_!\F))_P\simeq0$$
 at each point $P\in(\varpi_{\rm tot})^{-1}(Q)\simeq\{
\zeta\in\CC\ |\ |\zeta|=1\}\simeq S^1$.
By using the above estimate of the 
micro-support of the restriction of $\F\in\BDC(\CC_{W\bs E})$
 to $(W\bs E)\cap \{|y_1|<1\}$ we can show 
it by \cite[Theorem 6.3.1]{KS90}.
 This completes the proof.
 \end{proof}

\section{Main Theorems}\label{sec:9}
First we recall Fourier transforms of algebraic $\SD$-modules.
Let $X=\CC_z^N$ be a complex vector space 
and $Y=\CC_w^N$ its dual.
We regard them as algebraic varieties and 
use the notations $\SD_X$ and $\SD_Y$ for 
the rings of ``algebraic'' differential operators on them.
Denote by $\Modcoh(\SD_X)$ (resp. $\Modhol(\SD_X)$, 
$\Modrh (\SD_X)$) 
the category of coherent (resp. holonomic, 
regular holonomic) $\SD_X$-modules. 
Let $W_N := \CC[z, \partial_z]\simeq\Gamma(X; \SD_X)$ and
$W^\ast_N := \CC[w, \partial_w]\simeq\Gamma(Y; \SD_Y)$
be the Weyl algebras over $X$ and $Y$, respectively.
Then by the ring isomorphism
\begin{equation*}
W_N\simto W^\ast_N\hspace{30pt}
(z_i\mapsto-\partial_{w_i},\ \partial_{z_i}\mapsto w_i) 
\end{equation*}
we can endow a left $W_N$-module $M$ with 
a structure of a left $W_N^\ast$-module.
We call it the Fourier transform of $M$
and denote it by $M^\wedge$.
For a ring $R$ we denote by $\Mod_f(R)$
the category of finitely generated $R$-modules.
Recall that for the affine algebraic varieties $X$
and $Y$ we have the equivalences of categories
\begin{align*} 
\Modcoh(\SD_X)
&\simeq
\Mod_f(\Gamma(X; \SD_X)) = \Mod_f(W_N),\\
\Modcoh(\SD_Y)
&\simeq
\Mod_f(\Gamma(Y; \SD_Y)) = \Mod_f(W^\ast_N)
\end{align*}
(see e.g. \cite[Propositions 1.4.4 and 1.4.13]{HTT08}).
For a coherent $\SD_X$-module $\SM\in\Modcoh(\SD_X)$
we thus can define its Fourier 
transform $\SM^\wedge\in\Modcoh(\SD_Y)$.
It follows that we obtain an equivalence of categories
\begin{equation*}
( \cdot )^\wedge : \Modhol(\SD_X)\simto \Modhol(\SD_Y)
\end{equation*}
between the categories of holonomic $\SD$-modules.
Let
\[X\overset{p}{\longleftarrow}X\times 
Y\overset{q}{\longrightarrow}Y\]
be the projections.
Then by Katz-Laumon \cite{KL85}, for a holonomic $\SD_X$-module
$\SM\in\Modhol(\SD_X)$ we have an isomorphism
\[\SM^\wedge\simeq
\bfD q_\ast(\bfD p^\ast\SM\Dotimes \SO_{X\times Y}
e^{- \langle z, w \rangle }),\]
where $\bfD p^\ast, \bfD q_\ast, \Dotimes$ are
the operations for algebraic $\SD$-modules
and $\SO_{X\times Y}
e^{- \langle z, w \rangle }$ is the integral connection
of rank one on $X\times Y$ associated to the canonical
paring $\langle \cdot ,  \cdot \rangle : X\times Y\to\CC$. 
In particular the right hand side is concentrated in degree zero.
Let $\var{X}\simeq\PP^N$ (resp. $\var{Y}\simeq\PP^N$)
be the projective compactification of $X$ (resp. $Y$).
By the inclusion map $i_X : 
X=\CC^N\xhookrightarrow{\ \ \ }\var{X}=\PP^N$
we extend a holonomic $\SD_X$-module $\SM\in\Modhol(\SD_X)$
on $X$ to the one $\tl{\SM} := i_{X\ast}\SM\simeq\bfD i_{X\ast}\SM$
on $\var{X}$.
Denote by $\var{X}^{\an}$ the underlying complex manifold
of $\var{X}$ and define the analytification
$\tl{\SM}^{\an}\in\Modhol(\SD_{\var{X}^{\an}})$
of $\tl{\SM}$ by
$\tl{\SM}^{\an} = \SO_{\var{X}^{\an}}
\otimes_{\SO_{\var{X}}}\tl{\SM}$. 
Then we set 
\[Sol_{\var X}^\rmE(\tl\SM) := 
Sol_{\var{X}^{\an}}^\rmE(\tl{\SM}^{\an})
\in\BEC(\I\CC_{\var{X}^{\an}}).\]
Note that by Theorem \ref{thm-4} (v) there exists an isomorphism
\[Sol_{\var X}^\rmE( \tl{\SM} )
\simeq
\pi^{-1}\CC_{X^{\an}}\otimes Sol_{\var X}^\rmE( \tl{\SM}).\]
Similarly for the Fourier transform $\SM^\wedge\in\Modhol(\SD_Y)$ 
 we define $Sol_{\var Y}^\rmE(\tl{\SM^\wedge})
 \in\BEC(\I\CC_{\var{Y}^{\an}})$.

\begin{remark}
Let $j_X : X_{\infty}^{\an} =(X^{\an}, {\var X}^{\an}) \to 
{\var X}^{\an}$ 
be the canonical morphism of bordered spaces and 
\[Sol^{\rmE}_{X_\infty} : \BDC_{\rm hol}(\D_X) 
\to\BDC_{\rm hol} (\D_{X_{\infty}^{\an}})
\to \BEC(\I \CC_{X_{\infty}^{\an}})\]
the functor defined 
in \cite[Definition 4.14]{KS16-2}. 
Then for $\SM\in\Modhol(\SD_X)$ 
there exists an isomorphism 
\[Sol^{\rmE}_{X_\infty}(\M)\simeq {\rmE}j_X^{-1}
Sol^{\rmE}_{\var X}(\tl \M). \]
\end{remark}

Let
\[\var{X}^{\an}\overset{\var{p}}{\longleftarrow}
\var{X}^{\an}\times\var{Y}^{\an}\overset{\var{q}}{\longrightarrow}
\var{Y}^{\an}\]
be the projections. 
Then the following theorem is essentially due to
Kashiwara-Schapira \cite{KS16-2} and D'Agnolo-Kashiwara \cite{DK17}. 
For $F\in\BEC(\I\CC_{\var{X}^{\an}})$ we set
\[{}^\rmL F := \bfE\var{q}_{*}(\bfE\var{p}^{-1}F
\Potimes\EE_{X\times Y|\var{X}\times\var{Y}}^{
-{\Re} \langle  z, w \rangle }[N])
\in\BEC(\I\CC_{\var{Y}^{\an}} ) \]
(here we denote $X^{\an}\times Y^{\an}$ etc.
by $X\times Y$ etc. for short) and call 
it the Fourier-Sato (Fourier-Laplace) transform of $F$. 

\begin{theorem}\label{thm-1}
For $\SM\in\Mod_{\rm hol}(\SD_X)$ there exists an isomorphism
\[Sol_{\var Y}^\rmE(\tl{\SM^\wedge})
\simeq{}^\rmL Sol_{\var X}^\rmE(\tl{\SM}).\]
\end{theorem}

\begin{proof}
Let
\[\var{X}\overset{\var{p}}{\longleftarrow}
\var{X}\times\var{Y}\overset{\var{q}}{\longrightarrow}
\var{Y}\]
be the projections and 
$i_Y : Y=\CC^N\xhookrightarrow{\ \ \ }\var{Y}=\PP^N$, 
$i_{X\times Y} : X\times Y\xhookrightarrow{\ \ \ } \var X\times\var Y$ 
the inclusion maps of algebraic varieties. 
Then by \cite[Theorem 1.7.3 and Corollary 1.7.5]{HTT08}
for $\tl{\M^\wedge}\simeq i_{Y\ast}(\M^\wedge)$
we obtain an isomorphism
\[\tl{\M^\wedge}\simeq \bfD \var q_\ast
\big(\bfD \var p^{\ast}\tl{\M}\Dotimes
i_{X\times Y\ast}\SO_{X\times Y}
e^{-\langle z, w\rangle}\big).\]
Moreover we have 
$\big(i_{X\times Y\ast}\SO_{X\times Y}
e^{-\langle z, w\rangle}\big)^{\an}
\simeq
\E_{X\times Y|\var X\times \var Y}^{-\langle z, w\rangle}$.
Then the assertion follows from \cite[Proposition 4.7.2]{HTT08}
and Theorem \ref{thm-4}.
\end{proof}

Form now on, we focus our attention on Fourier transforms 
of regular holonomic $\SD_X$-modules.
For such a $\SD_X$-module $\SM$, 
by \cite[Theorem 7.1.1]{HTT08} we have 
an isomorphism $Sol_{\var X} (\tl{\SM}) 
\simeq i_{X!}Sol_X(\SM)$, where the right hand side 
$i_{X!} Sol_X(\SM)\in D^b(\CC_{\var{X}^{\an}})$ 
is the extension by zero of the classical solution complex 
of $\SM$ to $\var{X}^{\an}$. Moreover 
by \cite[Proposition 9.1.3 and Corollary 9.4.9]{DK16} 
there exists an isomorphism 
\[Sol_{\var X}^\rmE(\tl{\SM})
\simeq
\CC_{\var{X}^{\an}}^\rmE\Potimes\e(i_{X!}Sol_X(\SM)).\]
For an enhanced sheaf $F\in\BEC(\CC_{\var{X}^{\an}})$
on $\var{X}^{\an}$ we define its Fourier-Sato 
(Fourier-Laplace) transform
${}^\rmL F\in\BEC(\CC_{\var{Y}^{\an}})$ by
\[{}^\rmL F := \bfE\var{q}_{*}(\bfE\var{p}^{-1}F\Potimes
\rmE_{X\times Y|\var{X}\times\var{Y}}^{
-{\Re} \langle z, w \rangle}[N])
\in\BEC(\CC_{\var{Y}^{\an}}).\]
Since we have
\[{}^\rmL(\CC^\rmE_{\var{X}^{\an}}\Potimes( \cdot ))
\simeq
\CC^\rmE_{\var{Y}^{\an}}\Potimes{}^\rmL( \cdot )\]
it suffices to study the Fourier-Sato transform of 
the enhanced sheaf $\e(i_{X!}Sol_X(\SM))
\in\BEC( \CC_{\var{X}^{\an}})$
on $\var{X}^{\an}$. 
Fix a regular holonomic $\SD_X$-module $\SM$ 
and denote by $\ch(\SM)\subset T^\ast X\simeq X\times Y$ 
its characteristic variety. 

\begin{definition}\label{defi-1}
We define a (Zariski) open subset $\Omega \subset Y=\CC_w^N$ by:
\begin{equation*}
w\in\Omega  \quad \Longleftrightarrow \quad 
\begin{cases}
\mbox{
there exists an open neighborhood $U$ of $w$ in Y}\\
\mbox{such that the restriction $q^{-1}(U)\cap\ch(\SM)\to U$}\\
\mbox{of $q : X\times Y\to Y$ is 
an unramified covering.
}
\end{cases}
\end{equation*}
\end{definition}

Since $\ch(\SM)$ is $\CC^\ast$-conic,
$\Omega\subset Y=\CC_w^N$ is also $\CC^\ast$-conic.
Note that $\Omega$ is dense in $Y$.
Denote by $k\geq0$ the degree of the covering
$q^{-1}(\Omega)\cap\ch(\SM)\to\Omega$.
For a point $w\in\Omega\subset Y=\CC^N$,
let $\{\mu_1(w), \ldots, \mu_k(w)\} =
q^{-1}(w)\cap\ch(\SM)\subset T^\ast X$ be its
fiber by $q^{-1}(\Omega)\cap\ch(\SM)\to \Omega$. 
Then in a neighborhood of each point $\mu_i(w)\in
q^{-1}(w)\cap\ch(\SM)$ the characteristic variety
$\ch(\SM)$ is smooth and hence there exists a
(locally closed) complex submanifold $S_i\subset X$
such that $\mu_i(w)\in \ch(\SM) = T^\ast_{S_i}X$.
For $1\leq i\leq k$ set
\[\alpha_i(w) := p(\mu_i(w))\in S_i\subset X=\CC^N.\]
By our definition of $\Omega\subset Y=\CC^N$, 
it is easy to see that the restriction of the linear function
\[\ell(w) : X=\CC^N \rightarrow \CC
\hspace{30pt}
(z \mapsto \langle  z, w \rangle  ) \]
to $S_i$ has a non-degenerate (complex Morse)
critical point at $\alpha_i(w)\in S_i$
(see e.g. Kashiwara-Schapira \cite[Lemma 7.2.2]{KS85}).
For $1\leq i\leq k$ denote by
$m_i>0$ the multiplicity of $\SM$ at $\mu_i(w)\in\ch(\SM)$.

\begin{theorem}\label{thm-2}
Let $U\subset\Omega\subset Y=\CC^N$
be a connected and simply connected open subset of $\Omega$.
Then we have an isomorphism
\[\pi^{-1}\CC_U\otimes
\Big(Sol_{\var Y}^\rmE( \tl{\SM^\wedge} )\Big)
\simeq
\bigoplus_{i=1}^k
\pi^{-1}\CC_U\otimes
\Big(\underset{a\to+\infty}{\inj}\ 
\CC_{\{t\geq {\Re} \langle \alpha_i(w), 
w \rangle +a\}}^{\oplus m_i}\Big)\]
in $\BEC(\I\CC_{\var Y^\an})$.
\end{theorem}

\begin{proof}
It suffices to prove that there exists an isomorphism
\[{}^\rmL(\e( i_{X!} Sol_X(\SM)))
\simeq
\bigoplus_{i=1}^k\CC_{\{t\geq {\Re}
\langle \alpha_i(w), w \rangle\}}^{\oplus m_i}\]
of enhanced sheaves on $U \subset \Omega \subset Y$.
Let 
\[X\times\RR_s\overset{p_1}{\longleftarrow}
(X\times\RR_s)\times (Y\times\RR_t)
\overset{p_2}{\longrightarrow}Y\times\RR_t\]
be the projections.
Then by D'Agnolo-Kashiwara \cite[Lemma 7.2.1]{DK17}
on $Y \subset \var{Y}$ we have an isomorphism
\begin{equation*}
{}^\rmL\big(\e( i_{X!} Sol_X(\SM))\big)
\simeq \Q 
\bigl(\rmR p_{2!}(p_1^{-1}(\CC_{\{s\geq0\}}\otimes
\pi^{-1}Sol_X(\SM))\otimes\CC_{\{t-s-{\Re}
\langle z, w \rangle \geq0\}}[N])\bigr),
\end{equation*}
where $\Q : \BDC(\CC_{\var{Y}^{\an}\times\RR})
\to\BEC(\CC_{\var{Y}^{\an}})$ is the quotient functor.
For a point $(w, t)\in Y^{\an}\times\RR$
we have also isomorphisms
\begin{align*}
&\bigl(\rmR p_{2!}(p_1^{-1}(\CC_{\{s\geq0\}}\otimes
\pi^{-1}Sol_X(\SM))
\otimes\CC_{\{t-s-{\Re} \langle z, w 
\rangle \geq0\}}[N])\bigr)_{(w, t)}\\
&\simeq
\rmR\Gamma_c(\{(z, s)\in X^{\an}
\times\RR\ |\ t-s-{\Re} \langle z, w \rangle \geq0, 
s\geq0\}; \pi^{-1}Sol_X(\SM)[N])\\
&\simeq
\rmR\Gamma_c( X^{\an} ; \rmR \pi_!(
\CC_{\{ t-s-{\Re} \langle z, w \rangle \geq 0, s\geq 0 \}} 
\otimes 
\pi^{-1} Sol_X(\SM)[N])) \\
&\simeq
\rmR\Gamma_c( X^{\an} ; (\rmR \pi_! 
\CC_{\{ t-s-{\Re} \langle z, w \rangle \geq 0, s\geq 0 \}} ) 
\otimes Sol_X(\SM)[N]) \\
&\simeq 
\rmR\Gamma_c(\{  z \in X^{\an}\ |\ {\Re}
\langle  z, w \rangle \leq t\}; Sol_X(\SM)[N]), 
\end{align*}
where we used 
\[ \rmR \pi_! 
\CC_{\{ t-s-{\Re} \langle z, w \rangle \geq 0, s\geq 0 \}} 
\simeq \CC_{ \{ {\Re} \langle  z, w \rangle \leq t \} }
\]
in the last isomorphism. 
Fix $w\in U \subset \Omega\subset Y=\CC^N$. 
Then by an argument similar to the one in the proof 
of Esterov-Takeuchi \cite[Theorem 5.5]{ET15} 
we can prove the vanishing 
\[{\rmR}\Gamma_c(\{ z \in X^{\an}\ |
\ {\Re}
\langle  z, w \rangle \leq t\}; Sol_X(\SM)[N])\simeq0 \]
for $t\ll0$ as follows. 
Let $H_\infty := \var{X}\bs X\simeq\PP^{N-1}$
be the hyperplane at infinity of $\var{X}=\PP^N$ and
\[\ell(w) : X=\CC^N\to\CC\hspace{17pt}(z\longmapsto \langle z, w\rangle)\]
the linear function defined by $w$.
Then the meromorphic extension of $\ell(w)$
to $\var{X}=\PP^N$ has points of indeterminacy
in the complex submanifold $H(w) := \var{\ell(w)^{-1}(0)}
\cap H_\infty\simeq\PP^{N-2}$ of $\var{X}=\PP^N$.
Let $\rho : \var{X}^{H(w)}\to \var{X}$ be the complex blow-up
of $\var{X}$ along $H(w)$ and $\iota :
 X\xhookrightarrow{\ \ \ } \var{X}^{H(w)}$
the inclusion map.
Then the meromorphic extension $f$ 
of $\ell(w)$ to $\var{X}^{H(w)}$
has no point of indeterminacy and we obtain a commutative diagram 
\[\xymatrix@M=7pt@C=37pt{
X\ar@{^{(}->}[r]^-\iota\ar@{->}[d]_-{\ell(w)} & 
\var{X}^{H(w)}\ar@{->}[d]^-f\\
\CC\ar@{^{(}->}[r]_-i & \PP^1
}\]
of holomorphic maps, where $i : 
\CC\xhookrightarrow{\ \ \ }\PP^1$
is the inclusion map. See the proof of 
Lemma \ref{newlea} or \cite[Theorem 3.6]{MT13}. 
Note that $f$ is proper and set 
\[ L( Sol_X(\SM) ) 
:= \rmR p_{2!}\big(p_1^{-1}(\CC_{\{s\geq0\}}\otimes
\pi^{-1}Sol_X(\SM))\otimes
\CC_{\{t-s-{\Re}\langle z,w\rangle\geq0\}}[N]
\big)\in\BDC(\CC_{Y\times\RR}). \]
Then for $t\in\RR$ the stalk
\begin{align*}
& \big(  L( Sol_X(\SM) ) \big)_{(w, t)}\\
&\simeq
\rmR\Gamma_c\big(
\{z\in X^{\an}\ |\ {\Re}\langle z, w\rangle\leq t\};
Sol_X(\SM)[N]
\big)
\end{align*}
is calculated as follows:
\begin{align*}
&\rmR\Gamma_c\big(\{z\in X^{\an}\ |\ {\Re}\langle z,
 w\rangle\leq t\}; Sol_X(\SM)[N]\big)\\
&\simeq
\rmR\Gamma_c\big(\var{X}^{\an}; \CC_{\{z\in X^{\an}\ 
|\ {\Re}\langle z, w\rangle\leq t\}}
\otimes Sol_X(\SM)[N]\big)\\
&\simeq
\rmR\Gamma_c\big(\var{X}^{H(w)}; f^{-1}\CC_{\{\tau\in 
\CC\ |\ {\Re}\tau\leq t\}}
\otimes \iota_!Sol_X(\SM)[N]\big)\\
&\simeq
\rmR\Gamma_c\big(\PP; \CC_{\{\tau\in \CC\ |\ {\Re}\tau\leq t\}}
\otimes \rmR f_!\iota_!Sol_X(\SM)[N]\big)\\
&\simeq
\rmR\Gamma_c\big(\{\tau\in \CC\ |\ {\Re}\tau\leq t\};
\rmR f_\ast\iota_!Sol_X(\SM)[N]\big).
\end{align*}
Since the direct image $\rmR f_\ast\iota_!Sol_X(\SM)[N]$ 
of the perverse sheaf 
$\iota_!Sol_X(\SM)[N]\in\BDC_{\CC-c}(\CC_{\var{X}^{H(w)}})$ 
is constructible, for $t\ll0$ the restrictions of 
its cohomology sheaves to 
the closed half space $\{\tau\in\CC\ | \ 
{\Re}\tau\leq t\}\subset\CC \subset\PP^1$ of $\CC$ are 
locally constant. 
Thus for $t\ll0$ we obtain the vanishing 
\[\rmR\Gamma_c\big(\{z\in X^{\an}\ |\ {\Re}\langle z, 
w\rangle\leq t\}; Sol_X(\SM)[N]\big)\simeq0.\]
Let $H_{\infty}^{\prime}=f^{-1}( \{ \infty \} )
\subset \var{X}^{H(w)}$ be 
the proper transform of $H_{\infty}$ by 
the blow-up $\rho$ and 
$g: \var{X}^{H(w)} \setminus 
f^{-1}( \{ \infty \} ) \rightarrow \CC$ 
the restriction of $f$ to $\var{X}^{H(w)} \setminus 
f^{-1}( \{ \infty \} )$. 
Then by our construction 
of $\var{X}^{H(w)}$ for any $t \in \RR$ 
the real hypersurface $({\Re} g)^{-1}(t)$ 
intersects the exceptional divisor 
$E= \rho^{-1}(H(w)) 
\subset \var{X}^{H(w)}$ of $\rho$ 
transversally (see the proof of 
Lemma \ref{newlea}). Hence it is 
compact and smooth in $\var{X}^{H(w)}$. 
This implies that the morphism 
${\Re} g: \var{X}^{H(w)} \setminus 
f^{-1}( \{ \infty \} ) \rightarrow \RR$ 
is proper. 
Let $\iota^{\prime}: X \hookrightarrow 
 \var{X}^{H(w)} \setminus 
f^{-1}( \{ \infty \} )$ be the inclusion map. 
Then for any $t \in \RR$ we have 
isomorphisms 
\begin{align*}
&\rmR\Gamma_c(\{z\in X^{\an}\ |\ {\Re}
\langle z, w  \rangle   \leq t \}; 
Sol_X(\SM)[N]) \\
& \simeq 
\rmR\Gamma_c(\{ x \in \var{X}^{H(w)} \setminus 
f^{-1}( \{ \infty \} ) 
\ |\ {\Re} g (x) \leq t \}; 
 \iota_! Sol_X(\SM)[N]) \\
& \simeq 
\rmR\Gamma_c\big(\{ s \in \RR \ |\ s \leq t\};
\rmR ( {\Re} g )_\ast \iota^{\prime}_!Sol_X(\SM)[N]\big).
\end{align*}
For $1\leq i \leq k$ set $c_i := 
{\Re} \langle \alpha_i(w), w \rangle 
= {\Re} ( \ell (w)) ( \alpha_i(w))$. 
For the fixed $w\in\Omega$, 
after reordering $\alpha_1(w), \ldots, \alpha_k(w)\in X$ 
we may assume that
\[  c_1 \ \leq c_2 \ \leq \cdots  \cdots \ \leq c_k.\]
If for some $1\leq i<j \leq k$ such that 
$c_i < c_j$ 
the open interval $(c_i, c_j)\subset\RR$
does not intersect the set $\{ c_1,c_2, c_3, \ldots, c_k  \}$ 
(of the stratified critical values of ${\Re} ( \ell (w) ) :  
X^{\an}\to\RR)$,
then by Lemma \ref{newlea} for any $t_1, t_2\in\RR$ such that 
$c_i < t_1 < t_2 < c_j$
we obtain an isomorphism 
\begin{align*}
&\rmR\Gamma_c(\{z\in X^{\an}\ |\ {\Re}
\langle z, w  \rangle   \leq t_2\}; 
Sol_X(\SM)[N]) \\
&\simto
\rmR\Gamma_c(\{z\in X^{\an}\ |\ {\Re}
 \langle z, w \rangle \leq t_1\};
 Sol_X(\SM)[N]).
\end{align*}
Indeed, we can prove the equivalent one 
\begin{align*}
& \rmR\Gamma_c\big(\{ s \in \RR \ |\ s \leq t_2 \};
\rmR ( {\Re} g )_\ast \iota^{\prime}_!Sol_X(\SM)[N]\big) 
 \\
&\simto 
\rmR\Gamma_c\big(\{ s \in \RR \ |\ s \leq t_1 \};
\rmR ( {\Re} g )_\ast \iota^{\prime}_!Sol_X(\SM)[N]\big) 
\end{align*}
as follows. Let 
\[ \Lambda_{\ell (w)} = 
\{ (z, {\rm grad} \ell (w) (z))
\in T^*X \simeq X \times Y \ | \ 
z \in X \} \subset T^*X \]
be the (non-homogeneous) Lagrangian 
submanifold of $T^*X$ associated to 
the function $\ell (w) : X \rightarrow \CC$. 
Since $\Lambda_{\ell (w)} \simeq 
X \times \{ w \} \subset X \times Y$,  
the condition $w \in \Omega$ implies 
that $\Lambda_{\ell (w)}$ intersects 
$\ch\SM$ only over finitely many points 
in $X$. Hence, for any $\tau \in \CC$ the 
conormal bundle of the level set 
$\ell (w)^{-1}( \tau ) \simeq \CC^{N-1} 
 \subset X= \CC^N$ of $\ell (w)$ 
in $X$ intersects the 
micro-support of the perverse sheaf 
$Sol_X(\SM)[N]$ only in 
the zero-section $T^*_XX \simeq X$ of 
$T^*X$ on a neighborhood of 
$H(w) \subset \var{X}$. Note also that 
$\var{\ell(w)^{-1}( \tau )} \simeq \PP^{N-1}$ 
($\tau \in \CC$) are all the projective 
hyperplanes of $\var{X} = \PP^N$ 
containing $H(w) \simeq \PP^{N-2}$ 
and different from 
the special one $H_{\infty} \simeq \PP^{N-1}$. 
The same is true also for any small perturbation of 
$w$ in $\Omega$. This implies that 
the assumption of Lemma \ref{newlea} 
is satisfied for the submanifolds $N=H(w) \subset 
H=H_{\infty}$ in an affine open subset 
$M \simeq \CC^N$ of $\var{X} = \PP^N$ and 
$Sol_X(\SM)|_{M \setminus H_{\infty}} [N] 
\in \BDC_{\CC-c}(\CC_{M \setminus H_{\infty}})$. 
Then we obtain the desired isomorphism 
by applying Kashiwara's non-characteristic 
deformation lemma (see \cite[Proposition 2.7.2]{KS90}) 
to the complex 
$\rmR ( {\Re} g )_\ast \iota^{\prime}_!Sol_X(\SM)[N]$ 
on the real line $\RR$. 
Since $Sol_X(\SM)[N]$ is a perverse sheaf on $X$, 
by Kashiwara-Schapira \cite[Theorem 9.5.2]{KS85}, 
for any  $1\leq i \leq k$ there exists an isomorphism
\[Sol_X(\SM)[N]\simeq\CC_{S_i}^{\oplus m_i}[d_{S_i}]\]
in the localized category 
$\BDC(\CC_X; \mu_i(w))$ of 
$\BDC( \CC_X )$ 
at $\mu_i(w)\in T^\ast X$
(see \cite[Definition 6.1.1]{KS90}).
Moreover the restriction of the function
${\Re}(\ell(w)) : X^{\an}\to\RR$ to the submanifold
$S_i \subset X$ has a non-degenerate 
(real Morse) critical point
of Morse index $d_{S_i}$ at $\alpha_i(w)\in S_i$.
It follows that  for the closed subset
\[Z_i := \{z\in X^{\an}\ |\ {\Re}
 \langle z, w  \rangle \geq c_i \}
\subset X^{\an} \]
of $X^{\an}$ and $j\in\ZZ$ we have
\begin{align*}
H^j\rmR\Gamma_{Z_i}(Sol_X(\SM)[N])_{\alpha_i(w)}
&\simeq
H^j\rmR\Gamma_{Z_i}(\CC_{S_i}^{\oplus m_i}[
d_{S_i}])_{\alpha_i(w)}\\
&\simeq
\begin{cases}
\CC^{m_i} & (j=0)\\
\\
0 & (\mbox{otherwise}).
\end{cases}
\end{align*}
Hence for any $t=c_i 
= {\Re} \langle \alpha_i(w), w \rangle \in\RR$ 
($1\leq i\leq k$)
there exists $0<\e\ll1$ such that
\begin{align*}
&\rmR\Gamma_c(\{z\in X^{\an}\ |\ {\Re} \langle z, w \rangle \leq 
t+\e\}; Sol_X(\SM)[N])\\
&\simeq
\rmR\Gamma_c(\{z\in X^{\an}\ |\ {\Re} \langle z, w \rangle \leq t\}; 
Sol_X(\SM)[N])\\
&\simeq
\rmR\Gamma_c(\{z\in X^{\an}\ |\ {\Re} \langle z, w 
\rangle \leq t-\e\}; 
Sol_X(\SM)[N])\oplus\CC^{m_i}.
\end{align*}
This implies that the restriction of $L( Sol_X(\SM) )$
to the fiber $\pi^{-1}(w) \simeq\RR$ of $w\in U \subset \Omega$
is isomorphic to the sheaf
\[ \bigoplus_{i=1}^k\CC_{\{t\geq {\Re} \langle \alpha_i(w), 
w  \rangle \}}^{\oplus m_i}.\] 
Since the subsets $(U\times\RR)\cap\{t\geq
{\Re}\langle \alpha_i(w), w\rangle\}
\simeq U\times\RR_{\geq0}$ 
of $U\times\RR$
are connected and simply connected,
we can extend this isomorphism to $U \times \RR 
\subset Y^{\an}\times\RR$.
This completes the proof.
\end{proof}

\begin{corollary}\label{cor-3}
The restriction of the Fourier transform
$\SM^\wedge\in\Mod_{\rm hol}(\SD_Y)$ of $\SM$
to $\Omega\subset Y$ is an integrable connection.
Moreover its rank is equal to $\sum_{i=1}^km_i$.
\end{corollary}

\begin{proof}
By Theorem \ref{thm-2}, 
locally on $\Omega$ there exists an isomorphism
\begin{align*}
Sol_{\var Y}^\rmE ( \tl{\SM^\wedge} )
&\simeq
\bigoplus_{i=1}^k \underset{a\to+\infty}{\inj}\ 
\CC_{\{t\geq {\Re}  \langle \alpha_i(w), w 
\rangle +a\}}^{\oplus m_i}.
\end{align*}
Let $i_0 : {\var Y}^{\an}\hookrightarrow {\var Y}^{\an}
\times \RR_\infty$ be the map 
given by $y\mapsto(y, 0)$. 
Then locally on $\Omega$ we have isomorphisms
\begin{align*}
Sol_Y(\M^\wedge)
&\simeq
\alpha_{{\var Y}^{\an}}i_0^!{\rmR^\rmE}Sol_{\var Y}^{\rmE}(
\tl{\M^\wedge} )\\
&\simeq
\alpha_{{\var Y}^{\an}}i_0^!{\rmR^\rmE}
\Big(\bigoplus_{i=1}^k \underset{a\to+\infty}{\inj}\ 
\CC_{\{t\geq {\Re}  \langle \alpha_i(w), w 
\rangle +a\}}^{\oplus m_i}\Big)\\
&\simeq
\alpha_{{\var Y}^{\an}}i_0^!
\Bigl(\bigoplus_{i=1}^k \underset{a\to+\infty}{\inj}\ 
\CC_{\{t< {\Re} 
\langle \alpha_i(w), w \rangle +a\}}^{\oplus m_i}[1]\Bigr).
\end{align*}
On the other hand, for $a\gg0$ the sheaf
$\CC_{\{t< {\Re} \langle \alpha_i(w), w \rangle +a\}}$ being
isomorphic to the constant sheaf
$\CC_{Y^{\an}\times\RR}$
locally on a neighborhood of $i_0(\Omega)
\subset\Omega\times\RR$, 
there exist also an isomorphism
\[i_0^!\Bigl(\bigoplus_{i=1}^k \underset{a\to+\infty}{\inj}\ 
\CC_{\{t< {\Re} 
\langle \alpha_i(w), w \rangle +a\}}^{\oplus m_i}[1]\Bigr)
\simeq
\CC_\Omega^{\oplus_{i=1}^km_i}\]
locally on $\Omega$.
We thus obtain an isomorphism
\[Sol_Y(\SM^\wedge)\simeq\CC_\Omega^{\oplus_{i=1}^km_i}\]
locally on $\Omega$.
Then by $\ms{Sol_Y(\SM^\wedge)}=\ch(\SM^\wedge)$
the characteristic variety $\ch(\SM^\wedge)$ of
$\SM^\wedge$ is contained in the zero-section of 
$T^\ast Y^{\an}$ on $\Omega$.
Now all the assertions are clear.
\end{proof}

Next fix a point $w\in\Omega$ such that $w\neq0$ and set
\[\LL := \CC w = \{\lambda w\ |\ \lambda\in\CC\}\subset Y=\CC^N.\]
Then $\LL$ is a complex line isomorphic to $\CC_\lambda$.
By the $\CC^\ast$-conicness of $\Omega$ its open subset
$\LL\bs\{0\}\simeq\CC_\lambda^\ast$ is contained in $\Omega$.
Note that $\alpha_i(\lambda w)=\alpha_i(w)$ ($1\leq i\leq k$)
for any $\lambda\in\CC^\ast$. Since we fixed the point
$w\in\Omega$ we set $\alpha_i=\alpha_i(w)$ 
($1\leq i\leq k$) for short.
Let $\PP := \LL \sqcup \{\infty\} \subset \var{Y} = \PP^N$
be the projective compactification of $\LL$.
We extend $\SM^\wedge|_{\LL\bs\{0\}}$ to a 
meromorphic connection on $\PP$ and denote
it by $\SL\in\Modhol(\SD_\PP)$. 
Note that $\SL$ is isomorphic to 
the restriction of $\tl{\SM^\wedge}$ 
to $\PP \subset \var{Y}$ on a 
neighborhood of the point 
$\infty \in \PP$. 
The following result is a generalization (of some part)
of Esterov-Takeuchi \cite[Theorem 5.6]{ET15}.
Let $\varpi_\PP : \tl{\PP}\to\PP$ 
be the real blow-up of $\PP$ along the divisor 
$\{ \infty \} \subset \PP$. 

\begin{theorem}\label{thm-3}
For any point $\theta\in 
\varpi_{\PP}^{-1}( \{ \infty \} )\simeq S^1$
there exists its open neighborhood $W$ in $\tl{\PP}$
such that we have an isomorphism
\[\SL^\SA|_W
\simeq
\Bigr(\bigoplus_{i=1}^k
\bigl(
(\SE_{\LL|\PP}^{- \langle \alpha_i, w \rangle \lambda})^\SA
\bigr)^{\oplus m_i}
\Bigl)|_W
\]
of $\SD_\PP^\SA$-modules on $W$. In particular, 
the functions $- \langle \alpha_i, w \rangle \lambda$ 
of $\lambda$ are the exponential factors of 
the meromorphic connection $\SL$ at the point 
$\infty \in \PP$ $($the multiplicity of 
$- \langle \alpha_i, w \rangle \lambda$ is equal 
to $m_i)$. Moreover the meromorphic connection $\SL$ 
is regular at the origin 
$0 \in \LL \subset \PP$. 
\end{theorem}

\begin{proof}
By \cite[Corollary 9.4.12]{DK16}
we have isomorphisms
\begin{align*}
Sol_\PP^\rmE(\SE_{\LL|\PP}^{- \langle \alpha_i, w \rangle \lambda})
\simeq
\EE_{\LL|\PP}^{-{\Re}(  \langle \alpha_i, w \rangle \lambda)}
\simeq
\underset{a\to+\infty}{\inj}\ 
 \CC_{\{t\geq {\Re}( \langle \alpha_i,  w \rangle \lambda)+a\}}.
\end{align*} 
On the other hand, by Theorem \ref{thm-2}
for any $\theta\in 
\varpi_{\PP}^{-1}( \{ \infty \} )\simeq S^1$
there exists its sectorial neighborhood
$V_\theta\subset\PP\bs \{ \infty \}$ such that we have an isomorphism
\[\pi^{-1}\CC_{V_\theta}\otimes Sol_\PP^\rmE(\SL)
\simeq
\bigoplus_{i=1}^k 
\pi^{-1}\CC_{V_\theta}\otimes 
\Big( \underset{a\to+\infty}{\inj}\ 
\CC_{\{t\geq {\Re}( \langle \alpha_i, w \rangle 
\lambda)+a\}}^{\oplus m_i} \Big). 
\]
We thus obtain an isomorphism
\[\pi^{-1}\CC_{V_\theta}\otimes Sol_\PP^\rmE(\SL)
\simeq
\pi^{-1}\CC_{V_\theta}\otimes Sol_\PP^\rmE\Bigl(
\bigoplus_{i=1}^k\big(\SE_{\LL|\PP}^{-
\langle \alpha_i, w \rangle \lambda}\big)^{\oplus m_i}
\Bigr).
\]
Then the first assertion follows from Corollary \ref{cor-2}. 
The second one follows from Theorem \ref{new-thm}. 
We can show the remaining one by using the last part of 
Proposition \ref{prop-nn7}. 
\end{proof}

\begin{remark}
In Esterov-Takeuchi \cite{ET15} for the Fourier transform
$( \cdot )^\wedge : \Modhol(\SD_X)\to\Modhol(\SD_Y)$
the authors used the kernel $\SO_{X\times Y}e^{
\langle z, w \rangle }$
instead of the one $\SO_{X\times Y}e^{-
\langle z, w  \rangle }$ used in this paper.
This is the reason why we obtain 
$\SE_{\LL|\PP}^{- \langle \alpha_i, w \rangle \lambda}$
instead of $\SE_{\LL|\PP}^{ \langle \alpha_i, w \rangle \lambda}$ 
in Theorem \ref{thm-3}. 
\end{remark}

As in Esterov-Takeuchi \cite[Remark 5.7]{ET15}, 
by Theorems  \ref{thm-2} and \ref{thm-3}
we easily obtain the Stokes lines of the meromorphic connection 
$\SL\in\Modhol(\SD_\PP)$ at $\infty \in \PP$ as follows.
We define a subset $J$ of $\{1, 2, \ldots, k\}^2$ by 
\[J = \{(i, j)\in\{1, 2, \ldots, k\}^2  \ |\ 
 \langle \alpha_i, w \rangle 
\neq  \langle \alpha_j, w \rangle \}.\]
Then the union of the Stokes lines is equal to the set
\[\bigcup_{(i, j)\in J}\{\lambda\in\CC\simeq
\LL\ |\ {\Re} \langle \alpha_i-\alpha_j, \lambda w \rangle =0\}\]
in $\LL\simeq\CC$.

\medskip 
Now we shall study the enhanced ind-sheaf 
$Sol_{\var Y}^\rmE( \tl{\SM^\wedge} 
)\simeq{}^{\rm L}Sol_{\var X}^\rmE( \tl{\SM)} $ 
on the remaining set $D:= Y\bs\Omega\subset Y=\CC^N$. 
Fix a point $w\in D= Y\bs\Omega$ such that $w\neq 0$ and define 
$\ell(w) : X=\CC^N\to\CC, H(w) := \var{\ell(w)^{-1}(0)}
\cap H_\infty\simeq\PP^{N-2}, f : \var{X}^{H(w)}\to\PP^1$ etc, 
as before. 
Then there exists finite points $P_1, P_2, \ldots, P_\ell\in\CC$
in $\CC$ such that the restriction 
\[ \var{X}^{H(w)} \setminus f^{-1}( \{ \infty, P_1, P_2, \ldots, 
P_\ell\})\to \PP^1 \bs\{ \infty, P_1, P_2, \ldots, P_\ell\}\]
of $f$ is a stratified fiber bundle with respect to a
stratification associated to the perverse sheaf 
$\iota_!Sol_X(\SM)[N]\in\BDC_{\CC-c}(\CC_{\var{X}^{H(w)}})$ 
(see e.g. Goresky-MacPherson \cite[P43]{GM88}). 
Recall that we set 
\[ L( Sol_X(\SM) ) 
:= \rmR p_{2!}\big(p_1^{-1}(\CC_{\{s\geq0\}}\otimes
\pi^{-1}Sol_X(\SM))\otimes
\CC_{\{t-s-{\Re}\langle z,w\rangle\geq0\}}[N]
\big)\in\BDC(\CC_{Y\times\RR}). \]
For $t\in\RR$ its stalk 
$\big( L( Sol_X(\SM) ) \big)_{(w, t)}$ 
is calculated as in the case $w \in \Omega$. 
Moreover for $t\ll0$ we have also the vanishing 
\[\rmR\Gamma_c\big(\{z\in X^{\an}\ |\ {\Re}\langle z, 
w\rangle\leq t\}; Sol_X(\SM)[N]\big)\simeq0.\]
For $1\leq i\leq\ell$ define a function $h_i : 
\CC_\tau\to \CC$ by $h_i(\tau) = \tau-P_i$
so that we have $h_i^{-1}(0) = \{P_i\}\subset\CC$.
Let us set
\[K_i := \{\tau\in\CC\ |\ {\Re}(h_i(\tau))\geq0\}\subset\CC.\]
Then for $t\in\RR$ and $0<\e\ll 1$ there exists
a distinguished triangle
\begin{align*}
&\bigoplus_{i:{\Re}P_i=t} {\rmR}\Gamma_{K_i}\big(
\rmR f_\ast\iota_! Sol_X(\SM)[N]\big)_{P_i}\ \longrightarrow\\
&
{\rmR}\Gamma_c\big(\{\tau\in\CC\ | \ {\Re}\tau\leq t+\e\};
\rmR f_\ast\iota_! Sol_X(\SM)[N]\big)\ \longrightarrow\\
&{\rmR}\Gamma_c\big(\{\tau\in\CC\ | \ {\Re}\tau\leq t-\e\};
\rmR f_\ast\iota_! Sol_X(\SM)[N]\big)\ \overset{+1}{\longrightarrow}.
\end{align*}
Moreover for each $1\leq i\leq\ell$ such that
${\Re}P_i=t$ we have an isomorphism
\[{\rmR}\Gamma_{K_i}\big(
\rmR f_\ast\iota_! Sol_X(\SM)[N]\big)_{P_i}
\simeq
\phi_{h_i}\big(
\rmR f_\ast\iota_! Sol_X(\SM)[N]\big)[-1],\]
where $\phi_{h_i} : \BDC(\CC_\CC)\to\BDC(\CC_{h_i^{-1}(0)})
=\BDC(\CC_{\{P_i\}})$
is the vanishing cycle functor associated to $h_i$
(see Kashiwara-Schapira \cite{KS90} and Dimca \cite{Dim04} etc.).
Since $f$ is proper, 
we have also an isomorphism
\[\phi_{h_i}\big(\rmR f_\ast\iota_! Sol_X(\SM)[N]\big)[-1]
\simeq
{\rmR}\Gamma\big(f^{-1}(P_i); 
\phi_{h_i\circ f}\big(\iota_! Sol_X(\SM)[N]\big)[-1]\big).
\]
Note that $\phi_{h_i\circ f}\big(\iota_! Sol_X(\SM)[N]\big)[-1]$
is a perverse sheaf on 
$(h_i\circ f)^{-1}(0)=f^{-1}(P_i)\subset\var{X}^{H(w)}$. 
Moreover for the inclusion map 
$i_{\{ w \}} : \{ w \}\xhookrightarrow{\ \ \ } Y=\CC^N$
we have the following result.

\begin{lemma}\label{lemma-conica} 
For $a\gg0$ we have isomorphisms
\begin{align*}
Sol_{\{w\}}(\bfD i_{\{w\}}^\ast\SM^\wedge)
&\simeq
\big( L( Sol_X(\SM) ) \big)_{(w, a)}
\\
& \simeq
{\rm R}\Gamma_c(\{z\in X^{\an}\ |\ {\Re}\langle z,
 w \rangle\leq a\}; Sol_X(\SM)[N]).
\end{align*}
\end{lemma} 

\begin{proof}
First note that we have isomorphisms
\begin{align*}
Sol_{\{w\}}(\bfD i_{\{w\}}^\ast\SM^\wedge)
&\simeq
 \alpha_{\{w\}}i_0^!\bfR^\rmE 
Sol_{\{w\}}^\rmE(\bfD i_{\{w\}}^\ast\SM^\wedge)\\
 &\simeq
 \alpha_{\{w\}}i_0^!\bfR^\rmE \bfE i_{\{w\}}^{-1}
Sol_{Y}^\rmE(\SM^\wedge)
\end{align*}
By the isomorphism
\[Sol_{\var Y}^\rmE(\SM^\wedge)\simeq
``\underset{a\to +\infty}{\varinjlim}"\ 
\CC_{\{t\geq a\}}\Potimes 
{}^\rmL\big(\e( i_{X!} Sol_X(\SM))\big),
\]
we obtain an isomorphism 
\[\bfE i_{\{w\}}^{-1}Sol_{Y}^\rmE(\SM^\wedge)
\simeq
``\underset{a\to +\infty}{\varinjlim}"\ 
\CC_{\{t\geq a\}}\Potimes \Q 
\Big( \big( L( Sol_X(\SM) ) \big)
|_{\{w\}\times\RR} \Big).\]
Let us clarify the structure of the sheaf
$\big( L( Sol_X(\SM) ) \big) |_{\{w\}\times\RR}$
on $\{w\}\times\RR\simeq\RR$.
Denote by $k\ (\leq\ell)$ the cardinality of the subset
$\{{\Re}P_1, {\Re}P_2, \ldots, 
{\Re}P_\ell\}\subset\RR$ of $\RR$ and set
\[\{t_1 < t_2 <\cdots< t_k\} = \{{\Re}P_1, 
{\Re}P_2, \ldots, {\Re}P_\ell\}.\]
For $1\leq i\leq k-1$, we set also $I_i 
:= (t_i, t_{i+1})\subset \RR$.
Then for any $t\in I_i$, there exists an isomorphism
\begin{align*}
&\rmR\Gamma_c\big(\{\tau\in\CC\ |\ {\Re}\tau\leq t\} ;
 \rmR f_\ast\rmR\iota_!Sol_X(\SM)[N]\big)\\
&\simto
\rmR\Gamma_c\big(\{\tau\in\CC\ |\ {\Re}\tau\leq t_i\}
 ; \rmR f_\ast\rmR\iota_!Sol_X(\SM)[N]\big).
\end{align*}
This implies that we have an isomorphism
\[ \big( L( Sol_X(\SM) ) \big)_{(w, t)}
\simeq
\big( L( Sol_X(\SM) ) \big)_{(w, t_i)}
\]
for any $t\in I_i$.
Set $F_i= \big( L( Sol_X(\SM) ) \big)_{(w, t_i)}$ 
($1\leq i\leq k-1$) 
and $G= \big( L( Sol_X(\SM) ) \big)_{(w, t_k)}$.
Then by some distinguished 
triangles associated to the decomposition
$$\{t_1\leq t\}=\{t_1\leq t< t_2\}\sqcup \{t_2\leq t< t_3\}
\sqcup\cdots\sqcup \{t_{k-1}\leq t< t_k\}\sqcup\{t_k\leq t\}$$
of the interval $\{t_1\leq t\}\subset\RR$ 
the support of $\big( L( Sol_X(\SM) ) \big) 
|_{\{w\}\times\RR}$
is divided into those of $(F_i)_{\{t_i
\leq t< t_{i+1}\}}$ and $G_{\{t_k\leq t\}}$.
By (the proof of) Lemma \ref{llema}, 
for $1\leq i\leq k-1$ we have 
\begin{align*}
&i_0^!\bfR^\rmE\Big(
``\underset{a\to +\infty}{\varinjlim}"\ 
\CC_{\{t\geq a\}}\Potimes(F_i)_{\{t_i\leq t< t_{i+1}\}}\Big)\\
&\simeq
{\rm R}\pi_\ast\rihom\Big(
\CC_{\{t\geq0\}}\oplus\CC_{\{t\leq0\}},
``\underset{a\to +\infty}{\varinjlim}"\ 
(F_i)_{\{t_i+a\leq t< t_{i+1}+a\}} \Big)\\
&\simeq
{\rm R}\pi_\ast\rihom\Big(\CC_{\RR},
``\underset{a\to +\infty}{\varinjlim}"\ 
(F_i)_{\{t_i+a\leq t< t_{i+1}+a\}} \Big)\simeq 0.
\end{align*}
We thus obtain an isomorphism
\begin{align*}
&i_0^!\bfR^\rmE\Big(
``\underset{a\to +\infty}{\varinjlim}"\ 
\CC_{\{t\geq a\}}\Potimes
\big( L( Sol_X(\SM) ) \big) 
|_{\{w\}\times\RR}\Big)\\
&\simeq
i_0^!\bfR^\rmE\Big(
``\underset{a\to +\infty}{\varinjlim}"\ 
\CC_{\{t\geq a\}}\Potimes
G_{\{t_k\leq t\}}\Big).
\end{align*}
Moreover there exists an isomorphism
\begin{align*}
&i_0^!\bfR^\rmE\Big(
``\underset{a\to +\infty}{\varinjlim}"\ 
\CC_{\{t\geq a\}}\Potimes
G_{\{t_k\leq t\}}\Big)\\
&\simeq
{\rm R}\pi_\ast\rihom\Big(
\CC_{\{t\geq0\}}\oplus\CC_{\{t\leq0\}},
``\underset{a\to +\infty}{\varinjlim}"\ G_{\{t_k+a\leq t\}}\Big)\\
&\simeq
{\rm R}\pi_\ast\rihom\Big(\CC_{\RR},
``\underset{a\to +\infty}{\varinjlim}"\ G_{\{t_k+a\leq t\}}\Big)
\simeq G.
\end{align*}
Then the assertion immediately follows.
\end{proof}

By the proof of this lemma, 
we see that $Sol_Y(\SM^\wedge)$ 
is monodromic. 
Indeed, for $\lambda \in \RR_+$ we have 
${\Re} \ell ( \lambda w) = \lambda \cdot {\Re} \ell (w)$. 
This implies that $Sol_Y(\SM^\wedge)$ is $\RR_+$-conic 
(see Ito-Takeuchi \cite{IT18} for a precise proof). 
For a general theory of conic ind-sheaves see \cite{Pre11}. 
We can rewrite Lemma \ref{lemma-conica} 
more geometrically as follows. 

\begin{proposition}\label{prop-4}
For any $\tau\in\CC\bs\{P_1, P_2, \ldots, P_\ell\}$
we have
\[
\chi_w\big(Sol_{\{w\}}(\bfD i_{\{w\}}^\ast\SM^\wedge)\big)
= \chi\big(\rmR\Gamma_c(X; Sol_X(\SM)[N])\big)-
\chi\big(\rmR\Gamma_c(\ell(w)^{-1}(\tau); Sol_X(\SM)[N])\big). \] 
\end{proposition}

\begin{proof}
For $a\gg0$ the restriction of 
$\ell(w) : X=\CC^N\to\CC$ to 
the open subset 
$\{ \tau \in \CC \ | \ 
{\Re} \tau >a \} \subset \CC$ 
is a stratified fiber bundle with respect to a 
stratification associated to the perverse sheaf 
$Sol_X(\SM)[N]\in\BDC_{\CC-c}(\CC_{X})$. 
Then we obtain the assertion by applying 
the K{\"u}nneth formula to 
Lemma \ref{lemma-conica}. 
\end{proof}

Note that a special case of Proposition \ref{prop-4}
was obtained by Brylinski \cite[Corollaire 8.6]{Bry86}. 

\medskip 
From now, let us calculate (some part of) the characteristic cycle 
of the Fourier transform $\SM^\wedge$. 
Let $\cup_{i=1}^m D_i$ be the irreducible 
decomposition of $D=Y \setminus \Omega \subset Y= \CC^N$. 
For $1\leq i\leq m$ such that $d_{D_i}=N-1$
we shall calculate the multiplicity
\[{\rm mult}_{T^\ast_{D_i}Y}\SM^\wedge \geq0\]
of the Fourier transform $\SM^\wedge$ of $\SM$ along $T^\ast_{D_i}Y$. 
For this purpose, first we calculate the local Euler-Poincar\'e index 
\[\chi_v\big(Sol_Y(\SM^\wedge)\big) =
\sum_{j\in\ZZ}(-1)^j\dim H^jSol_Y(\SM^\wedge)_v\]
of $Sol_Y(\SM^\wedge)$ at generic 
smooth points $v\in D_{\rm reg}$ of $D$. 
Fix $1\leq i\leq m$ and let 
$v\in D_i \cap D_{\rm reg}$ be such a generic point.
Let $M\subset Y$ be a subvariety of $Y$ which
intersects $D_i$ at $v$ transversally.
We call it a normal slice of $D_i$ at $v$.
By definition $M$ is smooth on a neighborhood of $v$.
Let $i_M : M\xhookrightarrow{\ \ \ } Y=\CC^N$ be the inclusion map.
Then it is non-characteristic for $\SM^\wedge$ and hence
we obtain an isomorphism
\[Sol_Y(\SM^\wedge)|_M
\simeq
Sol_M(\bfD i_M^\ast\SM^\wedge).\]
Let us consider the special case where $d_{D_i}=N-1$. 
In this case we have $d_M=1$. 
For the holonomic $\SD$-module 
$\SK := {\bfD}i_M^\ast\SM^\wedge\in\Modhol(\SD_M)$ 
on the normal slice 
$M$ of $D_i$ at $v$, consider the distinguished triangle
\[\rmR\Gamma_{\{v\}}(\SK)\longrightarrow
\SK\longrightarrow
\SK(\ast\{v\})\overset{+1}{\longrightarrow}.
\]
Then by the results in Section \ref{sec:8} we obtain
\begin{align*}
\chi_v\big(Sol_Y(\SM^\wedge)\big)
&=\chi_v\big(Sol_M(\SK)\big)\\
&=
\chi_v\Big(Sol_M\big(\rmR\Gamma_{\{v\}}(\SK)\big)\Big)
+\chi_v\Big(Sol_M\big(\SK(\ast\{v\})\big)\Big)\\
&=
\chi_v\Big(Sol_{\{v\}}\big({\bfD}i_{\{v\}}^\ast\SM^\wedge\big)\Big)
+\chi_v\Big(Sol_M\big(\SK(\ast\{v\})\big)\Big).
\end{align*}
Combining this with Proposition \ref{prop-4}
we finally obtain the following theorem.

\begin{theorem}\label{new-thm-6}
Assume that $d_{D_i}=N-1$. Then for $|\tau|\gg0$ we have 
\begin{align*}
\chi_v\big(Sol_Y(\SM^\wedge)\big)
=\ &\chi\big(\rmR\Gamma_c(X; Sol_X(\SM)[N])\big)-
\chi\big(\rmR\Gamma_c(\ell(v)^{-1}(\tau); Sol_X(\SM)[N])\big) 
\\
&+\chi_v\Big(Sol_M\big(\SK(\ast\{v\})\big)\Big).
\end{align*}
\end{theorem}

For the meromorphic connection $\SK(\ast \{v\})$
on the Riemann surface $M$ we can calculate
$\chi_v \big(Sol_M\big(\SK(\ast\{v\})\big)\big)$
by Proposition \ref{prop-nn7} as follows. 
Recall that $-\chi_v\big(Sol_M\big(\SK(\ast\{v\})\big)\big)$ 
is equal to the irregularity 
${\rm irr} ( \SK(\ast\{v\}) ) \geq 0$ of 
$\SK(\ast\{v\})$. 
Let $\varpi_M : \tl{M}\to M$ be 
the real blow-up of $M$ along $\{ v \} \subset M$. 
Shrinking the normal slice $M$ 
if necessary we may assume that 
$M= \{ u \in \CC \ | \ | u | < 
\varepsilon \}$ for some $\varepsilon >0$, 
$\{ v \} = \{ u =0 \}$ and 
$M \setminus \{ v \} \subset \Omega$. 
Then we define Laurent Puiseux series 
$\varphi_i ( u )$ ($1 \leq i \leq k$) 
and their pole orders 
${\rm ord}_{\{ v \}}( \varphi_i ) \geq 0$ 
as in Section \ref{sec:1}. Moreover 
by Theorems \ref{thm-2} and  \ref{thm-4} (ii) 
for any point $\theta \in S_{\{ v \}} M
\simeq \varpi_M^{-1}( \{ v \} )\simeq S^1$
there exists its sectorial 
neighborhood $V_\theta\subset M \bs \{ v \}$
for which we have an isomorphism 
\begin{equation*}
\pi^{-1}\CC_{V_\theta}\otimes Sol_M^\rmE( \SK(\ast\{v\}) )
\simeq \bigoplus_{i=1}^k 
\big(  \EE_{V_\theta | M}^{- \Re\varphi_i}
 \big)^{\oplus m_i}.
\end{equation*}
Then by Theorem \ref{new-thm} we obtain the following result.
\begin{theorem}\label{thm-nn7}
The exponential factors appearing in the 
Hukuhara-Levelt-Turrittin decomposition
of the meromorphic connection $\SK(\ast\{v\})$ at $v\in M$
are the pole parts of $-\varphi_i\ (1\leq i \leq k)$.
Moreover for any $1\leq i\leq k$ the 
multiplicity of the pole part of $-\varphi_i$
is equal to $m_i$.
In particular we have
\[ {\rm irr} ( \SK(\ast\{v\}) ) = 
\sum_{i=1}^k m_i \cdot 
{\rm ord}_{\{ v \}}( \varphi_i ). \]
\end{theorem}

\begin{theorem}\label{new-thm-7}
Assume that $d_{D_i}=N-1$ and 
let $v_0 \in \Omega$ be a point of $\Omega$. 
Then for $|\tau|\gg0$ and a generic point 
$v \in D_i \cap D_{\rm reg}$ 
the multiplicity ${\rm mult}_{T^\ast_{D_i}Y}\SM^\wedge \geq0$
of the Fourier transform $\SM^\wedge$ along $T^\ast_{D_i}Y$
is given by 
\begin{align*}
{\rm mult}_{T^\ast_{D_i}Y}\SM^\wedge
=& 
\chi\big(\rmR\Gamma_c(\ell(v)^{-1}(\tau); Sol_X(\SM)[N])\big) \\
&-\chi\big(\rmR\Gamma_c(\ell(v_0)^{-1}(\tau); Sol_X(\SM)[N])\big) 
 + {\rm irr} ( \SK(\ast\{v\}) ).
\end{align*}
\end{theorem}

\begin{proof}
By Kashiwara's local index theorem for holonomic $\SD$-modules 
in \cite[Corollary 6.3.4]{Kas83-2}, 
we have 
\[{\rm mult}_{T^\ast_{D_i}Y}\SM^\wedge
=\chi_{v_0}\big(Sol_M(\SM^\wedge)\big)
-\chi_v\big(Sol_M(\SM^\wedge)\big).
\]
Note that for $|\tau|\gg0$ and 
$v_0\in\Omega$ the equality
\[\chi_{v_0}\big(Sol_M(\SM^\wedge)\big)
=
\chi\big(\rmR\Gamma_c(X; Sol_X(\SM)[N])\big)
-
\chi\big(\rmR\Gamma_c(\ell(v_0)^{-1}(\tau); Sol_X(\SM)[N])\big) 
\]
holds (see Brylinski \cite[Corollaire 8.6]{Bry86}).
Then the assertion follows from Theorem \ref{new-thm-6}. 
Recall that we have
${\rm irr} ( \SK(\ast\{v\}) ) =
- \chi_v \big(Sol_M ( \SK(\ast\{v\}) ) \big)$.

\end{proof}

\begin{example}
Let us consider the special case where 
the Fourier transform 
$\SM^\wedge$ is a confluent
$A$-hypergeometric system on $Y=\CC^2$. 
For the subset $A=\{2, 3\}$ of 
the $1$-dimensional lattice $\ZZ 
\subset \RR$ consider the embedding
\[i_T : T=\CC^\ast \xhookrightarrow{\ \ \ } X=\CC_{x, y}^2,\
s\longmapsto (s^2, s^3)\]
of the 1-dimensional torus $T=\CC^\ast$ associated to it.
For a complex number $c$ such that $c\notin\ZZ$ 
set $\SL=\SO_Ts^{c-1}\in\Modrh(\SD_T)$
and $\SM={\bfD} i_{T\ast}\SL$. 
Then $c$ is non-resonant in the sense of Adolphson \cite{Ado94} 
and we have $\SM\in\Modrh(\SD_X)$. 
Set $Z=\var{i_T(T)} = \{x^3=y^2\}\subset X$.
Then we can easily see that 
\[\ch\SM = T^\ast_{\{0\}}X\cup T^\ast_ZX\]
and 
\[{\rm mult}_{ T^\ast_{\{0\}}X}\SM = 2,\hspace{10pt}
{\rm mult}_{ T^\ast_{Z}X}\SM =1.\]
Moreover the open subset $\Omega\subset Y$
of $Y= \CC^2$ is given by
\[\Omega = \{(w_1, w_2)\in Y=\CC^2\ |\ 
w_1 w_2 \not= 0\}\subset Y.\]
Then by Corollary \ref{cor-3} the rank
of $\SM^\wedge|_\Omega$ is equal to
\[{\rm mult}_{ T^\ast_{\{0\}}X}\SM 
+{\rm mult}_{ T^\ast_{Z}X}\SM=2+1=3.\]
This coincides with Adolphson's result in \cite{Ado94}. 
Indeed, the normalized volume 
of the convex hull $\Delta \subset \RR$ of 
$\{ 0 \} \cup A \subset \RR$ is equal to 
$3$. Set $\SF=Sol_X(\SM)[2]\in\BDC_{\CC-c}(\CC_{X^\an}).$ 
Set also $D_1 := \{(w_1, w_2)\in Y\ |\ w_2 =0\}$. 
As a normal slice of $D_1\subset Y=\CC^2$ at the point $(1, 0)\in D_1$
let us consider the submanifold $M := \{(1, u)\ |\ u \in\CC\}
\simeq\CC_{u}$ of $Y=\CC^2$.
Then by Theorem \ref{thm-2} it is easy to 
show that the restriction of $Sol_{\var Y}^\rmE( \tl{\SM^\wedge})
\simeq{}^{\rmL}Sol_{\var X}^\rmE(\tl{\SM} )$ to 
$M \cap \Omega \simeq \CC^*_u$ is isomorphic to
\[\Big(\underset{a\to+\infty}{\inj}\ 
\CC_{\{t\geq a\}}^{\oplus 2}\Big)
\oplus
\Big(\underset{a\to+\infty}{\inj}\ 
\CC_{\{t\geq {\Re}(\varphi(u))+a\}}\Big),\]
where we set $\varphi(\eta) := \frac{4}{27 u^2}.$ 
By Theorem \ref{thm-nn7} this implies that 
the irregularity of the 
meromorphic connection obtained by 
restricting $\SM^\wedge$ to the normal slice $M$ 
is equal to $2$. 
Set $v=(1,0) \in M \cap D_1$ and 
$v_0=(1, \varepsilon ) \in M \cap \Omega = 
M \setminus D_1$ ($\varepsilon \not= 0$). 
Then for $| \tau | \gg 0$ we have 
\[ \chi\big(\rmR\Gamma_c(\ell(v)^{-1}(\tau); \SF)\big) 
-
\chi\big(\rmR\Gamma_c(\ell(v_0)^{-1}(\tau); \SF)\big) 
=(-2)-(-3)=1. \]
By Theorem \ref{new-thm-7} we thus obtain 
\[{\rm mult}_{T^\ast_{D_1}Y} \SM^\wedge=1 +2=3.\]
Similarly for $D_2 = \{(w_1, w_2)\in Y\ |\ w_1 =0\}$ 
we can show 
${\rm multi}_{T^\ast_{D_2}Y} \SM^\wedge=0$. 
Hence $\SM^\wedge$ is an integrable connection 
on $Y \setminus D_1 \supset \Omega$. 
In fact, for $A= \{ 2,3 \}$ 
this follows from Adolphson's result in \cite{Ado94}. 
\end{example}

\begin{example}
For the smooth hypersurface 
$Z= \{ z \in X= \CC^N \ | \ 
z_1^2+ \cdots +z_{N-1}^2+ z_N=1 \} \subset X$ consider 
the perverse sheaf 
$\SF= \CC_Z [N-1] 
\in\BDC_{\CC-c}(\CC_{X^\an})$ on $X= \CC^N$. 
Let $\SM \in\Modrh(\SD_X)$ be 
the regular holonomic $\SD_X$-module 
such that $Sol_X(\SM)[N] = \SF$. 
Then $\SM$ is not monodromic and we have 
\begin{align*}
& \ch\SM  =  T^\ast_ZX
\\
&= \{ (z, \zeta(2z_1, \ldots,  2z_{N-1}, 1))
\in T^*X \simeq X \times Y \ | \ 
\zeta \in \CC, z_1^2+ \cdots +z_{N-1}^2 + z_N=1 \}. 
\end{align*}
It follows that for the projection 
$q: T^*X \simeq X \times Y \rightarrow Y$ 
and a point $w \in Y \setminus \{ 0 \} 
= \CC^N  \setminus \{ 0 \}$ we have 
\[ w \in q ( T^\ast_ZX ) \qquad \Longleftrightarrow \qquad 
w_N \not= 0. \]
Moreover if $w \in q ( T^\ast_ZX ) \setminus \{ 0 \}$ 
then the set $q^{-1}(w) \cap T^\ast_ZX$ 
is explicitly calculated as 
\begin{equation*}
q^{-1}(w) \cap T^\ast_ZX =
\left\{(z , w)\in T^*X \simeq X \times Y \left|
\begin{array}{l}
z_i = \frac{w_i}{2w_N } \ (i = 1, \ldots, N-1),\\
z_N=1-\frac{1}{4w_N^2}(w_1^2+\cdots+w_{N-1}^2) 
\end{array}
\right.\right\}.
\end{equation*}
This shows that the open subset $\Omega \subset 
Y= \CC^N$ is given by 
\[ \Omega = \{ w \in Y= \CC^N \ | \ 
w_N \not= 0 \} \]
and the morphism $q^{-1} ( \Omega ) \cap 
\ch\SM \rightarrow \Omega$ induced by $q$ 
is a covering of degree $1$. 
By Corollary \ref{cor-3} 
the restriction of the Fourier transform 
$\SM^\wedge\in\Mod_{\rm hol}(\SD_Y)$ of $\SM$ 
to $\Omega\subset Y$ is an integrable connection 
of rank $1$. We can also easily see that 
$\SM^\wedge$ has some irregularities at 
infinity by using Theorem \ref{thm-3}. 
Set $D= Y \setminus \Omega = \{ w \in Y= \CC^N \ | \ w_N =0 \}$
 and take its normal slice 
\[  M = \{ (a_1, \ldots, a_{N-1}, u)
 \in Y= \CC^N \ | \ u \in \CC \} \subset Y \]
at the generic point $ (a_1, \ldots, a_{N-1}, 0)
\in D \setminus \{ 0 \}$. Then for 
a  point $v= (a_1, \ldots, a_{N-1}, u)
\in M \setminus D= M \cap \Omega$ $(u \not=0)$ 
we have 
\begin{equation*}
q^{-1}(v) \cap T^\ast_ZX =
\left\{(z , v)\in T^*X \simeq X \times Y \left|
\begin{array}{l}
z_i = \frac{a_i}{2u} \ (i = 1, \ldots, N-1),\\
z_N=1-\frac{1}{4u^2}(a_1^2+\cdots+a_{N-1}^2) 
\end{array}
\right.\right\}
\end{equation*}
Since the Puiseux series 
\[ \varphi(u)= 
\langle z, v \rangle 
= u+\frac{1}{4u}(a_1^2+\cdots+a_{N-1}^2) \]
of $u$ start from the negative degree $-1$, 
by Theorem \ref{thm-nn7} the irregularity of meromorphic 
connection obtained by restricting 
$\SM^\wedge$ to the normal slice $M \subset Y$ 
is equal to $1$. 
Set $v=(a_1, \ldots, a_{N-1}, 0) \in M \cap D$ and 
$v_0=(a_1, \ldots, a_{N-1}, \varepsilon ) \in M \cap \Omega = 
M \setminus D$ ($\varepsilon \not= 0$). 
Note that there exists an isomorphism 
$Z \simeq \CC^{N-1}$ induced by the projection 
$X= \CC^N \rightarrow \CC^{N-1}$, 
$z \mapsto (z_1, \ldots, z_{N-1})$. 
Then for $| \tau | \gg 0$ we have 
$\ell(v)^{-1}(\tau) \cap Z \simeq \CC^{N-2}$, 
$\ell(v_0)^{-1}(\tau) \cap Z \simeq 
\{ ( \lambda_1, \ldots, \lambda_{N-1}) \in \CC^{N-1} 
 \ | \ \lambda_1^2 + \cdots + \lambda_{N-1}^2 =1 \}$ and 
hence 
\begin{align*}
&\chi\big(\rmR\Gamma_c(\ell(v)^{-1}(\tau); \SF)\big) 
-
\chi\big(\rmR\Gamma_c(\ell(v_0)^{-1}(\tau); \SF)\big)\\
&=(-1)^{N-1}-\{(-1)^{N-1}+(-1)^{N-1}(-1)^{N-2}\}=1. 
\end{align*}
By Theorem \ref{new-thm-7} we thus obtain 
\[{\rm mult}_{T^\ast_{D}Y} \SM^\wedge=
1+1=2.\]
\end{example}

We shall rewrite Theorem \ref{new-thm-7} more explicitly. 
For this purpose, we introduce a ``conification" of the perverse sheaf
$\SF = Sol_X(\SM)[N]\in\BDC_{\CC-c}(\CC_{X^\an})$ as follows.
Let $j=i_X : X=\CC^N\xhookrightarrow{\ \ \ }\var{X}=\PP^N$ be the 
projective compactification 
of $X= \CC^N$ and 
$h$ the (local) defining equation of 
the hyperplane at infinity $H_\infty =\var{X}\bs X$ in $\var{X}$
such that $H_\infty = h^{-1}(0)$.
Moreover let $\gamma : X\bs\{0\}=\CC^N\bs\{0\}\twoheadrightarrow
H_\infty=\PP^{N-1}$ be the canonical projection.
Then
\[\SG := 
\gamma^{-1} \psi_h( j_! \SF)\in\BDC_{\CC-c}(\CC_{X\bs\{0\}})\]
is a perverse sheaf on $X\bs\{0\}$.
We call it the conification of $\SF$.
In particular $\SG$ is monodromic in 
the sense of \cite{Ver83} and \cite{Bry86}.
We extend it to a perverse sheaf on the whole $X$ 
and denote it also by $\SG$.
Let $\SN\in\Modrh(\SD_X)$ be the regular holonomic
$\SD_X$-module such that $Sol_X(\SN)[N]\simeq\SG$.
Now we shall recall the well-known relationship between
the characteristic cycle ${\rm CC}(\SN)$ of $\SN$
and that of its Fourier transform $\SN^\wedge$.
For a subvariety $V\subset X$ of $X$ set
\[T_V^\ast X:= \var{T_{V_{\rm reg}}^\ast X}\subset T^\ast X,\]
where $V_{\rm reg}\subset V$ stands for the regular part of $V$ .
Then there exist some $\CC^\ast$-conic subvarieties $V_i\subset X$ 
of $X$ and positive integer $\mi>0\ (1\leq i\leq r)$ such that
\[{\rm CC}(\SN) = \sum_{i=1}^r\mi\cdot[T^\ast_{V_i}X].\]
By the natural identification $T^\ast X \simeq T^\ast Y$ for any
$1\leq i\leq r$ there exists a $\CC^\ast$-conic subvariety
$W_i\subset Y$ of $Y=\CC^N$ such that $T^\ast_{V_i}X=T^\ast_{W_i}Y$.
Note that their projectivizations $\PP(V_i)\subset\PP(X)\simeq\PP^{N-1}$
and $\PP(W_i)\subset\PP(Y)\simeq\PP^{N-1}$ are dual varieties in the 
classical theory of projective duality
(see Gelfand-Kapranov-Zelevinsky \cite[\S1.3]{GKZ94}).
Then we have 
\[{\rm CC}(\SN^\wedge) = \sum_{i=1}^r\mi\cdot[T^\ast_{W_i}Y].\]
This equality can be easily seen also by using the arguments for 
the enhanced micro-supports ${\rm SS}^\rmE(\cdot)$
in D'Agnolo-Kashiwara \cite{DK17}.
We will show also that $W_i\subset D= Y\bs\Omega$
for $1\leq i\leq r$ satisfying some condition 
and $d_{W_i}=N-1$.
By $\SG_0 := \psi_h( j_! \F)\in\BDC_{\CC-c}(\CC_{H_\infty^\an})$ 
we can rewrite our results as follows. 

\begin{definition}\label{new-defc} 
Let $w\neq 0$ be a point of $Y=\CC^N$.
Then we say that $\F\in\BDC_{\CC-c}(\CC_{X^{\an}})$ 
is moderate at infinity over $w$ 
if there exists a (complex analytic) Whitney stratification 
$\overline{X}=\sqcup_{\alpha\in A}S_\alpha$ of $\overline{X}=\PP^N$
adapted to $j_! \F$ and subdividing the one 
$\overline{X}=X\sqcup H_\infty$ 
such that for any stratum $S_\alpha \subset X=\CC^N$ 
the set $\overline{T^\ast_{S_{\alpha}} X}\cap 
T^\ast_{H_{\infty}} \overline{X}$ is contained 
in the zero section of $T^* \overline{X}$ over 
a neighborhood of $H(w)$ in $\overline{X}$. 
Moreover for a subset $B \subset Y \setminus \{ 0 \}$ 
we say that $\F\in\BDC_{\CC-c}(\CC_{X^{\an}})$ 
is moderate at infinity over $B$ if it is so 
over any point $w \in B$ in it. 
\end{definition}

Obviously, if $\F\in\BDC_{\CC-c}(\CC_{X^{\an}})$ is monodromic then it is 
moderate at infinity over any point $w\neq0$ of $Y$. 
Moreover the set of the points $w \not= 0$ over which 
$\F\in\BDC_{\CC-c}(\CC_{X^{\an}})$ is moderate at 
infinity is $\CC^*$-conic in $Y= \CC^N$. 

\begin{example}\label{exm-near} 
Let $f(z) \in \CC [z_1,z_2, \ldots, z_N]$ be a polynomial 
on $X= \CC^N$ such that the hypersurface of 
$\PP^{N-1}$ defined by its top degree part is smooth. 
Then the constructible sheaf 
$\CC_{f^{-1}(0)} \in\BDC_{\CC-c}(\CC_{X^{\an}})$ is moderate at 
infinity over any point $w\neq 0$ of $Y=\CC^N$.
\end{example}

\begin{lemma}\label{l-conic} 
Assume that $\F\in\BDC_{\CC-c}(\CC_{X^{\an}})$ is moderate 
at infinity over a point $w\neq 0$ of $Y=\CC^N$.
Then for $|\tau|\gg0$ 
we have $$\chi({\rmR}\Gamma_c(\ell(w)^{-1}(\tau)\ ;\ \F))
=\chi({\rmR}\Gamma_c(H_\infty\bs H(w)\ ;\ \psi_h( j_! \F))).$$
\end{lemma}

\begin{proof}
For $w \in Y \setminus \{ 0 \}$ and $\tau \not= 0$ let 
$\gamma ( w, \tau ): \ell(w)^{-1}(\tau) \simto 
H_\infty\bs H(w) \simeq \CC^{N-1}$ 
be the isomorphism induced by the 
projection 
$\gamma : X\bs\{0\}=\CC^N\bs\{0\}\twoheadrightarrow
H_\infty=\PP^{N-1}$. For $\e >0$ let $U_{\e}  
\supset H(w)$ be the open neighborhood of $H(w) 
\simeq \PP^{N-2}$ in $H_{\infty} 
\simeq \PP^{N-1}$ consisting of points 
whose distances from $H(w)$ 
(with respect to the Fubini-Study 
metric of $H_{\infty}$) are less that $\e$. 
Then by our assumption there exist $0< \e \ll 1$ and $C \gg 0$ 
such that for any $\tau \in \CC$ 
with $| \tau | \geq C$ we have 
\[ \chi({\rmR}\Gamma_c( U_{\e} \bs H(w) \ ;\ \psi_h( j_! \F))) 
= \chi({\rmR}\Gamma_c( 
\gamma (  w, \tau  )^{-1} (U_{\e} \setminus H(w)) 
\ ;\ \F)). \]
Moreover the compact subset $K_{\e}  : = H_{\infty} \setminus U_{\e} 
\subset H_\infty\bs H(w) \simeq \CC^{N-1}$ 
is a closed ball with real analytic 
boundary $\partial K_{\e}$. 
Take an affine chart $W= \CC^N_x$ of 
$\var{X}=\PP^N$ such that $W \cap H_{\infty} 
= H_\infty\bs H(w) = \{ x_N=0 \} 
\subset W= \CC^N_x$, 
$\ell(w)^{-1}(\tau) = \{ x_N= \frac{1}{\tau} \}$ 
for any $\tau \not= 0$ and 
the restriction $\gamma |_W : W 
\rightarrow W \cap H_{\infty}$ of $\gamma$ 
to it is given by $x \mapsto (x_1, \ldots, x_{N-1})$. 
For the Whitney stratification of $\var{X}=\PP^N$ 
adapted to $j_! \F$, by the microlocal 
Bertini-Sard theorem (see \cite[Proposition 8.3.12]{KS90}) 
if $\e >0$ is generic enough the 
real analytic hypersurface $( \gamma |_W)^{-1} 
( \partial K_{\e} )$ of $W= \CC^N_x$ intersects 
its all strata transversally on a 
neighborhood of $H_{\infty}$. Then by 
the analytic curve selection lemma it is 
easy to show that for $| \tau |$ large enough 
we have an isomorphism 
\[ {\rmR}\Gamma (K_{\e}  ; \psi_h( j_! \F)) \simeq 
{\rmR}\Gamma ( \gamma (  w, \tau  )^{-1}(K_{\e}) ; \F). \]
Then the assertion immediately follows. 
\end{proof}

Now we can rewrite 
Proposition \ref{prop-4} as follows. 

\begin{proposition}\label{new-prop-4}
Assume that perverse sheaf 
$\F=Sol_X(\SM)[N]\in\BDC_{\CC-c}(\CC_{X^{\an}})$ is moderate 
at infinity over a point $w \neq 0$ of $Y=\CC^N$.
Then for $|\tau|\gg0$ we have
\[ 
\chi_w\big(Sol_{\{w\}}(\bfD i_{\{w\}}^\ast\SM^\wedge)\big)= 
\chi\big(\rmR\Gamma_c(X; \SF)\big)-\chi\big(
\rmR\Gamma_c(H_\infty\bs H(w); \psi_h( j_! \F))\big). \] 
\end{proposition}

Finally we obtain the following result. 

\begin{theorem}\label{thm-7}
Assume that $d_{W_i}=N-1$ and perverse sheaf
$\F=Sol_X(\SM)[N]\in\BDC_{\CC-c}(\CC_{X^{\an}})$ is moderate 
at infinity over a neighborhood of a generic point 
$v \in (W_i)_{\rm reg}$ in $Y \setminus \{ 0 \}$. 
Take a normal slice $M$ of $W_i$ 
at $v$ and consider the meromorphic connection 
$\SK(\ast\{v\})$ on it. 
Then the multiplicity 
${\rm mult}_{T^\ast_{W_i}Y}\SM^\wedge \geq0$
of the Fourier transform $\SM^\wedge$ along $T^\ast_{W_i}Y$
is given by
\begin{align*}
{\rm mult}_{T^\ast_{W_i}Y}\SM^\wedge
&=
{\rm mult}_{T^\ast_{V_i}X}\SN
+ {\rm irr} ( \SK(\ast\{v\}) ) 
\\
(&\geq {\rm mult}_{T^\ast_{V_i}X}\SN =\mi > 0).
\end{align*}
In particular, the conormal bundle $T^\ast_{W_i}Y$ is
contained in the characteristic variety $\ch(\SM^\wedge)$
of $\SM^\wedge$ and we have $W_i\subset D= Y\bs\Omega$.
\end{theorem}

\begin{proof}
By our assumption there exists a point 
$v_0 \in \Omega$ over which 
$\F\in\BDC_{\CC-c}(\CC_{X^{\an}})$ is moderate 
at infinity. Then by 
Theorem \ref{new-thm-7} and Lemma \ref{l-conic} 
we obtain 
\begin{align*}
{\rm mult}_{T^\ast_{W_i}Y}\SM^\wedge
= & \chi \Big(\rmR\Gamma_c\big(H_\infty\bs H(v); \psi_h(
 j_! \F)\big)\Big) 
-\chi\Big(\rmR\Gamma_c\big(H_\infty\bs H(v_0); \psi_h( j_! 
\F)\big)\Big)\\
& + {\rm irr} ( \SK(\ast\{v\}) ).
\end{align*}
Note that for the conification $\SG$ of $\SF$ we have 
$\SG_0 = \psi_h( j_! \SF)=\psi_h( j_! \SG)$. 
Then by replacing $\SM^\wedge$ with
the regular holonomic $\SD_Y$-module $\SN^\wedge$
we obtain also
\[{\rm mult}_{T^\ast_{W_i}Y}\SN^\wedge
=\chi\Big(\rmR\Gamma_c\big(H_\infty\bs H(v); \psi_h(
 j_! \F)\big)\Big)
-\chi\Big(\rmR\Gamma_c\big(H_\infty\bs H(v_0); \psi_h( j_! 
\F)\big)\Big).\]
From this the assertion immediately follows.
\end{proof}

\begin{remark}
We define a $\ZZ$-valued function $\phi : \PP (Y) 
\rightarrow \ZZ$ on $\PP (Y) \simeq \PP^{N-1}$ by
\[ \phi ([w])= 
\chi\Big(\rmR\Gamma \big( H(w); \G_0\big)\Big) 
\qquad ([w] \in \PP (Y) ). \]
This is the topological Radon transform of the 
constructible function $\chi ( \G_0 )$ 
on $\PP (X)= H_{\infty} \simeq \PP^{N-1}$ 
studied by many mathematicians. 
Since the characteristic cycle of $\chi ( \psi_h( j_! \F) )$ 
is a sum of the conormal bundles of 
$\PP (V_i) \subset \PP (X)$ and the 
zero-sections $\PP (X)$, by 
Ernstr{\"o}m \cite[Corollary 3.3]{Ern94} (see also 
Matsui-Takeuchi \cite{MT07} for a new 
proof to it and a generalization to the 
real case) the function 
$\phi$ is constant on $(W_i)_{\rm reg} 
\setminus ( \cup_{j \not= i} W_j) \subset 
(W_i)_{\rm reg}$. This fact would be 
very useful to apply our Theorem \ref{thm-7}. 
\end{remark}

\begin{example}\label{nexam-1} 
For the smooth hypersurface 
$Z= \{ z \in X= \CC^N \ | \ 
z_1^2+ \cdots + z_N^2=1 \} \subset X$ consider 
the perverse sheaf 
$\SF= \CC_Z [N-1] 
\in\BDC_{\CC-c}(\CC_{X^\an})$ on $X= \CC^N$. 
Let $\SM \in\Modrh(\SD_X)$ be 
the regular holonomic $\SD_X$-module 
such that $Sol_X(\SM)[N] = \SF$. 
Then $\SM$ is not monodromic and we have 
\begin{align*}
& \ch\SM  =  T^\ast_ZX
\\
&= \{ (z, ( \zeta z_1, \ldots,  \zeta z_N))
\in T^*X \simeq X \times Y \ | \ 
\zeta \in \CC, z_1^2+ \cdots + z_N^2=1 \}. 
\end{align*}
It follows that for the projection 
$q: T^*X \simeq X \times Y \rightarrow Y$ 
and a point $w \in Y \setminus \{ 0 \} 
= \CC^N  \setminus \{ 0 \}$ we have 
\[ w \in q ( T^\ast_ZX ) \qquad \Longleftrightarrow \qquad 
w_1^2+ \cdots + w_N^2 \not= 0. \]
Moreover if $w \in q ( T^\ast_ZX ) \setminus \{ 0 \}$ 
then the set $q^{-1}(w) \cap T^\ast_ZX$ 
is explicitly calculated as 
\begin{equation*}
q^{-1}(w) \cap T^\ast_ZX =
\left\{\Bigl( \frac{w}{ \zeta} , w \Big)
\in T^*X \simeq X \times Y \ \left|\ 
\zeta^2= w_1^2+ \cdots + w_N^2 
\right.\right\}.
\end{equation*}
This shows that the open subset $\Omega \subset 
Y= \CC^N$ is given by 
\[ \Omega = \{ w \in Y= \CC^N \ | \ 
w_1^2+ \cdots + w_N^2 \not= 0 \} \]
and the morphism $q^{-1} ( \Omega ) \cap 
\ch\SM \rightarrow \Omega$ induced by $q$ 
is a covering of degree $2$. 
By Corollary \ref{cor-3} 
the restriction of the Fourier transform 
$\SM^\wedge\in\Mod_{\rm hol}(\SD_Y)$ of $\SM$ 
to $\Omega\subset Y$ is an integrable connection 
of rank $2$. We can also easily see that 
$\SM^\wedge$ has some irregularities at 
infinity by using Theorem \ref{thm-3}. 
Set $D= Y \setminus \Omega = \{ w \in Y= \CC^N \ | \ 
w_1^2+ \cdots + w_N^2=0 \}$ and take its normal 
slice 
\[  M = \{ (1+u, \sqrt{-1}, 0, \ldots, 0)
 \in Y= \CC^N \ | \ u \in \CC \} \subset Y \]
at the point $(1, \sqrt{-1}, 0, \ldots, 0) 
\in D \setminus \{ 0 \}$. Then for 
a point $w= (1+u, \sqrt{-1}, 0, \ldots, 0) 
\in M \setminus D= M \cap \Omega$ $(u \not=0)$ 
we have 
\begin{equation*}
q^{-1}(w) \cap T^\ast_ZX =
\left\{\Bigl( \frac{w}{ \zeta} , w \Big)
\in T^*X \simeq X \times Y \ \left|\ 
\zeta = \pm \sqrt{u(2+u)} 
\right.\right\}.
\end{equation*}
Since the Puiseux series 
\[ \varphi_{\pm} (u)= 
\Big\langle \frac{\pm w}{\sqrt{u(2+u)}}, w \Big\rangle 
= \pm \sqrt{u(2+u)} \]
of $u$ start from the positive degree $\frac{1}{2}$, 
by Theorem \ref{thm-nn7} the meromorphic 
connection obtained by restricting 
$\SM^\wedge$ to the normal slice $M \subset Y$ 
is regular along the point 
$(1, \sqrt{-1}, 0, \ldots, 0) \in M$. 
On the other hand, by Example \ref{exm-near} 
the perverse sheaf $\SF= \CC_Z [N-1] 
\in\BDC_{\CC-c}(\CC_{X^\an})$ is moderate 
at infinity over any point 
$w \not= 0$ of $Y \setminus \{ 0 \}$. 
Set $V= \{ z \in X= \CC^N \ | \ 
z_1^2+ \cdots + z_N^2=0 \} \subset X$ 
and $W=D= \{ w \in Y= \CC^N \ | \ 
w_1^2+ \cdots + w_N^2=0 \} \subset Y$. 
Then by the natural identification 
$T^*X \simeq T^*Y$ we have $T^*_VX = T^*_WY$. 
Take a conification 
$\SG \in\BDC_{\CC-c}(\CC_{X^\an})$ 
of $\SF$ such that 
$\SG |_{X \setminus \{ 0 \}} \simeq 
\CC_{V \setminus \{ 0 \}} [N-1]$. 
and let $\SN \in\Modrh(\SD_X)$ be 
the (monodromic) regular holonomic $\SD_X$-module 
such that $Sol_X(\SN)[N] = \SG$. 
Then by Theorem \ref{thm-7} we obtain 
\[{\rm multi}_{T^\ast_{W}Y} \SM^\wedge= 
{\rm multi}_{T^\ast_{V}X} \SN +0 =1. \]
\end{example}

\begin{example}\label{nexam-2}
Let us consider the case where 
$\SM^\wedge$ is a confluent 
$A$-hypergeometric system on $Y=\CC^3$. 
Define a subset $A$ of the lattice $\ZZ^2
\subset \RR^2$ by 
\begin{equation*}
A= \Big\{ 
a(1)= \left( \begin{array}{c}
      2 \\
      -1 
    \end{array}  \right), \ 
a(2)= \left( \begin{array}{c}
      1 \\
      1 
    \end{array}  \right), \ 
a(3)= \left( \begin{array}{c}
      -2 \\
      0 
    \end{array}  \right)  
\Big\} \subset \ZZ^2.  
\end{equation*}
Then by the condition $\sum_{i=1}^3 \ZZ a(i) = \ZZ^2$ 
the morphism 
\[i_T : (\CC^\ast)^2\xhookrightarrow{\ \ \ }X=\CC^3,\
s=(s_1, s_2)\longmapsto
(s^{a(1)}, s^{a(2)}, s^{a(3)} )\]
associated to it 
of the 2-dimensional torus $T=( \CC^\ast )^2$ 
is a closed embedding. 
For $c=(c_1, c_2) \in \CC^2$ 
set $\SL=\SO_T s_1^{c_1-1}s_2^{c_2-1} \in\Modrh(\SD_T)$
and $\SM={\bfD} i_{T\ast}\SL \in\Modrh(\SD_X)$. 
In this case, $Z= i_T(T) \subset X= \CC^3$ 
is a closed hypersurface and explicitly given by 
\[ Z = \{ z \in X= \CC^N \ | \ 
z_1^2 z_2^2 z_3^3 =1 \}. \]
Hence we have 
\begin{align*}
& \ch\SM  =  T^\ast_ZX
\\
&= \{ (z, ( 2 \zeta z_1 z_2^2 z_3^3, 
2 \zeta z_1^2 z_2 z_3^3, 
 3 \zeta z_1^2 z_2^2 z_3^2))
\in T^*X \simeq X \times Y \ | \ 
\zeta \in \CC, z_1^2 z_2^2 z_3^3 =1 \}
\end{align*}
and ${\rm mult}_{ T^\ast_{Z}X}\SM =1$. 
It follows that for the projection 
$q: T^*X \simeq X \times Y \rightarrow Y$ 
and a point $w \in Y \setminus \{ 0 \} 
= \CC^N  \setminus \{ 0 \}$ we have 
\[ w \in q ( T^\ast_ZX ) \qquad \Longleftrightarrow \qquad 
w_1w_2w_3 \not= 0. \]
Moreover if $w \in q ( T^\ast_ZX ) \setminus \{ 0 \}$ 
then the set $q^{-1}(w) \cap T^\ast_ZX$ 
is explicitly calculated as 
\[ q^{-1}(w) \cap T^\ast_ZX 
= \Big\{ \big( ( \frac{2 \zeta}{w_1}, \frac{2 \zeta}{w_2}, 
\frac{3 \zeta}{w_3}) , w \big)
\in T^*X \simeq X \times Y \ | \ 
\zeta^7 = \frac{w_1^2 w_2^2 w_3^3}{4 \cdot 4 \cdot 27} 
\Big\}. \]
This shows that the open subset $\Omega \subset 
Y= \CC^N$ is given by 
\[ \Omega = \{ w \in Y= \CC^N \ | \ 
w_1w_2w_3 \not= 0 \} \]
and the morphism $q^{-1} ( \Omega ) \cap 
\ch\SM \rightarrow \Omega$ induced by $q$ 
is a covering of degree $7$. 
By Corollary \ref{cor-3} 
the restriction of the Fourier transform 
$\SM^\wedge\in\Mod_{\rm hol}(\SD_Y)$ of $\SM$ 
to $\Omega\subset Y$ is an integrable connection 
of rank $7$. This coincides with Adolphson's result 
in \cite{Ado94}. Indeed, the normalized volume 
of the convex hull $\Delta \subset \RR^2$ of 
$\{ 0 \} \cup A \subset \RR^2$ is equal to 
$7$. We see also that 
$\SM^\wedge$ has some irregularities at 
infinity by Theorem \ref{thm-3}. 
Set $D= Y \setminus \Omega = \{ w \in Y= \CC^N \ | \ 
w_1w_2w_3 = 0 \}$ and take its normal 
slice 
\[  M = \{ (1, 1, u)
 \in Y= \CC^N \ | \ u \in \CC \} \subset Y \]
at the point $(1, 1, 0) 
\in D \setminus \{ 0 \}$. Set 
\[ \Big\{ \zeta \in \CC \ | \ 
\zeta^7= \frac{1}{4 \cdot 4 \cdot 27} \Big\} 
= \{ \zeta_1, \zeta_2, \ldots, \zeta_7 \}. \]
Then for  
a point $w= (1, 1, u) 
\in M \setminus D= M \cap \Omega$ $(u \not=0)$ 
we have 
\[ q^{-1}(w) \cap T^\ast_ZX 
= \Big\{ \big( ( 
\frac{2 \zeta_i}{w_1} u^{\frac{3}{7}}, 
\frac{2 \zeta_i}{w_2} u^{\frac{3}{7}}, 
\frac{3 \zeta_i}{w_3} u^{\frac{3}{7}}) , w \big)
\in T^*X \simeq X \times Y \ | \ 
1 \leq i \leq 7 \Big\}. \]
Since the Puiseux series 
\[ \varphi_{i} (u)= 
7 \zeta_i u^{\frac{3}{7}} 
\qquad ( 1 \leq i \leq 7 ) \]
of $u$ start from the positive degree $\frac{3}{7}$, 
by Theorem \ref{thm-nn7} the meromorphic 
connection obtained by restricting 
$\SM^\wedge$ to the normal slice $M \subset Y$ 
is regular along the point 
$(1, 1, 0) \in M$. 
On the other hand, we can easily see that 
the perverse sheaf $\SF= Sol_X(\SM) [N] 
\in\BDC_{\CC-c}(\CC_{X^\an})$ is moderate 
at infinity over any point 
$w \not= 0$ of $Y \setminus \{ 0 \}$. 
Set $V_1= \{ z_2=z_3=0 \}, V_2= \{ z_1=z_3=0 \}, 
V_3= \{ z_1=z_2=0 \} \subset X$ and 
$W_1= \{ w_1=0 \}, W_2= \{ w_2=0 \}, 
W_3= \{ w_3=0 \} \subset Y$. 
Then by the natural identification 
$T^*X \simeq T^*Y$ for any $1 \leq i \leq 3$ 
we have $T^*_{V_i}X = T^*_{W_i}Y$. 
Take a conification 
$\SG \in\BDC_{\CC-c}(\CC_{X^\an})$ 
of $\SF$ such that 
$\SG |_{X \setminus \{ 0 \}} \simeq 
\SF |_{X \setminus \{ 0 \}}$ 
and let $\SN \in\Modrh(\SD_X)$ be 
the (monodromic) regular holonomic $\SD_X$-module 
such that $Sol_X(\SN)[N] = \SG$. 
Then for any singular point $z \not= 0$ 
of the normal crossing divisor 
$\{ z_1z_2z_3=0 \} \subset X$ we have 
$\chi_z ( \SG )=0$. This follows from 
the well-known fact that the 
Euler characteristic of the Milnor 
fiber of the function $z_1^2z_2^2z_3^3$ 
at such a point is equal to zero. 
By Kashiwara's local index theorem 
for holonomic $\SD$-modules 
in \cite[Corollary 6.3.4]{Kas83-2}, 
we thus obtain  
${\rm multi}_{T^\ast_{V_1}X} \SN =5$, 
${\rm multi}_{T^\ast_{V_2}X} \SN =5$ 
and ${\rm multi}_{T^\ast_{V_3}X} \SN =4$. 
Moreover by Theorem \ref{thm-7} 
for any $1 \leq i \leq 3$ we have 
\[{\rm multi}_{T^\ast_{W_i}Y} \SM^\wedge= 
{\rm multi}_{T^\ast_{V_i}X} \SN +0 =
{\rm multi}_{T^\ast_{V_i}X} \SN. \]
\end{example}

\end{document}